\documentclass[a4paper,twoside,11pt]{amsart}

\usepackage{ae}
\usepackage{amsmath}
\usepackage{amssymb}
\usepackage{diagrams}
\usepackage{url}
\usepackage{setspace}

\newtheorem{thm}{Theorem}[section]
\newtheorem{prop}[thm]{Proposition}
\newtheorem{lem}[thm]{Lemma}
\newtheorem{cor}[thm]{Corollary}

\theoremstyle{definition}
\newtheorem{defi}[thm]{Definition}
\newtheorem{example}[thm]{Example}

\theoremstyle{remark}

\newtheorem{remark}[thm]{Remark}

\newenvironment{packed_enum}{
\begin{enumerate}
  \setlength{\itemsep}{1pt}
  \setlength{\parskip}{0pt}
  \setlength{\parsep}{0pt}
}{\end{enumerate}}

\DeclareMathAlphabet{\mathscr}{OT1}{pzc}{m}{it}

\def\romanenum{\renewcommand{\labelenumi}{\textup{(}\roman{enumi}\textup{)}}}

\def\sh{\textup{sh}}

\def\b{\mathfrak{b}}

\def\Q{\mathbb{Q}}

\def\Spec{\textup{Spec }}
\def\Spf{\textup{Spf }}

\def\ker{\textup{ker }}
\def\im{\textup{im }}

\def\Hom{\textup{Hom}}

\def\A{\mathbb{A}}
\def\O{\mathcal{O}}

\def\Z{\mathbb{Z}}
\def\N{\mathbb{N}}

\def\G{\mathbb{G}}

\def\m{\mathfrak{m}}
\def\a{\mathfrak{a}}

\def\plim{\varprojlim}

\def\P{\mathbb{P}}
\def\p{\mathfrak{p}}

\def\id{\textup{id}}
\def\sep{{\textup{\footnotesize sep}}}
\def\Gal{\textup{Gal}}

\def\Tor{\textup{Tor}}

\def\Frak{\textup{Frac}\,}

\def\an{\textup{an}}
\def\rig{\textup{rig}}
\def\Rig{\textup{Rig}}
\def\FF{\textup{FF}}
\def\TF{\textup{TF}}

\def\uRig{\textup{uRig}}
\def\D{\mathbb{D}}
\def\srig{\textup{urig}}
\def\urig{\textup{urig}}

\def\sAff{\textup{sAff}}

\def\sSp{\textup{sSp\,}}
\def\sp{\textup{sp}}
\def\sI{\mathcal{I}}
\def\sJ{\mathcal{J}}

\def\sT{\mathcal{T}}

\def\B{\mathbb{B}}
\def\fU{\mathfrak{U}}
\def\fV{\mathfrak{V}}
\def\fW{\mathfrak{W}}

\def\fG{\mathfrak{G}}
\def\fH{\mathfrak{H}}
\def\fX{\mathfrak{X}}
\def\fY{\mathfrak{Y}}
\def\fZ{\mathfrak{Z}}
\def\fS{\mathfrak{S}}
\def\fT{\mathfrak{T}}
\newcommand{\ul}[1]{\underline{#1}}

\def\cl{\textup{cl}}

\def\F{\mathbb{F}}

\def\sup{\textup{sup}}

\def\sF{\mathcal{F}}
\def\sG{\mathcal{G}}

\def\fC{\mathfrak{C}}
\def\parent{\textup{par}}
\def\children{\textup{ch}}
\def\leaves{\textup{lv}}
\def\subtree{\textup{subt}}

\def\sr{\textup{ur}}
\def\ur{\textup{ur}}
\def\r{\textup{r}}

\def\alg{\textup{alg}}

\def\Spa{\textup{Spa}}
\def\ad{\textup{ad}}

\def\End{\textup{End}}
\def\sG{\mathscr{G}}

\def\discr{\textup{discr}}
\def\Mod{\textup{Mod}}
\def\uTor{\textup{uTori}}

\def\uqTor{\textup{uqTori}}

\def\phi{\varphi}

\romanenum

\makeatletter
\newcommand\@dotsep{3.5}
\def\@tocline#1#2#3#4#5#6#7{\relax
  \ifnum #1>\c@tocdepth 
  \else
    \par \addpenalty\@secpenalty\addvspace{#2}%
    \begingroup \hyphenpenalty\@M
    \@ifempty{#4}{%
      \@tempdima\csname r@tocindent\number#1\endcsname\relax
    }{%
      \@tempdima#4\relax
    }%
    \parindent\z@ \leftskip#3\relax
    \advance\leftskip\@tempdima\relax
    \rightskip\@pnumwidth plus1em \parfillskip-\@pnumwidth
    #5\leavevmode\hskip-\@tempdima #6\relax
    \leaders\hbox{$\m@th
      \mkern \@dotsep mu\hbox{.}\mkern \@dotsep mu$}\hfill
    \hbox to\@pnumwidth{\@tocpagenum{#7}}\par
    \nobreak
    \endgroup
  \fi}
\makeatother

\begin{document}

\title{\textsc{Néron models of formally finite type}}

\author{Christian Kappen}
\email{christian.kappen@uni-due.de}

\address{
Institut für Experimentelle Mathematik\\
Ellernstrasse 29, 45326 Essen}

\begin{abstract}
We introduce Néron models of formally finite type for uniformly rigid spaces, and we prove that they generalize the notion of formal Néron models for rigid-analytic groups as it was defined by Bosch and Schlöter. Using this compatibility result, we give examples of uniformly rigid groups whose Néron models are not of topologically finite type.
\end{abstract}

\maketitle

\section{Introduction}
Let $K$ be a complete discretely valued field with valuation ring $R$, valuation ideal $\m_R\subseteq R$ and residue field $k$, and let $\pi\in R$ be a uniformizer; we equip $R$ with the the $\m_R$-adic topology. Let $\FF_R$ denote the category of formal $R$-schemes of locally formally finite (ff) type, that is, the category of locally noetherian formal $R$-schemes which are locally isomorphic to formal spectra of quotients of mixed formal power series rings in finitely many variables
\[
R[[S_1,\ldots,S_m]]\langle T_1,\ldots,T_n\rangle\;,
\]
where an ideal of definition is generated by $\m_R$ and by the $S_i$. An object in $\FF_R$ is said to be of locally topologically finite (tf) type if we can take $m=0$ everywhere in this local description; let $\TF_R\subseteq\FF_R$ denote the full subcategory of formal $R$-schemes of locally tf type. In \cite{berthelot_rigcohpreprint}, Berthelot has introduced a generic fiber functor $\rig$ from $\FF_R$ to the category $\Rig_K$ of rigid $K$-spaces; it is characterized by the fact that its restriction to $\TF_R$ coincides with Raynaud's generic fiber functor and that it maps admissible blowups to isomorphisms. 

In \cite{bosch_schloeter}, Bosch and Schlöter developed a theory of formal Néron models of locally tf type for smooth rigid spaces; their work has found many applications, for instance within the framework of motivic integration (cf.\ \cite{nicsebmotintrig}). The Néron model of a smooth rigid space is defined as follows:
\begin{defi}\label{nmrigdefi}
Let $X$ be a smooth rigid $K$-space. A Néron model for $X$ is a pair $(\fX,\phi)$, where $\fX$ is a smooth object in $\TF_R$ and where $\phi:\fX^\rig\rightarrow X$ is a morphism such that the natural map
\[
\Hom_R(\fZ,\fX)\rightarrow\Hom_K(\fZ^\rig,X)\quad\,\quad \ul{\psi}\mapsto\phi\circ\ul{\psi}^\rig
\]
is bijective for all smooth objects $\fZ\in\TF_R$.
\end{defi}
This definition agrees with Definition 1.1 in \cite{bosch_schloeter} if $\phi$ is an open immersion; if $\phi$ is an open immersion on any quasi-compact open part of $\fX^\rig$, we recover the notion of a Néron quasi-model (Definition 6.4 in \cite{bosch_schloeter}). In many interesting cases, the morphism $\phi$ will not be an isomorphism; that is, a Néron model for $X$ needs not be a model of $X$. Néron models of rigid spaces should rather be viewed as smooth models for the loci of unramified points.

Bosch and Schlöter transferred the construction process for algebraic Néron models to the framework of formal and rigid geometry. Let us recall their main existence results for rigid-analytic groups (cf.\ \cite{bosch_schloeter} Theorem 1.2, Criterion 1.4, Theorem 6.2 and Theorem 6.6):

\begin{thm}[Bosch and Schlöter \cite{bosch_schloeter}]\label{bsthm}
Let $G$ be a rigid-analytic $K$-group.
\begin{enumerate}
\item If $\fG\in\TF_R$ is a smooth formal R-group scheme and if $\phi:\fG^\rig\hookrightarrow G$ is a retrocompact open immersion respecting the group structures, then $(\fG,\phi)$ is a formal Néron model of $G$ if and only if the image of $\phi$ contains all unramified points of $G$.
\item If the group of unramified points of $G$ is bounded, then the Néron model $(\fG,\phi)$ of $G$ exists; it is quasi-compact, and $\phi$ is an open immersion.
\item If $G$ is the analytification of a smooth quasi-compact $K$-group scheme $\mathcal{G}$, then the Néron model $\fG$ of $G$ exists if and only if the Néron lft-model (cf.\ \cite{blr} 10.1/1) $\ul{\mathcal{G}}$ of $\mathcal{G}$ exists. In this case, $\fG$ is the formal completion of $\ul{\mathcal{G}}$ along its special fiber, and $\phi$ is an open immersion on every quasi-compact admissible subset of its domain. If moreover the special fiber of $\ul{\mathcal{G}}$ is quasi-compact or if $\mathcal{G}$ is commutative, then $\phi$ is a retrocompact open immersion.
\end{enumerate}
\end{thm}

\begin{example}\label{opendiscexample}
For instance, the Néron model of the open unit disc $(\Spf R[[S]])^\rig$ is $(\Spf R\langle T\rangle,\phi)$, where $\phi:(\Spf R\langle T\rangle)^\rig\rightarrow(\Spf R[[S]])^\rig$ is the open immersion that is given by $S\mapsto \pi T$. Indeed, if we equip $\Spf R[[S]]$ with the multiplicative group structure and if $\Spf R\langle T\rangle$ is given the group structure $(x,y)\mapsto ((1+\pi x)(1+\pi y)-1)/\pi$, then $\phi$ is a homomorphism, and the image of $\phi$, which is the closed subdisc of radius $|\pi|$ around $S=0$, contains all unramified points of the open unit disc.
\end{example}  
 
In his preprint \cite{bisecartpre}, Chai proposed to study Néron models of rigid spaces within the framework of smooth formal $R$-schemes of locally ff type; more precisely speaking, he suggested to replace $\TF_R$ by $\FF_R$ in Definition \ref{nmrigdefi} above, using formal smoothness of topological algebras as a notion of smoothness in $\FF_R$. Chai's interest in Néron models of ff type was stimulated by his investigations of the base change conductor for abelian varieties. 
It is natural to ask whether Chai's definition is compatible with Definition \ref{nmrigdefi} in the sense that a Néron model $(\fX,\phi)$ of a rigid $K$-space $X$ according to Definition \ref{nmrigdefi} automatically satisfies the stronger universal property for smooth test objects $\fZ$ in $\FF_R$. Example \ref{opendiscexample} shows that this is not the case, not even in situations where $\fX$ and $X$ are group objects and where $\phi$ is a homomorphism: the identity on the open unit disc is the generic fiber of the identity on the smooth object $\Spf R[[S]]$ in $\FF_R$, but it does not factor through the image of $\phi$. This observation already indicates that Chai's definition of Néron models of ff type might not be the optimal one.

When trying to write down interesting examples of Néron models of ff type in the sense of Chai, the main obstruction to checking the universal property for smooth test objects in $\FF_R$ originates in the fact that rigid generic fibers of objects in $\FF_R$ admit too many morphisms to rigid spaces: for example, an unbounded function on the open rigid unit disc defines a morphism to the rigid projective line which does not extend to quasi-compact models; its graph does not admit a schematic closure in $\Spf R[[S]]\times\P^1_R$ (cf.\ the introduction of \cite{urigspaces}). Such phenomena do not occur if one restricts $\rig$ to $\TF_R$: given two quasi-paracompact objects in $\TF_R$, any morphism of their rigid generic fibers extends to a morphism of models after admissible blowup, cf.\ \cite{bosch_frgnotes} Theorem 2.8/3; this fact is essential for the arguments of Bosch and Schlöter.

The superabundance of morphisms between rigid generic fibers of objects in $\FF_R$ can be eliminated by means of additional rigidifying "uniform" structure. Formally, this is achieved by means of the category $\uRig_K$ of uniformly rigid $K$-spaces which we introduced in the article \cite{urigspaces}. A uniformly rigid $K$-space may be regarded as a rigid $K$-space together with extra structure that is encoded in a coarser Grothendieck topology and a smaller sheaf of analytic functions; morphisms of uniformly rigid spaces are, locally, defined in terms of \emph{bounded} functions on products of open and closed polydiscs. There is a generic fiber functor $\urig:\FF_R\rightarrow\uRig_K$ which maps an affine object $\Spf A$ to the semi-affinoid $K$-space $\sSp (A\otimes_RK)$, cf.\ \cite{urigspaces} Section 2.4. Raynaud's theory of formal models in rigid geometry allows us to view the category of quasi-paracompact and quasi-separated rigid $K$-spaces as a full subcategory of $\uRig_K$, cf.\ \cite{urigspaces} 2.59 and 2.60: if $\fX$ is any quasi-paracompact model of locally tf type for some quasi-paracompact and quasi-separated rigid $K$-space $X$, then the "Raynaud-type" uniform structure $X^\ur:=\fX^\urig\in\uRig_K$, also called the uniform rigidification of $X$, does not depend on $\fX$, and it is functorial in $X$.   We refer to the introduction of \cite{urigspaces} for a more comprehensive discussion.

To obtain a good notion of Néron models of formally finite type, we propose to replace, in Definition \ref{nmrigdefi}, not just $\TF_R$ by $\FF_R$, as suggested by Chai, but at the same time $\TF_R$ by $\FF_R$, $\Rig_K$ by $\uRig_K$ and $\rig$ by $\urig$:

\begin{defi}\label{nmurigdefi}
Let $X$ be a smooth uniformly rigid $K$-space. A Néron model for $X$ is a pair $(\fX,\phi)$, where $\fX$ is a smooth object in $\FF_R$ and where $\phi:\fX^\urig\rightarrow X$ is a morphism such that the natural map
\[
\Hom_R(\fZ,\fX)\rightarrow\Hom_K(\fZ^\urig,X)\quad\,\quad \ul{\psi}\mapsto\phi\circ\ul{\psi}^\urig
\]
is bijective for all smooth objects $\fZ\in\FF_R$.
\end{defi}

Here a uniformly rigid $K$-space $X$ is called smooth if its underlying rigid $K$-space $X^\r$ (cf.\ \cite{urigspaces} Prop.\ 2.5.5) is smooth, while smoothness in $\FF_R$ is, as before, defined via formal smoothness on the level of topological $R$-algebras; equivalently, it can be characterized in terms of the formal analog of the Jacobian criterion, cf.\ \cite{ajr1} and \cite{ajr2}.

Examples for smooth objects in $\FF_R$ which are Néron models of their uniformly rigid generic fibers are provided by the following result:

\textbf{Propositions \ref{affinenmprop} and \ref{genfibneronprop}:}
\textit{Smooth affine objects and smooth group objects in $\FF_R$ are Néron models of their uniformly rigid generic fibers.}

We are mostly interested in situations where the universal morphism $\phi$ of the Néron model $(\fX,\phi)$ of a uniformly rigid $K$-space $X$ is not surjective. The main result of the present paper is the following; it is an analog of Theorem \ref{bsthm} above, and it concerns Raynaud-type uniform structures on quasi-paracompact rigid groups:

\textbf{Theorem \ref{maincompthm} and its Corollaries:}
\textit{
Let $G$ be a smooth and quasi-paracompact rigid $K$-group.
\begin{enumerate}
\item If $\fG\in\TF_R$ is a smooth quasi-paracompact formal R-group scheme and if $\phi:\fG^\rig\hookrightarrow G$ is a retrocompact open immersion respecting the group structures such that the image of $\phi$ contains all formally unramified points of $G$, then $(\fG,\phi^\ur)$ is a formal Néron model of $G^\ur$.
\item If the group of unramified points of $G$ is bounded, then the Néron model of $G^\ur$ is the uniform rigidification $(\fG,\phi^\ur)$ of the Néron model $(\fG,\phi)$ of $G$ (which exists by Theorem \ref{bsthm}).
\item If $G$ is the analytification of a smooth quasi-compact $K$-group scheme $\mathcal{G}$ whose Néron lft-model $\ul{\mathcal{G}}$ exists and if $\ul{\mathcal{G}}$ is quasi-compact or if $\mathcal{G}$ is commutative, then the conclusion of ($ii$) holds.
\end{enumerate}
}

This result provides the desired compatibility statement saying that in many interesting cases, the Néron model $(\fG,\phi)$ of a rigid $K$-group $G$ satisfies the stron\-ger universal property for smooth test objects $\fZ$ in $\FF_R$ and with respect to morphisms respecting the uniform structure $\fZ^\urig$ on $\fZ^\rig$ and the Raynaud-type uniform structure $G^\ur$ on $G$. By Proposition \ref{neronmodinverseprop}, this universal property is indeed a strengthening of the universal property of $(\fG,\phi)$. The existence results in the above theorem use the existence results of Bosch and Schlöter which we quoted above. Let us note that in statement ($i$), we require $\im\phi$ to contain all points of $G$ with values in possibly infinite formally unramified extensions of $R$, whereas in Theorem \ref{bsthm} ($i$) it suffices to consider finite unramified extensions. Let us moreover remark that $G^\ur$ is a group object (cf.\ the last paragraph of Section 2 in \cite{urigspaces}) and that the morphisms $\phi^\ur$ in statements ($ii$) and ($iii$) above are open immersions (cf.\ the second last paragraph of Section 2 in \cite{urigspaces}). 

\begin{example}
The rigid open unit disc $\D_K=(\Spf R[[S]])^\rig$ over $K$ carries two canonical uniform structures: the structure $\overline{\D}_K:=(\Spf R[[S]])^\urig$, for which it is semi-affinoid, and the Raynaud-type uniform structure $\D_K^\ur:=(\Spf R[[S]])^{\rig,\ur}$ for which it is not quasi-compact.  By Proposition \ref{affinenmprop}, the Néron model of $\overline{\D}_K$ is given by $\Spf R[[S]]$, while Theorem \ref{maincompthm} and Example \ref{opendiscexample} show that the Néron model of $\D_K^\ur$ is given by $(R\langle T\rangle,\phi^\ur)$, where $\phi$ is the morphism that pulls $S$ back to $\pi T$. It is a priori clear that there exists no surjective morphism $\psi:\overline{\D}_K\rightarrow\D_K^\ur$: indeed, since $\overline{\D}_K$ is quasi-compact and since $\D_K^\ur$ is an admissible infinite ascending union of closed subdiscs of radii $<1$, every morphism $\psi$ from $\overline{\D}_K$ to $\D_K^\ur$ must factor through a closed subdisc of radius $<1$. The fact that $(\Spf R\langle T\rangle,\phi^\ur)$ is the Néron model of $\D_K^\ur$ shows that every such $\psi$ actually factors through the closed subdisc of radius $|\pi|$ around the origin. 
\end{example}

Let us give some indications on the proof of Theorem \ref{maincompthm}. A main difficulty we have to face is that smooth formal $R$-schemes of ff type have, in general, few formally unramified points. For example, if $R'/R$ is any formally unramified flat extension of complete discrete valuation rings, then every morphism $\Spf R'\rightarrow\Spf R[[S]]$ has the property that the pullback of $S$ is $\pi$-divisible; that is, every such morphism factors generically through the closed disc of radius $|\pi|$ around the origin. In fact, if the residue field of $R$ is not perfect, there even exist nonempty smooth affine objects $\fZ$ in $\FF_R$ such that $\fZ(R')=\emptyset$ for all formally unramified local extensions of discrete valuation rings $R'/R$; an example for such a $\fZ$ is obtained by completing the affine line over $R$ along an inseparable closed point of its special fiber. However, if we we pass from affine smooth formal $R$-schemes of ff type to their $R$-envelopes (cf.\ the next paragraph), then new formally unramified points appear; the complete discrete valuation rings where these points take values are in general not finite over $R$, which is the reason why we need to take formally unramified extensions of $R$ into account.

If $\fZ=\Spf A$ is an affine smooth formal $R$-scheme of ff type, its $R$-envelope $\fZ^\pi$ is defined to be $\Spf A^\pi$, where $A^\pi$ is the ring $A$ equipped with its $\m_R$-adic topology. The formal $R$-scheme $\fZ^\pi$ is affine, noetherian  and $\m_R$-adic, but in general it is not of ff type over $R$. The formation of $R$-envelopes commutes, in general, neither with localization nor with fibered products, and fibered products of envelopes need not be locally noetherian. Using algebraization techniques modulo powers of $\m_R$, we show (cf.\ Section \ref{envsection}) that morphisms from $\fZ^\urig$ to uniform rigidifications $X^\ur$ of separated and quasi-paracompact objects $X\in\Rig_K$ extend to formally unramified boundary points of the generic fiber $\fZ^\pi_K$ of $\fZ^\pi$; the latter is an adic space in the sense of Huber, and it may be viewed as a compactification of $\fZ^\rig$. In the proof of Theorem \ref{maincompthm}, we use these formally unramified boundary points together with the fact that in the cases of interest, the Néron model $\fG$ of the rigid group $G$ under consideration commutes with formally unramified base change (cf.\ \cite{blr} 10.1/3 and \cite{wegel} Theorem 4). Let us note that $R$-envelopes already appeared in \cite{strauch_deformation} and in \cite{huber_finiteness_ii} under the name of quasi-compactifications.

\begin{example}
Let $R[[S]]_{(\pi)}^\wedge$ denote the $\pi$-adic completion of $R[[S]]_{(\pi)}$. Then $R[[S]]_{(\pi)}^\wedge$ is a complete discrete valuation ring which is flat and formally unramified over $R$, and the natural morphism 
\[
\Spf (R[[S]]_{(\pi)}^\wedge)\rightarrow \Spf(R[[S]]^\pi)
\]
preserving $S$ has the property that the pullback of $S$ is not $\pi$-divisible. On the other hand, the pullback of $S'$ under any morphism 
\[
\sSp (R[[S]]_{(\pi)}^\wedge[\pi^{-1}])\rightarrow(\Spf R[[S']])^{\rig,\ur}
\]
is $\pi$-divisible.
\end{example}

In order to rigorously implement the strategy of proof outlined above, we first have to establish a number of preliminary results in analytic arithmetic geometry which may be of independent interest. For instance, we prove a flat base change theorem for higher proper direct images in locally noetherian formal geometry (cf.\ Section \ref{higherimbcsec}), we study schematic images and schematically dominant morphisms of formal schemes (cf.\ Sections \ref{schemimsec} and \ref{schemdomsec}), and we prove an ff type local analog of the Hartogs continuation theorem (cf.\ Section \ref{ancontsec}). For the detailed structure of the present paper, we refer to the table of contents below.

Our main Theorem \ref{maincompthm} produces an ample collection of examples of Néron models for uniformly rigid spaces. Starting out from these examples, we can construct new uniformly rigid spaces $X$ together with their Néron models $(\fX,\phi)$  by means of completion and descent techniques (cf. Section \ref{constrsec}); the resulting Néron models $\fX$ will in general not be of locally tf type, and their universal morphisms $\phi$ will in general not be surjective. As an interesting class of examples, we introduce uniformly rigid tori (cf.\ Section \ref{urigtorisec}) which, if the residue field $k$ of $R$ is algebraically closed, provide a link between algebraic $K$-tori and abelian $K$-varieties with potentially ordinary reduction.

We expect Néron models of uniformly rigid spaces to have interesting applications in arithmetic geometry. Chai and the author are currently working on further developing the methods of \cite{bisecartpre} within the framework of uniformly rigid spaces, in order to prove a formula for the base change conductor of abelian varieties with potentially ordinary reduction.





The present article contains parts of the author's doctoral thesis. He would like to express his gratitude to his thesis advisor Siegfried Bosch. 

\begin{spacing}{0.01}
\tableofcontents
\end{spacing}


\section{Formal geomtry}

In this section, we establish results in locally noetherian formal geometry (cf.\ \cite{egain} \S 10) which we will need later.

\subsection{Faithful flatness}
We begin by gathering some standard facts on flatness and faithful flatness in locally noetherian formal geometry. Let us recall that a morphism $\phi\colon\fY\rightarrow\fX$ of ringed spaces is called flat at a point $y$ of $\fY$ if the induced homomorphism of stalks is flat. It is called flat if it is flat in every point of $\fY$, and it is called faithfully flat\index{morphism!flat}\index{morphism!faithfully flat} if it is flat and surjective.

From now on, let $\fX$ and $\fY$ denote locally noetherian formal schemes. If $\fY$ and $\fX$ are affine, then $\phi$ is flat if and only if the underlying homomorphism $\phi^*$ of rings of global sections is flat; this is easily seen by considering completed stalks and by using the fact that flatness of ring homomorphisms can be checked after faithfully flat base change, cf.\ \cite{ajl} 7.1.1. The corresponding statement for faithful flatness is not true unless $\phi$ is adic. For example, if $\fX=\Spf A$ is affine and if $\phi$ is the completion of $\fX$ along an ideal $I\subseteq A$ such that $A$ is $I$-adically complete and such that $I$ strictly contains the biggest ideal of definition of $A$, then $\phi^*$ is an isomorphism of rings of global sections, but $\phi$ is not surjective. 

We refer to \cite{yekutieli_rescompl} Def.\ 1.14 for the notion of a morphism of locally formally finite (ff) type between locally noetherian formal schemes. Let us consider a cartesian diagram
\[
\begin{diagram}
\fY\times_\fX\fX'&\rTo^{\phi'}&\fX'\\
\dTo<{\psi'}&&\dTo>\psi\\
\fY&\rTo^\phi&\fX
\end{diagram}
\]
where $\fX$, $\fX'$ and $\fY$ are locally noetherian formal schemes. If $\phi$ or $\psi$ is of locally ff type, then the fibered product $\fY\times_\fX\fX'$ is locally noetherian, and $\phi'$ respectively $\psi'$ is of locally ff type, cf.\ \cite{yekutieli_rescompl} Thm.\ 1.22. If $\fX_0$, $\fX_0'$ and $\fY_0$ are subschemes of definition for $\fX$, $\fX'$ and $\fY$ respectively, then by \cite{egain} 10.7, the fibered product $\fY_0\times_{\fX_0}\fX_0'$ is a subscheme of definition for $\fY\times_\fX\fX'$. Let us remark that subschemes of definition of locally noetherian formal schemes are quasi-separated. In the following, we assume that $\fY\times_\fX\fX'$ is locally noetherian.

\begin{lem}\label{flatnessbasechangeprop}
If $\phi$ is adic or if $\psi$ is of locally ff type, then the properties of being flat or faithfully flat propagate from $\phi$ to $\phi'$.
\end{lem}

\begin{proof}
If $\phi$ is flat, then flatness of $\phi'$ is shown in \cite{ajl} Prop.\ 7.1 (b). By \cite{egain} 3.6.1 (ii), surjectivity of morphisms of schemes is stable under base change; hence we obtain the surjectivity part of the statement on faithful flatness after passing to subschemes of definition. 
\end{proof}

\begin{lem}\label{isodesclem}
If $\phi$ is of tf type and separated and if $\psi$ is faithfully flat and of ff type, then the property of being an isomorphism descends from $\phi'$ to $\phi$.
\end{lem}

\begin{proof}
We may work locally on $\fX$ and thereby assume that $\fX$ is quasi-compact. Let us write $|\cdot|$ to denote the functor sending a formal scheme to its underlying topological space. The bijective map $|\phi'|$ admits the factorization
\[
|\phi'|\,:\,|\fY\times_\fX\fX'|\rightarrow|\fY|\times_{|\fX|}|\fX'|\rightarrow|\fX'|\;,
\]
where the first map is the natural surjection and where the second map is the projection to the second factor. Since $|\phi'|$ is injective, the first map is also injective and, hence, bijective; since $|\phi'|$ is bijective, it follows that the second map is bijective.
 Since $|\psi|$ is surjective and since bijectivity of a map of sets can be checked after surjective base change, it follows that $|\phi|$ is bijective. Let $\fX_0$ be a subscheme of definition for $\fX$, and let $\fY_0$ be the $\phi$-pullback of $\fX_0$; then since $\phi$ is of tf type, $\fY_0$ is a subscheme of definition for $\fY$, and the restriction $\phi_0:\fY_0\rightarrow\fX_0$ of $\phi$ to $\fY_0$ is of finite type. Since $|\phi|=|\phi_0|$, it follows that $\phi_0$ is quasi-finite. Since $\phi$ is separated, $\phi_0$ is separated as well, and by Zariski's Main Theorem \cite{egaiv} 18.12.13 it follows that $\phi_0$ is finite. By \cite{egaiii} 4.8.1, we see that $\phi$ is finite. To prove that $\phi$ is an isomorphism, we may thus assume that $\fX$ and $\fY$ are affine and that $\phi$ is associated to a finite ring homomorphism. In the finite case, complete tensor products agree with ordinary tensor products, so we conclude by \cite{egaiii} 4.8.8 and by means of descent theory for schemes that $\phi$ is indeed an isomorphism.
\end{proof}
\subsection{Higher direct images and flat base change}\label{higherimbcsec}
If $Y$, $X$ and $X'$ are schemes, if $f:Y\rightarrow X$ is a separated morphism of finite type, if $g:X'\rightarrow X$ is a flat morphism and if $f':Y':=Y\times_XX'\rightarrow X'$ and $g':Y'\rightarrow X'$ denote the projections in the resulting cartesian diagram, then by \cite{egaiii} 1.4.15 for any quasi-coherent $\O_Y$-module $\sF$  with $g'$-pullback $\sF':=(g')^*\sF$, the natural morphism
\[
g^*(R^if_*\sF)\rightarrow R^if'_*\sF'
\]
is an isomorphism. In this section we prove the corresponding statement for locally noetherian formal schemes, where we assume that $f$ is proper and that $\sF$ is coherent. The proof which works in the algebraic setting does not carry over to the formal situation because of the appearance of complete tensor products in \v{C}ech complexes and because inverse limits in general do not preserve exactness. Our main input is \cite{egaiii} 3.4.4 which, for an affine base $\fX$, gives precise information on the relation of the cohomology groups $H^q(\fY,\sG)$ to the cohomology groups $H^q(\fY,\sG_n)$ of the reduction $\sG_n$ of $\sG$ modulo some $(n+1)$-st power of an ideal of definition of $\fX$.

We begin by collecting some elementary facts on projective systems.

Let $A$ be a ring, let $(M_\alpha)_{\alpha\in\N}$ be a system of $A$-modules, and let $\phi_\alpha\colon M_{\alpha+1}\rightarrow M_\alpha$ be a system of $A$-homomorphisms. For natural numbers $\beta\geq \alpha$, we write
\[
\phi_{\alpha\beta}\,\mathrel{\mathop:}=\,\phi_\alpha\circ\cdots\circ\phi_{\beta-1}\colon M_{\beta}\rightarrow M_\alpha\;.
\]
Let us recall from \cite{egaiii} 0.13.1.1 that the projective system $(M_\alpha,\phi_\alpha)_{\alpha\in\N}$ is said to satisfy the Mittag-Leffler (ML) condition if for each $\alpha\in\N$, there exists some $\beta\geq \alpha$ such that for all $\gamma\geq\beta$,
\[
\phi_{\alpha\gamma}(M_\gamma)=\phi_{\alpha\beta}(M_\beta)\;.
\]
We say that the system $(M_\alpha,\phi_\alpha)_{\alpha\in\N}$ is Artin-Rees (AR) null if there exists some $\beta\in\N$ such that $\phi_{\alpha,\alpha+\beta}=0$ for all $\alpha\in\N$. This property is preserved under any base change $\cdot\otimes_AA'$. Moreover, if $(M_\alpha,\phi_\alpha)_{\alpha\in\N}$ is (AR) null, then $(M_\alpha,\phi_\alpha)_{\alpha\in\N}$ clearly satisfies (ML), and $\varprojlim (M_\alpha,\phi_\alpha)_{\alpha\in\N}=0$.

\begin{lem}\label{ml1lem}
If $(M_\alpha,\phi_\alpha)_{\alpha\in\N}$ satisfies \textup{(ML)} and if $\varprojlim (M_\alpha,\phi_\alpha)_{\alpha\in\N}=0$, then for each $\alpha\in\N$ there exists some $\beta\geq\alpha$ such that $\phi_{\alpha\beta}=0$.
\end{lem}
\begin{proof}
For each $\alpha\in\N$, we set
\[
M'_\alpha\,\mathrel{\mathop:}=\,\bigcap_{\beta\geq\alpha}\phi_{\alpha\beta}(M_\beta)\;;
\]
the homomorphisms $\phi_\alpha$ restrict to a projective system $(M'_\alpha,\phi'_\alpha)_{\alpha\in\N}$, and 
\[
\varprojlim (M'_\alpha,\phi'_\alpha)_{\alpha\in\N}=\varprojlim (M_\alpha,\phi_\alpha)_{\alpha\in\N}=0\;.
\]
Since $(M_\alpha,\phi_\alpha)_{\alpha\in\N}$ satisfies (ML), the morphisms $\phi'_\alpha$ are surjective, cf.\ \cite{egaiii} 0.13.1.2; hence $M'_\alpha=0$ for all $\alpha\in\N$. Since $(M_\alpha,\phi_\alpha)_{\alpha\in\N}$ satisfies (ML), it follows moreover that for each $\alpha\in\N$ there exists some $\beta\geq\alpha$ in $\N$ such that $M'_\alpha=\phi_{\alpha\beta}(M_\beta)$; thus for this $\beta$ we have $\phi_{\alpha\beta}=0$, as desired.
\end{proof}

We immediately obtain: 
\begin{cor}\label{ml1cor}
Under the assumptions of Lemma \ref{ml1lem},
\[
\plim(M_\alpha\otimes_AA',\phi_\alpha\otimes_AA')_{\alpha\in\N}=0
\]
for any $A$-algebra $A'$.
\end{cor}


Let $\a\subseteq A$ be an ideal. Let us recall from \cite{egaiii} 0.13.7.7 that a filtration $(N_\alpha)_{\alpha\in\N}$ of an $A$-module $N$ is called $\a$-good if $\a N_n\subseteq N_{n+1}$ for all $n\in\N$ and if equality holds for all $n\in\N$ greater than some $n_0\in\N$.

\begin{lem}\label{ml2lem}
Let $A$ be a noetherian adic ring, let $\a$ be an ideal of definition of $A$, let $M$ be a finite $A$-module, let $(M_\alpha,\phi_\alpha)_{\alpha\in\N}$ be a projective system of $A$-modules, and let
\[
(\rho_\alpha\colon M\rightarrow M_\alpha)_{\alpha\in\N}
\]
be a system of $A$-module homomorphisms that is compatible with the transition homomorphisms $\phi_\alpha$. For $\alpha\in\N$, we let $N_\alpha$ denote the kernel of $\rho_\alpha$. Let us assume that the following holds:
\begin{packed_enum}
\item For any $\alpha\in\N$, multiplication by elements in $\a^{\alpha+1}$ is trivial on $M_\alpha$.
\item The filtration $(N_\alpha)_{\alpha\in\N}$ of $M$ is $\a$-good.
\item The system $(M_\alpha,\phi_\alpha)_{\alpha\in \N}$ satisfies the \textup{(ML)} condition.
\item The morphisms $\rho_\alpha$ induce an isomorphism of modules 
\[
M\overset{\sim}{\rightarrow}\varprojlim (M_\alpha,\phi_\alpha)_{\alpha\in\N}\;.
\]
\end{packed_enum}
Then for any noetherian adic flat topological $A$-algebra $A'$, the morphisms $\rho_\alpha\otimes_AA'$ induce an isomorphism of modules
\[
M\otimes_AA'\rightarrow\varprojlim(M_\alpha\otimes_AA',\phi_\alpha\otimes_AA')_{\alpha\in\N}\;.
\]
\end{lem}

\begin{proof}
In the following, we drop transition morphisms in projective systems from the notation when no confusion can result. By ($i$), for $\alpha\in\N$ the $A$-homomorphism $\rho_\alpha$ factors uniquely over an $A$-homomorphism
\[
\tilde{\rho}_\alpha\colon M/\a^{\alpha+1}M\rightarrow M_\alpha\;.
\]
Let $\tilde{N}_\alpha$ and $Q_\alpha$ denote its kernel and its cokernel respectively, and let us consider the resulting projective system of exact sequences
\[
0\rightarrow\tilde{N}_\alpha\rightarrow M/\a^{\alpha+1}M\overset{\tilde{\phi}_\alpha}{\longrightarrow} M_\alpha\rightarrow Q_\alpha\rightarrow 0\;.
\]
By ($ii$), there exists some $\beta\in\N$ such that $\a N_\gamma=N_{\gamma+1}$ for all $\gamma\geq\beta$, which implies that $N_{\alpha+\beta}=\a^\alpha N_{\beta}$ for all $\alpha\in\N$. It follows that all transition maps
\[
\tilde{N}_{\alpha+1+\beta}\rightarrow\tilde{N}_\alpha
\]
are zero, and hence $(\tilde{N}_\alpha)_{\alpha\in\N}$ is (AR) null. The above exact sequences decompose into two projective systems of short exact sequences
\[
0\rightarrow\tilde{N}_\alpha\rightarrow M/\a^{\alpha+1}M\rightarrow Q_\alpha'\rightarrow 0
\]
and
\[
0\rightarrow Q_\alpha'\rightarrow M_\alpha\rightarrow Q_\alpha\rightarrow 0\;,
\]
where $(Q_\alpha')_{\alpha\in\N}$ satisfies (ML) by \cite{egaiii} 0.13.2.1 (i) and where the system $(\tilde{N}_\alpha)_{\alpha\in\N}$ satisfies (ML) since it is (AR) null. Hence, we may pass to the limit without losing exactness; reassembling the resulting two short exact sequences, we obtain an exact sequence
\[
0\rightarrow\varprojlim(\tilde{N}_\alpha)_{\alpha\in\N}\rightarrow \varprojlim (M/\a^{\alpha+1}M)_{\alpha\in\N}\rightarrow \varprojlim (M_\alpha)_{\alpha\in\N}\rightarrow \varprojlim (Q_\alpha)_{\alpha\in\N}\rightarrow 0\;.
\]
Since $M$ is $\a$-adically complete and since 
\[
\varprojlim(\rho_\alpha)_{\alpha\in\N}\colon M\rightarrow\varprojlim (M_\alpha)_{\alpha\in\N}
\]
is an isomorphism by assumption ($iv$), it follows that $\varprojlim (Q_\alpha)_{\alpha\in\N}=0$. Since $(M_\alpha)_{\alpha\in\N}$ satisfies (ML), the same holds for $(Q_\alpha)_{\alpha\in\N}$ by \cite{egaiii} 0.13.2.1 (i); it follows that $(Q_\alpha)_{\alpha\in\N}$ satisfies the conditions of Lemma \ref{ml1lem}. Since $A'$ is flat over $A$, we have exact sequences
\[
0\rightarrow\tilde{N}_\alpha\otimes_AA'\rightarrow M/\a^{\alpha+1}M\otimes_AA'\rightarrow M_\alpha\otimes_AA'\rightarrow Q_\alpha\otimes_AA'\rightarrow 0\;.
\]
Now $(\tilde{N}_\alpha\otimes_AA')_{\alpha\in\N}$ is (AR)-null, and $\varprojlim (Q_\alpha\otimes_AA')_{\alpha\in\N}=0$ by Corollary \ref{ml1cor}. As above, we decompose the above exact sequence into two short exact sequences, pass to the limit without losing exactness and reassemble the result; we obtain an isomorphism of modules
\[
\varprojlim(M/\a^{\alpha+1}M\otimes_AA')_{\alpha\in\N}\overset{\sim}{\rightarrow}\varprojlim(M_\alpha\otimes_AA')_{\alpha\in\N}\;.
\]
Since $M\otimes_AA'$ is finite over the noetherian ring $A'$ and since $A'$ is $\a A'$-adically complete and separated, cf. \cite{egain} 0.7.2.4, we conclude that
\[
M\otimes_AA'\rightarrow\varprojlim(M_\alpha\otimes_AA')_{\alpha\in\N}
\]
is an isomorphism of modules, as desired.
\end{proof}

We can now apply the above results on projective systems with \cite{egaiii} 3.4.4 to obtain the following formal flat base change theorem:

\begin{thm}\label{mainflatbasechangethm}
Let 
\[
\begin{diagram}
\fY'&\rTo^{\phi'}&\fX'\\
\dTo<{\psi'}&&\dTo>\psi\\
\fY&\rTo^\phi&\fX
\end{diagram}
\]
be a cartesian diagram of locally noetherian formal schemes, where $\phi$ is proper and where $\psi$ is flat. Let $\sF$ be a coherent $\O_\fY$-module; then the natural morphism of coherent $\O_{\fX'}$-modules
\[
\psi^*R^q\phi_*\sG\rightarrow R^q\phi'_*(\psi')^*\sG
\]
is an isomorphism.
\end{thm}
\begin{proof}
We may work locally on $\fX'$; hence we may assume that $\fX=\Spf A$ and $\fX'=\Spf A'$ are affine, and by \cite{egaiii} 3.4.2, 3.4.6 it suffices to show that for all $q\in\N_{\geq 0}$, the natural homomorphism of finite $A'$-modules
\[
H^q(\fY,\sF)\otimes_AA'\rightarrow H^q(\fY',(\psi')^*\sF)
\]
is an isomorphism. Let $\a$ be the biggest ideal of definition of $A$, and let a subscript $n$ denote reduction modulo $\a^{n+1}$. By \cite{egaiii} 3.4.4, the natural system of $A$-homomorphisms
\[
H^q(\fY,\sF)\rightarrow H^q(\fY,\sF_n)
\]
satisfies the conditions of Lemma \ref{ml2lem}; by Lemma \ref{ml2lem}, we thus obtain an induced isomorphism
\[
H^q(\fY,\sF)\otimes_AA'\rightarrow \varprojlim(H^q(\fY,\sF_n)\otimes_AA')_{n\in\N}\;.
\]
Now $H^q(\fY,\sF_n)$ is naturally identified with $H^q(\fY_n,\sF_n)$, and $\fY_n$ is a scheme over $\Spec A_n$. By \cite{egaiii} 1.4.15, we have a natural identification
\[
H^q(\fY_n,\sF_n)\otimes_AA'\,\cong\,H^q(\fY_n\otimes_AA',\sF_n\otimes_AA')\;.
\]
Since $\fY'_n$ is obtained from $\fY_n\otimes_AA'$ via formal completion, the Comparison Theorem \cite{egaiii} 4.1.5 provides a natural identification
\[
H^q(\fY_n\otimes_AA',\sF_n\otimes_AA')\cong H^q(\fY'_n,(\psi')^*\sF_n)\;,
\]
so we obtain a natural isomorphism
\[
H^q(\fY,\sF)\otimes_AA'\,\cong\,\varprojlim(H^q(\fY'_n,(\psi')^*\sF_n)_{n\in\N}\quad.
\]
Let $\a'$ be the biggest ideal of definition of $A'$, and let a subscript $m$ denote reduction modulo $(\a')^{m+1}$. We note that $\a'$ induces an ideal of definition of $\fY'_n$ for all $n\in\N$. By \cite{egaiii} 3.4.4, there is a natural isomorphism
\[
H^q(\fY'_n,(\psi')^*\sF_n)\cong\varprojlim (H^q(\fY'_n,((\psi')^*\sF_n)_m))_{m\in\N}\quad,
\]
so we obtain natural isomorphisms
\begin{eqnarray*}
H^q(\fY,\sF)\otimes_AA'&\cong&\varprojlim(H^q(\fY'_n,(\psi')^*\sF_n)_{n\in\N}\\
&\cong&\varprojlim (\varprojlim (H^q(\fY'_n,((\psi')^*\sF_n)_m))_{m\in\N})_{n\in\N}\\
&\cong&\varprojlim (H^q(\fY',((\psi')^*\sF)_m))_{m\in\N}\\
&\cong&H^q(\fY',(\psi')^*\sF)\;,
\end{eqnarray*}
where the last identification is again provided by \cite{egaiii} 3.4.4, and where we have used that $((\psi')^*\sF)_{n,m}=((\psi')^*\sF)_m$ for all $m\geq n$. One verifies that the natural isomorphism thus obtained agrees with the natural homomorphism in the statement of the proposition.
\end{proof}

\subsection{Schematic images of proper morphisms of formal schemes}\label{schemimsec}

\begin{defi}\label{schemimdefi}
Let $\phi\colon\fY\rightarrow\fX$ be a morphism of locally noetherian formal schemes. A closed formal subscheme $\fV\subseteq\fX$ is called a schematic image of $\phi$ if $\phi$ factors through $\fV$ and if $\fV$ is minimal among all closed formal subschemes of $\fX$ with this property.
\end{defi}

\begin{lem}
The schematic image of a morphism of locally noetherian formal schemes is unique, and it always exists.
\end{lem}
\begin{proof}
Let $\phi:\fY\rightarrow\fX$ be a morphism of locally noetherian formal schemes. The closed formal subschemes of $\fX$ correspond to the coherent ideals in $\O_\fX$. Let $M$ denote the set of coherent subsheaves of $\ker\phi^\sharp$; then $M$ is nonempty because it contains the zero ideal, it is partially ordered by inclusion, and it is directed for this partial ordering because if $\sI$ and $\sJ$ are in $M$, then $\sI+\sJ\in M$.  By definition, a schematic image of $\phi$ corresponds to a biggest element of $M$; in particular, it is unique if it exists. To show that $M$ has a biggest element, it suffices to show that the sum $\sum_{\sI\in M}\sI\subseteq\O_\fX$ is coherent. Let $\fU\subseteq\fX$ be any affine formal subscheme of $\fX$; then $(\sum_{\sI\in M}\sI)|_\fU=\sum_{\sI\in M}(\sI|_\fU)$, and it suffices to see that this submodule of $\O_\fU$ is coherent. Let $N$ denote the partially ordered set of coherent ideals of $\O_\fU$, and let $\psi:M\rightarrow N$ be the map sending $\sI$ to $\sI_\fU$. Then $\psi$ is order-preserving, and $\psi(M)$ is directed because $\sI|_U+\sJ|_U=(\sI+\sJ)|_U$. To show that $\sum_{\sI\in M}\sI|_U=\sum_{\sI\in\psi(M)}\sI$ is coherent, it suffices to see that $\psi(M)$ has a biggest or, equivalently, a maximal element. However, the coherent ideals of $\O_\fU$ correspond to the ideals in the noetherian ring of global sections of $\fU$, and every nonempty set of ideals in a noetherian ring has a maximal element.
\end{proof}

If $\fX$ is defined over a complete discrete valuation ring $R$ and if $\fY$ is $R$-flat, then the schematic image of $\phi$ is $R$-flat, because the ideal of $\pi$-torsion of a locally noetherian formal $R$-scheme is coherent.

Let us consider an example where $\ker\phi^\sharp$ is not coherent:

\begin{example}
Let $k$ be any field, and let $\phi\colon\D^1_k\rightarrow\A^1_k$ be the completion of the affine line over $k$ in the origin. Then the schematic image of $\phi$ exists and in fact coincides with $\A^1_k$, while $\ker\phi^\sharp$ is not coherent.
\end{example}
\begin{proof}
Any closed subscheme of the affine scheme $\A^1_k$ is associated to an ideal in the ring of global sections of $\A^1_k$. Since $\phi$ induces an injection of rings of global sections, $\phi$ cannot factor through a proper closed subscheme of $\A^1_k$; hence the schematic image of $\phi$ exists and coincides with $\A^1_k$. If the ideal $\ker\phi^\sharp$ was coherent, it would be trivial since it has no nonzero global section. However, for any nonempty open subscheme $U\subseteq\A^1_k$ not containing the origin, $(\ker\phi^\sharp)(U)=\O_{\A^1_k}(U)$ since $\phi^{-1}(U)=\emptyset$.
\end{proof}

From now on, we will only consider schematic images of morphisms $\phi$ in situations where $\ker\phi^\sharp$ is coherent; then the formation of the schematic image of $\phi$ is local on $\fX$. For example, the latter is the case when $\ker\phi^\sharp$ is trivial or when $\phi_*\O_\fY$ is coherent. If $\phi$ is proper, then $\phi_*\O_\fY$ is coherent by \cite{egaiii} 3.4.2. In this case, the formation of the schematic image commutes with flat base change:

\begin{prop}\label{flatbasechangeprop}
Let $\phi\colon\fY\rightarrow\fX$ and $\psi\colon\fX'\rightarrow\fX$ be morphisms of locally noetherian formal schemes, where $\phi$ is proper and where $\psi$ is flat, and let
\[
\begin{diagram}
\fY'&\rTo^{\phi'}&\fX'\\
\dTo<{\psi'}&&\dTo>\psi\\
\fY&\rTo^\phi&\fX
\end{diagram}
\]
be the induced cartesian diagram. If $\fV\subseteq\fX$ and $\fV'\subseteq\fX'$ are the schematic images of $\phi$ and $\phi'$ respectively, then
\[
\fV'\,=\,\psi^{-1}(\fV)\;,
\]
where $\psi^{-1}(\fV)=\fV\times_\fX\fX'$ denotes the schematic preimage of $\fV$ under $\psi$.
\end{prop}
\begin{proof}
Let $\sI$ denote the coherent $\O_\fX$-ideal defining $\fV$. Then $\sI$ is defined by the natural exact sequence of coherent $\O_\fX$-modules
\[
0\rightarrow\sI\rightarrow\O_\fX\rightarrow\phi_*\O_\fY\quad.
\]
Since $\psi$ is flat, we obtain an induced exact sequence
\[
0\rightarrow\psi^*\sI\rightarrow\O_{\fX'}\rightarrow\psi^*\phi_*\O_\fY\quad,
\]
and $\psi^*\sI$ is the coherent $\O_{\fX'}$-ideal defining $\psi^{-1}(\fV)$. By Theorem \ref{mainflatbasechangethm}, $\psi^*\phi_*\O_\fY$ is naturally identified with $(\phi')_*\O_{\fY'}$, so we see that $\psi^*\sI$ also defines the schematic image $\fV'$ of $\phi'$, as desired.
\end{proof}

\subsection{Schematically dominant morphisms of formal schemes}\label{schemdomsec}

According to \cite{egaiv} 11.10, a morphism $\phi\colon Y\rightarrow X$ of schemes is called schematically dominant if $\phi^\sharp$ is a monomorphism. Let $S$ be any scheme such that $X$ is defined over $S$. By \cite{egaiv} 11.10.1, the following are equivalent: 
\begin{packed_enum}
\item $\phi$ is schematically dominant.
\item For any open subscheme $U\subseteq X$, the restriction $\phi^{-1}(U)\rightarrow U$ of $\phi$ does not factor though a proper closed subscheme of $U$.
\item For any open subscheme $U\subseteq X$ and any $S$-scheme $X'$, two $U$-valued points $x_1,x_2\in X'(U)$ of $X'$ coincide if and only if the $\phi$-induced $\phi^{-1}(U)$-valued points of $X'$ coincide.
\end{packed_enum}
If $\phi\colon Y\rightarrow X$ is a morphism of $S$-schemes, then according to \cite{egaiv} 11.10.8, $\phi$ is called universally schematically dominant over $S$ if for all $S$-schemes $T$, the base change $\phi_T\colon Y_T\rightarrow X_T$ of $\phi$ from $S$ to $T$ is schematically dominant.

We generalize this concept to the setting of morphisms of locally noetherian formal schemes that are adic over a given base. Let $\fS$ be a locally noetherian formal base scheme. 

\begin{defi}\label{univschemdom}\index{morphism!schematically dominant}
A morphism $\phi\colon\fY\rightarrow\fX$ of adic locally noetherian formal $\fS$-schemes is called universally schematically dominant above $\fS$ if for any scheme $T$ and any morphism $T\rightarrow\fS$, the induced morphism of schemes
\[
\phi_T\colon\fY\times_\fS T\rightarrow\fY\times_\fS T
\]
is schematically dominant. 
\end{defi}

Here we do not assume $T$ to be locally noetherian, and both $\fX$ and $\fY$ are not assumed to be of locally tf type over $\fS$. If $\phi$ is an open immersion, we say that $\phi$ identifies $\fY$ with an $\fS$-dense open formal subscheme of $\fX$. In analogy with \cite{egaiv} 11.10.9, universal schematic dominance is, in the flat case, equivalent to schematic dominance in the fibers:

\begin{lem}\label{univschemdomfiberlem}
Let $\phi\colon\fY\rightarrow\fX$ be a morphism of locally noetherian adic $\fS$-schemes, and let us assume that $\fY$ is flat over $\fS$. Then the following are equivalent:
\begin{packed_enum}
\item $\phi$ is universally schematically dominant above $\fS$.
\item For any point $s\in\fS$, the fiber 
\[
\phi_s\colon\fY\times_\fS k(s)\rightarrow\fX\times_\fS k(s)
\]
is a schematically dominant morphism of schemes.
\end{packed_enum}
\end{lem}
\begin{proof}
The implication ($i$)$\Rightarrow$($ii$) is trivial, so let us assume that ($ii$) is satisfied. We show that ($i$) holds. Let $T$ be any $\fS$-scheme; we must show that $\phi_T$ is schematically dominant. Since schematic dominance is local on the target, we may assume that $T$ is affine and, hence, quasi-compact. Then the morphism $T\rightarrow \fS$ factors through an infinitesimal neighborhood $\fS_n$ of the smallest subscheme of definition $\fS_0$ of $\fS$ in $\fS$. It thus suffices to show that $\phi_{\fS_n}$ is universally schematically dominant. By Proposition \ref{flatnessbasechangeprop}, the scheme $\fY\times_\fS\fS_n$ is flat over $\fS_n$, so by \cite{egaiv} 11.10.9, $\phi_{\fS_n}$ is universally schematically dominant if and only if $\phi_{\fS_n,s}$ is schematically dominant for each point $s\in\fS_n$. However, the underlying topological spaces of $\fS$ and $\fS_n$ coincide, and for any physical point $s$ of $\fS$, $\phi_{\fS_n,s}=\phi_s$.
\end{proof}

\begin{cor}\label{schemdensecor}
Let $R$ be a complete discrete valuation ring with residue field $k$. An open immersion $\fU\hookrightarrow\fX$ of flat adic locally noetherian formal $R$-schemes is universally schematically dominant if and only if $\fU_k\subseteq\fX_k$ is schematically dense; in this case we say that $\fU$ is $R$-dense in $\fX$. If in addition $\fX_k$ is reduced, then this condition is equivalent to $\fU_k$ being dense in $\fX_k$.
\end{cor}

\subsection{Smoothness in locally noetherian formal geometry}
The notion of smoothness in locally noetherian formal geometry has been extensively studied in Chapter 0 of \cite{egaiv} and in \cite{ajr1}, \cite{ajr2}. Let us briefly recall the main definitions and results. According to \cite{ajr1} 2.6, a morphism of locally ff type of locally noetherian formal schemes $\phi\colon\fY\rightarrow\fX$ is called smooth\index{morphism!smooth}, unramified or étale if it satisfies the corresponding infinitesimal lifting property, cf.\ \cite{ajr1} 2.1, 2.3. These properties are stable under composition as well as under base change, cf.\ \cite{ajr1} 2.9. Moreover, they are local on the domain, cf.\ \cite{ajr1} 4.1. The morphism $\phi$ is called smooth, unramified or étale in a point $x\in\fX$ if it satisfies the corresponding property in an open neighborhood of $\fX$, cf.\ \cite{ajr1} 4.3. Open immersions are étale, and closed immersions are unramified, cf.\ \cite{ajr1} 2.12. By \cite{ajr1} 4.4, completion morphisms are étale. Smoothness is conveniently characterized in terms of sheaves of continuous relative differentials: Let $\phi\colon\fY\rightarrow\fX$ be of locally ff type, where $\fX$ and $\fY$ are locally noetherian. Then the sheaf of continuous relative differential forms $\Omega^1_{\fY/\fX}$ is coherent, and the morphism $\phi\colon\fY\rightarrow\fX$ is unramified if and only if $\Omega^1_{\fY/\fX}$ is trivial, cf.\ \cite{ajr1} 4.6. If $\phi$ is smooth, then $\Omega^1_{\fY/\fX}$ is locally free by \cite{ajr1} 4.8; its rank is called the relative dimension of $\fY$ over $\fX$. Let $\fS$ be a locally noetherian formal base scheme, and let $\fX$ be smooth of locally ff type over $\fS$. Then $\phi$ is smooth if and only if $\fY$ is smooth over $\fS$ and the natural sequence
\[
0\rightarrow \phi^*\Omega^1_{\fX/\fS}\rightarrow\Omega^1_{\fY/\fS}\rightarrow\Omega^1_{\fY/\fX}\rightarrow 0
\]
is exact and locally split, cf.\ \cite{ajr1} 4.12. One derives the Jacobian criterion for smoothness in locally noetherian formal geometry: let us assume that $\fX$ is smooth of locally ff type over $\fS$ and that $\phi\colon\fY\hookrightarrow\fX$ is a closed immersion. Let $y\in \fY$ be a point, and let us consider local sections of $\fX$ defining $\phi$ near $y$. Then $\fY$ is smooth in $y$ over $\fS$ if and only if the associated Jacobian matrix has an appropriate minor that does not vanish in $y$, cf.\ \cite{ajr1} 4.15 and \cite{ajr2} 5.11, 5.12. By \cite{ajr2} 5.4, smoothness of a morphism $\phi\colon\fY\rightarrow\fX$ of locally ff type in a point $y$ is equivalent to formal smoothness of the corresponding homomorphism of (completed) stalks with respect to the maximal-adic topologies. Moreover, by \cite{ajr2} 5.4, $\phi$ is smooth in $y$ if and only if it is flat in $y$ and if its fiber over the image of $y$ in $\fX$ is smooth in $y$; this fiber is in general not a scheme but a formal scheme, unless $\phi$ is adic. A smooth morphism of locally ff type is locally an étale morphism of ff type to a relative closed formal unit ball, cf.\ \cite{ajr2} 5.9. Moreover, a smooth morphism of locally ff type is locally a completion of a smooth morphism of tf type, cf.\ \cite{ajr2} 7.12.

Let $R$ be a complete discrete valuation ring with residue field $k$. Smooth morphisms of formal $R$-schemes of ff type are regular:

\begin{prop}\label{regprop}\index{morphism!regular}
Let $\phi\colon\fY\rightarrow\fX$ be a smooth morphism of affine formal $R$-schemes of ff type. Then $\phi$ induces a regular homomorphism of rings of global sections.
\end{prop}
\begin{proof}
By \cite{valabrega1} Proposition 7 and \cite{valabrega2} Theorem 9, $R$-algebras of ff type are excellent, and by \cite{egaiv} 7.8.3 ($v$), completion homomorphisms of excellent noetherian rings are regular. If $(\fX_i)_{i\in I}$ is a finite affine open covering of $\fX=\Spf A$ and if $\fX_i=\Spf A_i$, then the induced homomorphisms $A\rightarrow A_i$ are completion homomorphisms, and the induced homomorphism from $A$ to the direct sum of the $A_i$ is faithfully flat. By \cite{egaiv} 6.8.3 ($ii$), we may thus work locally on $\fY$, and by \cite{egaiv} 7.8.3 ($v$) and  \cite{ajr2} 7.12 it follows that we may assume that $\phi$ is adic. By \cite{elkik} Théorème 7, $\fY$ is then obtained from a smooth algebra over the ring of global functions on $\fX$ via formal completion. By \cite{egaiv} 6.8.6, smooth morphisms of locally noetherian schemes are regular. Since algebras of finite type over excellent noetherian rings are excellent (\cite{egaiv} 7.8.3 ($ii$)), the statement now follows from a third application of \cite{egaiv} 7.8.3 ($v$).
\end{proof}

Regarding the notion normality for formal $R$-schemes of locally ff type, we refer to the discussion in \cite{conrad_irred} Section 1.2.

\begin{cor}\label{smoothimpliesnormalcor}
If $\phi\colon\fY\rightarrow\fX$ is a smooth morphism of formal $R$-schemes of locally ff type and if $\fX$ is normal, then $\fY$ is normal as well.
\end{cor}
\begin{proof}
This follows from Proposition \ref{regprop} together with \cite{egaiv} 6.5.4 ($ii$).
\end{proof}

In Section \ref{smoothlocalstructureprop}, we will recurrently complete smooth formal $R$-schemes of locally ff type in closed points in order to reduce to local situations. Completion morphisms being étale, the local formal $R$-schemes of ff type thus obtained are smooth. The structure of a local smooth formal $R$-scheme $\fX$ of ff type is particularly simple: if the residue field of $\fX$ coincides with $k$, then $\fX$ is a formal open unit ball, and in general there always exists a finite extension $R'/R$ of discrete valuation rings such that $\fX\times_RR'$ is a finite disjoint union of  formal open unit balls over $R'$:

\begin{lem}\label{localstructurerationalpointlem}
Let $\fX$ be a local smooth formal $R$-scheme of ff type whose residue field naturally coincides with $k$. Then there exists an $R$-isomorphism 
\[
\fX\cong\Spf R[[T_1,\ldots,T_d]]\;,
\]
where $d$ denotes the relative dimension of $\fX$ over $R$.
\end{lem}
\begin{proof}
Let us write $\fX=\Spf A$, $A_k=A/\pi A$. By \cite{egaiv} 0.19.6.4, there exists a $k$-isomorphism $\phi_k\colon A_k\overset{\sim}{\rightarrow} k[[T_1,\ldots,T_d]]$. Since $A$ is $\pi$-adically complete and formally $R$-smooth for the maximal-adic topologies,  \cite{egaiv} 0.19.7.1.5 shows that $\phi_k$ extends to an $R$-isomorphism $A\overset{\sim}{\rightarrow} R[[T_1,\ldots,T_d]]$.
\end{proof}

\begin{prop}\label{smoothlocalstructureprop}
Let $\fX$ be a local smooth formal $R$-scheme of ff type. There exists a finite extension $R'/R$ of discrete valuation rings with ramification index $1$ such that $\fX\times_RR'$ is $R'$-isomorphic to a disjoint union of finitely many copies of $\Spf R'[[T_1,\ldots,T_d]]$, where $d$ denotes the relative dimension of $\fX$ over $R$.
\end{prop}
\begin{proof}
Let us write $\fX=\Spf A$, let $k_A$ denote the residue field of $A$, and let $k'$ be the normal envelope of $k_A$ over $k$. Then the $k'$-algebra $k_A\otimes_kk'$ modulo its nilradical is a finite direct sum of copies of $k'$. By \cite{liu_ag} Lemma 10.3.32, the extension $k'/k$ lifts to a finite extension of discrete valuation rings $R'/R$ such that $e_{R'/R}=1$. Hence, $\fX\times_RR'$ is a finite disjoint union of local smooth formal $R'$-schemes of ff type whose residue fields all coincide with $k'$. The statement thus follows from Lemma \ref{localstructurerationalpointlem}.
\end{proof}

\begin{remark}
While the finite extension of discrete valuation rings $R'/R$ which is provided by Proposition \ref{smoothlocalstructureprop} has ramification index $1$, it needs not be unramified: the extension of residue fields is allowed to be inseparable.
\end{remark}

To conclude our discussion of smoothness for formal $R$-schemes of ff type, we investigate the structure of formal tubes: let $\fX$ be a smooth affine formal $R$-scheme of ff type, let $f_1,\ldots,f_r$ be a regular sequence in the ring of global functions on $\fX$ (cf.\ \cite{egaiv} 0.15.1.7) such that the closed formal subscheme $\fV$ defined by the $f_i$ is $R$-smooth, and let $n\in\N$ be a natural number. 
\begin{prop}\label{formaltubestructureprop}
The affine open part
\[
\fX\langle\pi^{-(n+1)}f\rangle\,\mathrel{\mathop:}=\,\fX\left\langle\frac{f_1}{\pi^{n+1}},\ldots,\frac{f_r}{\pi^{n+1}}\right\rangle
\]
of the admissible blowup of $\fX$ in the ideal $(\pi^{n+1},f_1,\ldots,f_r)$ where the pullback ideal is generated by $\pi^{n+1}$ is $R$-smooth, and if $\fV$ is connected, then $\fX\langle\pi^{-(n+1)}f\rangle$ is connected as well.
\end{prop}
\begin{proof}
Let us write $f_0\mathrel{\mathop:}=\pi^{n+1}$, and let $A$ denote the ring of global functions on $\fX$. By \cite{eisenbudca} Exercise 17.14, it follows that the blowup
of $\Spec A$ in the ideal $(f_0,\ldots,f_r)$ is covered by the affine open subschemes 
\[
\Spec A[T_0,\ldots,\hat{T}_i,\ldots,T_r]/((f_iT_j-f_j)_{0\leq j\leq r\,,\,j\neq i})\;;
\]
hence
\[
\fX\langle\pi^{-(n+1)}f\rangle\,=\,\Spf A\langle T_1,\ldots,T_r\rangle/((\pi^{n+1}T_j-f_j)_{1\leq j\leq r})\;.
\]
Since $\fX\langle\pi^{-(n+1)}f\rangle$ is $R$-flat, the fiber criterion for smoothness (\cite{ajr2} 5.4) shows that to prove $R$-smoothness of $\fX\langle\pi^{-(n+1)}f\rangle$, it thus suffices to prove that 
\[
\fX\langle\pi^{-(n+1)}f\rangle_k\,=\,\Spf ((A/(f_1,\ldots,f_r)\otimes_Rk)\langle T_1,\ldots,T_r\rangle)
\]
is smooth over $k$. Since $\fV$ is $R$-smooth by assumption, its special fiber 
\[
\fV_k=\Spf (A/(f_1,\ldots,f_r)\otimes_Rk)
\]
is $k$-smooth, so the statement is clear. The statement concerning connectedness is now equally obvious.
\end{proof}

\subsection{Normalizations of formal $R$-schemes of ff type}

Let $R$ be a complete discrete valuation ring with fraction field $K$. We refer to \cite{conrad_irred} Section 1.2 and 2.1 for the concepts of normality and normalization for formal $R$-schemes of ff type and for rigid $K$-spaces. In this section, we show that if $\fX$ is a flat formal $R$-scheme of ff type whose rigid generic fiber $\fX^\rig$ is normal, then the normalization morphism $\tilde{\fX}\rightarrow\fX$ is a finite admissible blowup.

\begin{lem}\label{fingenisoisblowuplem}
Let $\phi$ be a finite morphism of flat formal $R$-schemes of ff type such that $\phi^\rig$ is an isomorphism. Then $\phi$ is a finite admissible blowup.
\end{lem}
\begin{proof}
It suffices to show that the morphism $\phi^\sharp:\O_\fX\rightarrow\phi_*\O_\fY$ of coherent $\O_\fX$-modules is injective, that there exists an element $c$ in $R$ such that the image $c(\phi_*\O_{\fY})$ of the $\O_\fX$-module homomorphism
\[
\phi_*\O_{\fY} \rightarrow \phi_*\O_{\fY}\quad,\quad f\mapsto c\cdot f
\] 
lies in $\phi^\sharp(\O_\fX)$ and that for any such $c$, $\phi$ has the universal property of the admissible blowup of $\fX$ in the $\phi^\sharp$-preimage of $c(\phi_*\O_{\fY})$. Since $\fX$ is quasi-compact, we may work locally on $\fX$; we may thus assume that $\fX$ and, hence, $\fY$ are affine. Let $\phi^*:A\rightarrow B$ denote the continuous homomorphism of topological $R$-algebras of global sections corresponding to $\phi$. Since $\phi^\rig$ is an isomorphism, \cite{dejong_crystalline} 7.1.9 shows that  $\phi^*\otimes_RK:A\otimes_RK\rightarrow B\otimes_RK$ is an isomorphism. We use $\phi^*$ to identify $A\otimes_RK$ with $B\otimes_RK$ and to view $A$ as a subring of $B$. Since $B$ is a finite $A$-module, there exists an element $c\in R$ such that the $A$-submodule $I:=cB$ of $B$ is contained in $A$. Precisely as in the proof of \cite{frg1} 4.5, we see that $\fY$ is the admissible blowup of $\fX$ in $I\O_\fX$, as desired.
\end{proof}

\begin{cor}\label{normisblowupcor}\index{normalization!by finite blowups}
Let $\fX$ be a flat formal $R$-scheme of ff type whose rigid generic fiber $\fX^\rig$ is normal. Then the normalization of $\fX$ is a finite admissible blowup.
\end{cor}
\begin{proof}
Let $\phi\colon\tilde{\fX}\rightarrow\fX$ denote the normalization of $\fX$; by \cite{conrad_irred} Theore 2.1.3, $\phi^\rig$ is an isomorphism because $\fX^\rig$ is normal. By Lemma \ref{fingenisoisblowuplem}, it suffices to show that $\tilde{X}$ is $R$-flat. To do so, we may assume that $\fX$ is affine, $\fX=\Spf A$. Since $A$ is $R$-flat, it is contained in $A\otimes_RK$ which is normal and, hence, reduced; it follows that $A$ is reduced, which implies that the normalization of $A$ is the integral closure of $A$ in its total ring of fractions $\Frak(A)$. Since $A$ has no $\pi$-torsion, the same holds for $\Frak(A)$ and, hence, for the integral closure of $A$ in $\Frak(A)$, as desired.
\end{proof}

\section{Berthelot's construction and quasi-compactness}
Let $R$ be a complete discrete valuation ring with fraction field $K$ and residue field $k$, and let $\fX$ be a formal $R$-scheme of ff type. If $\fX$ is of tf type over $R$, then the rigid-analytic generic fiber $\fX^\rig$ of $\fX$ is quasi-compact by construction. The converse implication holds if $\fX$ is $R$-flat; this result is stated without proof in \cite{rapoport-zink} Proposition 5.12, where it is attributed to R.\ Huber. In the following, we provide a proof of the above statement, and we generalize it to a relative situation.

\begin{lem}\label{qctftlem}
Let $\fX$ be a flat formal $R$-scheme of ff type such that $\fX^\rig$ is quasi-compact. Then $\fX$ is of tf type over $R$.
\end{lem}


\begin{proof}
The property of being of locally tf type over $R$ can be checked locally on $\fX$; we may thus assume that $\fX$ is affine. There exists a flat $R$-model of tf type $\fY$ of $\fX^\rig$ together with a morphism $\phi\colon\fY \rightarrow \fX$ inducing an isomorphism of rigid generic fibers: indeed, since $\fX^\rig$ is quasi-compact and quasi-separated, there exists a flat $R$-model of tf type for $\fX^\rig$, cf.\ \cite{frg1} 4.1. Let $\ul{A}$ denote the $R$-algebra of global functions on $\fX$, and let $(s,t)$ be a formal generating system for $\ul{A}$ in the sense of \cite{urigspaces} 2.1. By \cite{frg1} 4.5, there exists a finite admissible blowup $\fY'$ of $\fY$ such that the components of $(s,t)$ extend to global functions on $\fY'$. After replacing $\fY$ by $\fY'$, the existence of $\phi$ follows from the fact that $R$-morphisms to $\fX$ correspond to continuous $R$-homomorphisms of rings of global sections, cf.\ \cite{egain} 10.4.6. 

Since $\fY$ is adic over $R$, there exists a natural number $n\geq 1$ such that the special fiber $\fY_k$ of $\fY$ factors through the $n$-th infinitesimal neighborhood of the smallest subscheme of definition $\fX_0$ of $\fX$ in $\fX$. Let $ \a\subseteq\ul{A}$ denote the ideal corresponding to $\fX_0$, let $\fX''\rightarrow \fX$ be the admissible blowup of $\fX$ in $(\pi, \a^{n+1})$, and let $\fX'\subseteq \fX''$ be the corresponding dilatation (cf.\ \cite{nicaise_traceformula} Def.\ 2.20). By the universal property of dilatations (cf.\ \cite{nicaise_traceformula} 2.22), $\phi$ extends uniquely to a morphism $\phi'\colon\fY \rightarrow \fX'$. Since $\fX$ is $R$-flat by assumption, the specialization map $\sp_{\fX''}$ is surjective. Since $\phi^\rig$ is an isomorphism, it follows that the inclusion $\fX'\subseteq \fX''$ is in fact an equality. Let $\ul{A}'$ denote the flat $R$-algebra of global functions on the affine formal scheme $\fX'$; then $\ul{A}'$ is of tf type over $R$, cf.\ \cite{nicaise_traceformula} Prop.\ 2.23 or \cite{dejong_crystalline} Lemma 7.1.2 (b). By \cite{urigspaces} Corollary 2.15, it follows that $\ul{A}$ is of tf type over $R$ as well.
\end{proof}

\begin{remark}
The flatness assumption in Lemma \ref{qctftlem} cannot be dropped; for instance, $\Spf k[[S]]$ is of ff type over $R$, but not of tf type over $R$, and its generic fiber is empty.
\end{remark}

We carry Lemma \ref{qctftlem} to a relative setting. To do so, we first generalize a construction in \cite{dejong_crystalline} 7.1.13 and \cite{dejong_crystalline_err}. Let us recall that if $\fX$ is a formal $R$-scheme of locally ff type, then the special fiber $\fX_k$ of $\fX$ is cut out by some uniformizing element $\pi$ of $R$; it is not necessarily a scheme, but a formal $k$-scheme of locally ff type. We refer to the end of Section 2.3 in \cite{nicaise_traceformula} for a discussion of dilatations in formal geometry.

\begin{lem}\label{dJlem}
Let $\fX$ be an affine flat formal $R$-scheme of ff type, and let $\fV\subseteq\fX$ be a closed formal subscheme. For each $n\in\N$, we let $\fV_{(n)}\subseteq\fX$ denote the $n$-th infinitesimal neighborhood of $\fV$ in $\fX$, and we let $\fX_{(n)}$ denote the dilatation of $\fX$ in $\fV_{(n)}\cap\fX_k$. There exists an integer $n_0$ such that for all $n\geq n_0$, there exists a natural closed immersion $\phi_n\colon\fV_{(n-n_0)}\hookrightarrow \fX_{(n)}$ such that the diagram
\[
\begin{diagram}
&&\fX_{(n)}\\
&\ruInto<{\phi_n}&\dTo\\
\fV_{(n-n_0)}&\rInto&\fX
\end{diagram}
\]
commutes, where the horizontal arrow is the given closed immersion and where the vertical arrow is the dilatation morphism. Moreover, for varying $n$, the $\phi_n$ are compatible.
\end{lem}
\begin{proof}
Let $A$ denote the $R$-algebra of global functions on $\fX$, and let $I\subseteq A$ be the ideal corresponding to $\fV$. By the Artin-Rees Lemma (\cite{eisenbudca} Lemma 5.1), there exists an integer $n_0$ such that for all $n\geq n_0$,
\[
\pi A\cap I^n\,=\, I^{n-n_0}(\pi A\cap  I^{n_0})\,\subseteq\,\pi I^{n-n_0}\;.
\]
We claim that for any integer $t\geq 1$ and any $n\geq n_0$,
\[
\pi^t A\cap(\sum_{i=1}^t \pi^{t-i} I^{in})\,\subseteq\,\pi^t I^{n-n_0}\;.
\]
Indeed, let us argue by induction on $t$. We have chosen $n_0$ such that the statement holds for $t=1$, so we may assume that $t>1$. Let $a$ be an element on the left hand side. Since $a$ lies in $\sum_{i=1}^t \pi^{t-i} I^{in}$, we can write
\[
a\,=\,\sum_{i=1}^t\pi^{t-i}a_i
\]
with elements $a_i\in I^{in}$. Since $a$ is divisible by $\pi$, the same must be true for $a_t$; hence $a_t\in\pi A\cap I^{tn}$, and $\pi A\cap I^{tn}\subseteq\pi I^{tn-n_0}$ by our choice of $n_0$. Since $n\geq n_0$, it follows that $a_t\in\pi I^{(t-1)n}$. We set $a_i'\mathrel{\mathop:}=a_i$ for $1\leq i \leq t-2$, $a_{t-1}'\mathrel{\mathop:}=a_{t-1}+\pi^{-1}a_t\in I^{(t-1)n}$; then $a_i'\in I^{in}$ for $1\leq i\leq t-1$, and
\[
a\,=\,\pi\sum_{i=1}^{t-1}\pi^{t-1-i}a'_i\;.
\]
Since $\pi^{-1}a\in\pi^{t-1}A$, the induction hypothesis implies that $a\in\pi^t I^{n-n_0}$, as desired; thus our claim has been shown. 
Let us now consider an element $a\in A[\frac{ I^n}{\pi}]\subseteq A\otimes_RK$. We choose a representation
\[
a\,=\,\sum_{i=0}^t\frac{a_i}{\pi^i}\quad,\quad a_i\in I^{in}
\]
of $a$. The class of $a_0$ modulo $ I^{n-n_0}$ does not depend on the choice of the representation. Indeed, if $a=0$, then $\pi^ta_0\in\pi^t I^{n-n_0}$ by the above claim. It is now clear that $a\mapsto a_0$ defines a surjective $A$-algebra homomorphism $A[I^n/\pi]\rightarrow A/ I^{n-n_0}$. The statement thus follows via formal completion with respect to an ideal of definition of $A$. 
\end{proof}

We can now provide the announced relative version of Lemma \ref{qctftlem}:

\begin{prop}\label{rqcadicprop}
Let $\phi\colon\fY\rightarrow \fX$ be a morphism of flat formal $R$-schemes of locally ff type. If $\phi^\rig$ is quasi-compact, then $\phi$ is adic.
\end{prop}
\begin{proof}
We may assume that $\fX$ and $\fY$ are affine, $\fX=\Spf A$, $\fY=\Spf B$. Using $\phi$, we consider $B$ as an $A$-algebra. Let $ \a\subseteq A$ be an ideal of definition; we must show that $\a B$ is an ideal of definition of $B$. Let $n\in\N$ be a natural number, and let $\fX'$ and $\fY'$ denote the dilatations of $\fX$ and $\fY$ in the respective ideals generated by $\pi$ and $ \a^{n+1}$; then $(\fY')^\rig$ is the $\phi^\rig$-preimage of $(\fX')^\rig$. By assumption on $\phi$, $(\fY')^\rig$ is quasi-compact, and by Lemma \ref{qctftlem}, it follows that $\fY'$ is of tf type over $R$. By Lemma \ref{dJlem}, there exists an $n_0\in\N$ such that for all $n\geq n_0$, there exists a closed immersion of $\Spf B/\a^{n-n_0}B$ into $\fY'$. It follows that $\a B$ is an ideal of definition for $B$.
\end{proof}

\section{Uniformly rigid geometry}

We establish results in uniformly rigid geometry (cf.\ \cite{urigspaces}) to be used later.

\subsection{Higher direct images and uniformly rigid generic fibers}
We show that the formation of higher direct images of coherent modules under proper morphisms of formal $R$-schemes of locally ff type commutes with the uniformly rigid generic fiber functor. Here and in the following, $R$ is a complete discrete valuation ring; let $K$ denote its field of fractions, let $\m_R$ and $k$ denote its maximal ideal and its residue field, and let $\pi\in\m_R$ be a uniformizer. We introduce the following notation: 
\begin{defi}\label{cohmodgenfibdefi}
If $\sF$ is a coherent sheaf on a formal $R$-scheme $\fX$ of locally ff type, we set
\[
\sF_K\,:=\,\sF\otimes_RK\;,
\]
where we regard $R$ and $K$ as constant sheaves on $\fX$.
\end{defi}
\begin{remark}
In the situation of Definition \ref{cohmodgenfibdefi}, we have $\sF_K(\fU)=\sF(\fU)\otimes_RK$ for every quasi-compact open subset $\fU\subseteq\fX$, because $K$ is $R$-flat and because $\fX$ is locally noetherian. Moreover, it is clear from the definitions that $\sF_K$ is naturally identified with $(\sp_\fX)_*\sF^\urig$, where $\sF^\urig$ denotes the uniformly rigid generic fiber of $\sF$ and where $\sp_\fX:\fX^\urig\rightarrow\fX$ is the specialization map, cf.\ \cite{urigspaces} Section 3.
\end{remark}

\begin{lem}\label{cohkacycliclem}
Let $\fX$ be an affine formal $R$-scheme of ff type, and let $\sF$ be a coherent $\O_\fX$-module. Then 
\[
H^q(\fX,\sF_K)\,=\,\check{H}^q(\fX,\sF_K)\,=\,0
\]
for all $q\geq 1$.
\end{lem}
\begin{proof}
By \cite{egaiii} 3.4.4, we know that $H^q(\fU,\sF)=0$ for every affine open part $\fU$ of $\fX$ and for all $q\geq 1$. Thus for every affine open part $\fU$ of $\fX$ and every affine open covering of $\fU$, the associated \v{C}ech-to-derived functor spectral sequence (cf.\ \cite{godement} 5.9) degenerates, which shows that $\check{H}^q(\fU,\sF)=H^q(\fU,\sF)$ for all $q\geq 0$; it follows that $\check{H}^q(\fU,\sF)=0$ for all $q\geq 1$. Since $\cdot\otimes_RK$ is an exact functor which commutes with direct limits, we obtain that $\check{H}^q(\fU,\sF_K)=0$ for all affine opens $\fU$ in $\fX$ and all $q\geq 1$. By Cartan's theorem (cf.\ \cite{godement} 5.9.2), it follows that $H^q(\fU,\sF_K)=0$ for all affine opens $\fU$ in $\fX$ and all $q\geq 1$.
\end{proof}

\begin{lem}\label{genfibbclem}
Let $\phi\colon\fY\rightarrow\fX$ be a proper morphism of formal $R$-schemes of locally ff type, and let $\sF$ be a coherent $\O_\fY$-module. Then for any $q\geq 0$, the natural homomorphism of $\O_{\fX,K}$-modules
\[
(R^q\phi_*\sF)_K\rightarrow (R^q\phi_*)(\sF_K)
\]
is an isomorphism.
\end{lem}
\begin{proof}
By \cite{egaiii} 3.4.6, it suffices to show that in the case where $\fX$ is affine, the natural homomorphism
\[
H^q(\fY,\sF)\otimes_RK\rightarrow H^q(\fY,\sF_K)
\]
is an isomorphism. By Lemma \ref{cohkacycliclem}, we may work with \v{C}ech cohomology instead of derived functor cohomology, and hence the statement follows from the fact that the functor $\cdot\otimes_RK$ is exact and that it commutes with direct limits, together with the fact that $\fY$ is quasi-compact when $\fX$ is affine.
\end{proof}

Let us recall that a formal blowup in the sense of \cite{temkin_desing} 2.1 is called admissible if it can be defined by a $\pi$-adically open ideal. Following the terminology in \cite{urigspaces}, a formal blowup of a locally noetherian formal scheme will, this paper, simply be called a blowup of that formal scheme. The following proposition shows that in the situation of Lemma \ref{genfibbclem}, the $\O_{\fX,K}$-modules $R^q\phi_*\sF_K$ are unaffected by admissible blowups of $\fX$. Let us point out that the notation $R^q\phi_*\sF_K$ is unambiguous by Lemma \ref{genfibbclem}. In the following, we will use without further comments the fact (cf.\ \cite{egaiii} 3.4.2) that higher direct images of coherent sheaves under proper morphisms of locally noetherian formal schemes are coherent.

\begin{prop}\label{blowupdirimcompprop}
Let $\phi\colon\fY\rightarrow\fX$ be a proper morphism of formal $R$-schemes of locally ff type, and let $\psi\colon\fX'\rightarrow\fX$ be an admissible blowup. Let
\[
\begin{diagram}
\fY'&\rTo^{\phi'}&\fX'\\
\dTo<{\psi'}&&\dTo>{\psi}\\
\fY&\rTo^{\phi}&\fX
\end{diagram}
\]
be the induced cartesian diagram in the category of flat formal $R$-schemes of locally ff type. Then for every coherent $\O_\fY$-module $\sF$ and every $q\geq 0$, the natural comparison morphism $\psi^*R^q\phi_*\sF\rightarrow R^q\phi'_*(\psi')^*\sF$
induces an isomorphism
\[
\psi^*R^q\phi_*\sF_K\rightarrow R^q\phi'_*(\psi')^*\sF_K\;.
\]
\end{prop}
\begin{proof}
By \cite{urigspaces} 2.24, $\psi'$ is an admissible blowup, and $\phi'$ is the strict transform of $\phi$ under $\psi$. By the ff type version of \cite{luetkebohmertformalrigid} 2.1, it suffices to show that the induced homomorphism
\[
R^q\phi_*\sF_K\rightarrow \psi_*R^q\phi'_*(\psi')^*\sF_K
\]
is an isomorphism. By Lemma \ref{genfibbclem} and the ff type version of \cite{luetkebohmertformalrigid} 2.1, the functors $R^q\psi_*$ and $R^q\psi'_*$ vanish, for $q\geq 1$, on sheaves obtained from coherent sheaves via $\cdot\otimes_RK$. Hence, the spectral sequences
\begin{eqnarray*}
R^p\psi_*R^q\phi'_*&\Rightarrow&R^{p+q}(\psi_*\phi'_*)\\
R^p\phi_*R^q\psi'_*&\Rightarrow&R^{p+q}(\phi_*\psi'_*)
\end{eqnarray*}
yield, for all $q\geq 0$, edge morphisms
\begin{eqnarray*}
R^q(\psi_*\phi'_*)&\rightarrow&\psi_*R^q\phi'_*\\
(R^q\phi_*)\psi'_*&\rightarrow&R^q(\phi_*\psi'_*)
\end{eqnarray*}
which are isomorphisms when evaluated in a sheaf that is obtained from a coherent sheaf via $\cdot\otimes_RK$. Since $\psi_*\phi'_*=\phi_*\psi'_*$, we obtain natural identifications
\begin{eqnarray*}
\psi_*R^q\phi'_*(\psi')^*\sF_K&\cong&R^q\psi_*\phi'_*(\psi')^*\sF_K\\
&\cong&R^q(\phi_*\psi'_*)(\psi')^*\sF_K\\
&\cong&R^q(\phi_*)\psi'_*(\psi')^*\sF_K\\
&\cong&R^q\phi_*\sF_K\;,
\end{eqnarray*}
where the last equality is again the ff type version of \cite{luetkebohmertformalrigid} 2.1. One checks that the resulting isomorphism agrees with the homomorphism that is induced by the comparison homomorphism.
\end{proof}

\begin{cor}\label{highdirimsamaffpbcor}
The statement of Proposition \ref{blowupdirimcompprop} still holds when $\psi$ is a composition of admissible blowups, open immersions and completion morphisms.
\end{cor}
\begin{proof}
We may assume that $\psi$ is an open immersion, a completion morphism or an admissible blowup. If $\psi$ is an open immersion, there is nothing to show, and if $\psi$ is a completion morphism, the statement follows from Theorem \ref{mainflatbasechangethm}, completion morphisms of locally noetherian formal schemes being flat. In the remaining case where $\psi$ is an admissible blowup, the desired statement is provided by Proposition \ref{blowupdirimcompprop}.
\end{proof}

We conclude that the formation of higher direct images commutes with passage to uniformly rigid generic fibers:

\begin{cor}\label{propdirimgenfibercor}
Let $\ul{\phi}\colon\fY\rightarrow\fX$ be a proper morphism of formal $R$-schemes of locally ff type, let $\phi\colon Y\rightarrow X$ denote its uniformly rigid generic fiber, and let $\sF$ be a coherent $\O_\fY$-module. Then $R^q\phi_*(\sF^\srig)$ is a coherent $\O_X$-module for each $q\geq 0$, and the morphism 
\[
\sp_\fY^\sharp\colon\O_\fY\rightarrow\sp_{\fY,*}\O_Y
\]
(cf.\ \cite{urigspaces} Section 2.4) induces natural isomorphisms
\[
(R^q\ul{\phi}_*\sF)^\srig\overset{\sim}{\rightarrow}R^q\phi_*(\sF^\srig)
\]
for $q\geq 0$.
\end{cor}
\begin{proof}
Let $q\geq 0$ be fixed. By \cite{tohoku} 3.7.2, whose proof also works for general Grothendieck topologies, $R^q\phi_*(\sF^\srig)$ is associated to presheaf 
\[
V\mapsto H^q(\phi^{-1}(V),\sF^\srig)\;,
\]
where $V$ varies among the admissible subsets of $X$. Let $(\fX_i)_{i\in I}$ be an affine open covering of $\fX$; then the semi-affinoid subdomains in the semi-affinoid subspaces $X_i\mathrel{\mathop:}=\fX_i^\srig$ of $X$ form a basis of the uniformly rigid G-topology of $X$. Let $U$ be a semi-affinoid subdomain of $X_i$ for some $i\in I$, and let $\fU\rightarrow\fX_i$ be a morphism which represents $U$ as a semi-affinoid subdomain in $X$ in the sense of \cite{urigspaces} Definition 2.22; by Corollary \ref{highdirimsamaffpbcor}, we have a natural identification
\[
(R^q\ul{\phi}_*\sF)^\srig|_U\,\cong\,R^q\ul{\phi}_{\fU}^*(\sF|_{\fY_\fU})^\srig\;,
\]
where $\ul{\phi}_{\fU}$ and $\sF|_{\fY_\fU}$ are defined via pullback to $\fU$ and $\fY_\fU\mathrel{\mathop:}=\fY\times'_\fX\fU$ respectively and where $\times'$ denotes fibered product in the category of flat formal $R$-schemes of locally ff type, cf.\ \cite{urigspaces} Lemma 2.24 and the subsequent remark. It thus suffices to observe, in the case where $\fX$ is affine, that there is a natural isomorphism
\[
H^q(\fY,\sF)\otimes_RK\overset{\sim}{\rightarrow}H^q(Y,\sF^\srig)\;.
\]
We may use the \v{C}ech cohomology groups attached to an affine covering $(\fY_i)_{i\in I}$ of $\fY$ to calculate the derived functor cohomology groups. Indeed, $(\fY_i)_{i\in I}$ is a Leray covering of $\fY$ for $\sF$ by \cite{egaiii} 3.4.4, and $(\fY_i^\srig)_{i\in I}$ is a Leray covering for $Y$ by \cite{urigspaces} Corollary 2.43. Since $\cdot\otimes_RK$ is exact, we thus obtain the desired natural isomorphisms.
\end{proof}

\subsection{Schematic images and schematic closures}

If $\fX$ is a formal $R$-scheme of locally ff type and if $Z\subseteq X$ is a closed uniformly rigid subspace (cf.\ \cite{urigspaces} Def.\ 3.9 and Prop.\ 3.11), then according to the discussion in \cite{urigspaces} Section 3.1 there exists a unique closed $R$-flat formal subscheme $\fZ$ of $\fX$ such that $\urig$ identifies $\fZ^\urig$ with $Z$; the closed formal subscheme $\fZ\subseteq\fX$ is called the schematic closure of $Z\subseteq X$ in $\fX$.

\begin{lem}\label{schemcldirimlem}
Let $\ul{\phi}\colon\fY\rightarrow\fX$ be a proper morphism of formal $R$-schemes of locally ff type whose uniformly rigid generic fiber $\phi\colon Y\hookrightarrow X$ is a closed immersion. Then the schematic closure of $Y$ in $\fX$ coincides with the schematic image of $\ul{\phi}$. Moreover, the natural morphism from $\fY$ to the schematic image of $\ul{\phi}$ is surjective on physical points.
\end{lem}
\begin{proof}
Since $\fY$ is $R$-flat, the schematic image $\fV$ of $\ul{\phi}$ is $R$-flat as well. To show that $\fV$ is the schematic closure of $Y$ in $\fX$, it remains to see that the unique factorization $\ul{\phi}'\colon\fY\rightarrow\fV$ of $\ul{\phi}$ induces an isomorphism of uniformly rigid generic fibers. Let $\sI\subseteq\O_\fX$ be the coherent ideal defining $\fV$ as a closed formal subscheme of $\fX$. The exact sequence of coherent $\O_\fX$-modules
\[
0\rightarrow\sI\rightarrow\O_\fX\rightarrow\ul{\phi}_*\O_\fY
\]
induces an exact sequence of coherent $\O_X$-modules
\[
0\rightarrow\sI^\srig\rightarrow\O_X\rightarrow(\ul{\phi}_*\O_\fY)^\srig\;,
\]
the functor $\srig$ being exact. By Corollary \ref{propdirimgenfibercor}, it follows that
\[
0\rightarrow\sI^\srig\rightarrow\O_X\rightarrow \phi_*\O_Y\rightarrow 0
\]
is exact. Hence, $Y$ and the uniformly rigid generic fiber of $\fV$ coincide as closed uniformly rigid subspaces of $X$, as stated.

It remains to see that $\ul{\phi}'$ is surjective. Since $\fV$ is $R$-flat, the specalization map $\sp_{\fV}$ is surjective onto the set of closed points of $\fV$;  since $\ul{\phi}'$ is an isomorphism on uniformly rigid generic fibers, it follows that the image of $\ul{\phi}'$ contains every closed point of $\fV$. Since $\ul{\phi}'$ is of locally ff type, the image of $\ul{\phi}'$ is locally constructible by \cite{egaiv} 1.8.5. Hence, its complement its locally constructible as well.
Since the closed points of $\fV$ lie very dense, we conclude from \cite{egain} 0.2.6.2 (b) that $\ul{\phi}'$ is surjective, as desired.
\end{proof}

\subsection{Properness of graphs}\label{graphpropersec}
Let $\fX$ and $\fY$ be flat formal $R$-schemes of locally ff type with uniformly rigid generic fibers $X$ and $Y$. Let us assume that $\fX$ is separated; then $X$ is separated as well. Let $\phi\colon Y\rightarrow X$ be a uniformly rigid morphism, and let $\ul{\Gamma}_\phi$ denote the schematic closure of the graph of $\phi$ in $\fY\times\fX$, where now and in the following fibered products without indication of the base should be understood over $R$ or $K$, depending on the context. In this section, we prove that the projection $p_\fY|_{\ul{\Gamma}_\phi}$ from $\ul{\Gamma}_\phi$ to $\fY$ is proper.

\begin{lem}\label{graphclosureadiclem}
The projection $p_\fY|_{\ul{\Gamma}_\phi}$ is adic.
\end{lem}
\begin{proof}
Indeed, the morphism $p_\fY|_{\ul{\Gamma}_\phi}$ induces an isomorphism of uniformly rigid generic fibers and, hence, an isomorphism of rigid-analytic generic fibers. Since isomorphisms are quasi-compact, we conclude with Proposition \ref{rqcadicprop} that $p_\fY|_{\ul{\Gamma}_\phi}$ is adic.
\end{proof}

\begin{lem}\label{graphproperlem}
Let us assume that there exists a proper morphism $\ul{\psi}\colon\fY'\rightarrow\fY$ inducing an isomorphism of uniformly rigid generic fibers such that $\phi$ extends to a morphism $\ul{\phi}\colon\fY'\rightarrow\fX$. Then $p_\fY|_{\ul{\Gamma}_\phi}$ is proper.
\end{lem}
\begin{proof}
By Lemma \ref{schemcldirimlem}, $\ul{\Gamma}_\phi$ coincides with the schematic image of the proper morphism
\[
(\ul{\psi}\times\id_\fX)\circ\Gamma_{\ul{\phi}}\quad\colon\quad\fY'\hookrightarrow\fY'\times\fX\rightarrow\fY\times\fX\quad.
\]
Let $\ul{\phi}'\colon\fY'\rightarrow\ul{\Gamma}_\phi$ denote the unique factorization of this morphism; then $\ul{\phi}'$ is an isomorphism on uniformly rigid generic fibers. Let us consider the commutative diagram
\[
\begin{diagram}
\fY'&\rTo^{\ul{\phi}'}&\ul{\Gamma}_\phi\\
&\rdTo<{\ul{\psi}}&\dTo>{p_\fY|_{\ul{\Gamma}_\phi}}\\
&&\fY
\end{diagram}
\]
of morphisms inducing isomorphisms of uniformly rigid generic fibers. As in the proof of Lemma \ref{graphclosureadiclem}, we use Proposition \ref{rqcadicprop} to see that the morphisms $\ul{\phi}'$ and $p_\fY|_{\ul{\Gamma}_\phi}$ are of tf type. By Lemma \ref{schemcldirimlem}, $\ul{\phi}'$ is surjective onto the physical points of $\ul{\Gamma}_\phi$. Since $\fX$ is separated, $p_\fY$ is separated, and hence $p_\fY|_{\ul{\Gamma}_\phi}$ is separated as well. After reducing the above diagram of adic morphisms of locally noetherian formal schemes modulo an ideal of definition, we conclude from \cite{egaii} 5.4.3 ($ii$) that $p_\fY|_{\ul{\Gamma}_\phi}$ is proper, as desired.
\end{proof}


The conclusion of Lemma \ref{graphproperlem} holds without the assumption on the existence of the morphism $\ul{\psi}$, which is the statement of the following theorem. To prove it, we use induction over a treelike covering of $Y$ (cf.\ \cite{urigspaces} Definition 2.31).

\begin{thm}\label{graphproperthm}
The morphism $p_\fY|_{\ul{\Gamma}_\phi}\colon\ul{\Gamma}_\phi\rightarrow\fY$ is proper. 
\end{thm}
\begin{proof}
By Lemma \ref{graphclosureadiclem}, $\ul{\Gamma}_\phi$ is adic over $\fY$. To show that $p_\fY|_{\ul{\Gamma}_\phi}$ is proper, we may assume that $\fY$ is affine such that $Y$ is semi-affinoid. Let $(\fX_i)_{i\in I'}$ be an affine open covering of $\fX$, and let $(X_i)_{i\in I'}$ be the induced admissible covering of $X$.  Let $(Y_i)_{i\in I}$ be a leaflike refinement of $(\phi^{-1}(X_i))_{i\in I'}$, and let $(Y_j)_{j\in J}$ be a treelike covering of $Y$ enlarging $(Y_i)_{i\in I}$ together with a suitable rooted tree structure on $J$ such that $I$ is identified with the set of leaves of $J$. Let us choose a formal presentation 
\[
(\fY_j,\ul{\phi}_j,\ul{\beta}_j)_{j\in J}
\]
of $(Y_j)_{j\in J}$ with respect to $\fY$ according to the discussion in \cite{urigspaces} following Def.\ 2.31. By \cite{urigspaces} Cor.\ 2.14 ($iv$) and ($ii$), we may replace the $\fY_j$ with $j\in\leaves(J)$ by finite admissible blowups such that for each leaf $j$ of $J$, the restriction 
\[
\phi|_{Y_j}\colon Y_j\rightarrow X
\]
extends uniquely to a morphism $\ul{\phi}_j\colon\fY_j\rightarrow\fX$. Let us forget that the morphisms 
\[
\ul{\beta}_j\colon\ul{\fY}_j'\rightarrow\ul{\fY}_j
\]
are admissible blowups, and let us merely retain the information that they are proper and that they induce isomorphisms of uniformly rigid generic fibers. Let us argue by induction along $J$. The statement is trivially true if we replace $\fY$ be $\fY_j$ for some $j\in\leaves(J)$, which gets our inductive argument off the ground. By the induction hypothesis, we may now assume that for any child $j$ of the root $r$ of $J$, the schematic closure $\ul{\Gamma}_{\phi_j}$ of the graph of $\phi_j\mathrel{\mathop:}=\phi|_{Y_j}$ in $\fY_j\times\fX$ is proper over $\fY_j$. By uniqueness, the formation of schematic closures commutes with localization; hence the $\ul{\Gamma}_{\phi_j}$ glue to a closed formal subscheme $\ul{\Gamma}'_{\phi_r}\subseteq\fY_r'\times\fX$ that is proper over $\fY_r'$ and, hence, proper over $\fY_r$. We are now in the situation of Lemma \ref{graphproperlem}, with $\fY'=\ul{\Gamma}'_{\phi_r}$, so we may conclude that $\ul{\Gamma}_\phi$ is indeed proper over $\fY=\fY_r$.
\end{proof}

\begin{cor}\label{graphpropercor}
If $\fX$ and $\fY$ are formal $R$-schemes of locally ff type, if $\fX$ is separated and if $\phi:\fY^\urig\rightarrow\fX^\urig$ is a morphism of uniformly rigid generic fibers, then exists a proper morphism $\ul{\psi}\colon\fY'\rightarrow\fY$ inducing an isomorphism of uniformly rigid generic fibers such that $\phi$ extends to a morphism $\ul{\phi}$ from $\fY'$ to $\fX$.
\end{cor}
\begin{proof}
Indeed, by Theorem \ref{graphproperthm} we may take $\fY'=\ul{\Gamma}_\phi$ and $\ul{\psi}=p_\fY|_{\ul{\Gamma}_\phi}$.
\end{proof}

\begin{remark}Let us make the following remarks:
\begin{enumerate}
\item The proper morphism $p_\fY|_{\ul{\Gamma}_\phi}$ will in general not be induced, locally on $\fY$, by proper morphisms of schemes via formal completion. 
\item Corollary \ref{graphpropercor} may be viewed as a partial uniformly rigid analog of \cite{bosch_frgnotes} Theorem 2.8/3, which says that if $\fX$ and $\fY$ are quasi-paracompact and of locally tf type over $R$, then every morphism of rigid-analytic generic fibers $\phi:\fY^\rig\rightarrow\fX^\rig$ extends to a morphism $\fY'\rightarrow\fX$ for some admissible blowup $\fY'$ of $\fY$.
\item As we have pointed out in the introduction, the example of an unbounded function on the open unit disc shows that the statement of Corollary \ref{graphpropercor} does not hold if we consider rigid-analytic generic fibers instead of uniformly rigid generic fibers.
\end{enumerate}
\end{remark}

\subsection{Change of base field}

Let $K'/K$ be a possibly non-finite discrete analytic field extension, and let $R'$ denote its valuation ring. We define a base extension functor
\[
\cdot\hat{\otimes}_KK'\colon\sAff_K\rightarrow\sAff_{K'}
\]
as follows: if $A$ is a semi-affinoid $K$-algebra, then by \cite{urigspaces} Corollary 2.14 ($iv$) the system of $R$-model of ff type $\ul{A}\subseteq A$ is directed and functorial in $A$; we set
\[
A\hat{\otimes}_KK'\,\mathrel{\mathop:}=\,\varinjlim_{\ul{A}}(\ul{A}\hat{\otimes}_RR')\otimes_{R'}K'\;,
\]
where $\ul{A}\subseteq A$ varies amongst the $R$-models of ff type for $A$. By \cite{urigspaces} Corollary 2.14 ($ii$), an inclusion $\ul{A}\subseteq\ul{A}'$ of $R$-models of ff type for $A$ corresponds to a finite admissible blowup of formal spectra; by \cite{urigspaces} Lemma 2.24, the induced inclusion $\ul{A}\hat{\otimes}_RR'\subseteq\ul{A}'\hat{\otimes}_RR'$ of flat $R'$-algebras of ff type corresponds to a finite admissible blowup as well and, hence, yields an isomorphism after applying $\cdot\otimes_{R'}K'$. Thus all transition morphisms in the above inductive system of semi-affinoid $K'$-algebras are in fact isomorphisms, which shows that the functor $\cdot\hat{\otimes}_KK'$ is well-defined.

\begin{remark}
It is possible to define base extension functors for non-semi-affinoid uniformly rigid spaces satisfying certain finiteness conditions.
\end{remark}

\subsection{Galois descent}

We outline some aspects of Galois descent theory for uniformly rigid $K$-spaces. Let $K'/K$ be a finite Galois field extension with Galois group $\Gamma$, and let $R'$ denote the normalization of $R$ in $K'$.

By \cite{urigspaces} 2.2 ($iv$), the category of semi-affinoid $K$-spaces is anti-equivalent to the category of semi-affinoid $K$-algebras. Hence, Galois descent for $K$-algebras shows that the natural functor from semi-affinoid $K$-spaces to semi-affinoid $K'$-spaces equipped with a semi-linear $\Gamma$-action is fully faithful.

\begin{lem}\label{semaffdesclem}
Let $A'$ be a semi-affinoid $K'$-algebra together with an action of  $\Gamma=\Gal(K'/K)$ which extends the tautological action on $K'$. Then $A\mathrel{\mathop:}=(A')^\Gamma$ is a semi-affinoid $K$-algebra, and the natural $K'$-homomorphism
\[
A\otimes_KK'\rightarrow A'
\]
is a $\Gamma$-equivariant isomorphism.
\end{lem}
\begin{proof}
The last statement follows by means of Galois descent for $K$-algebras; it suffices to see that $A$ is semi-affinoid over $K$. Let $\ul{A}'$ be an $R'$-model of ff type for $A'$. The $\Gamma$-translates of a power-bounded element of $A'$ are again power-bounded; by \cite{urigspaces} Corollary 2.12, we may thus assume that $\ul{A}'$ admits a $\Gamma$-stable formal generating system over $R'$ and, hence, is $\Gamma$-stable. Let 
\[
((s_1,\ldots,s_m),(t_1,\ldots,t_n))
\]
be any formal generating system of $\ul{A}'$ over $R'$, and let $r$ denote the cardinality of $\Gamma$. For each $1\leq i\leq m$, let $u_{i,1},\ldots,u_{i,r}$ be the elementary symmetric polynomials in $(\gamma(s_i))_{\gamma\in \Gamma}$. Similarly, for each $1\leq j\leq n$ let $v_{j,1},\ldots,v_{j,r}$ be the elementary symmetric polynomials in $(\gamma(t_j))_{\gamma\in\Gamma}$. Then the $u_{il}\in\ul{A}'$ are topologically nilpotent. Since $R'/R$ is finite, we may consider $\ul{A}'$ as an $R$-algebra of ff type. Let $\ul{B}\subseteq\ul{A}'$ denote the $R$-algebra of ff type generated by the $u_{il}$ and the $v_{jl}$. That is, $\ul{B}$ is the image of the $R$-homomorphism
\[
R[[(X_{il})_{1\leq i\leq m\,,\,1\leq l\leq r}]]\langle (Y_{jl})_{1\leq j\leq n\,,\,1\leq l\leq r}\rangle\rightarrow\ul{A}'
\]
sending $X_{il}$ to $u_{il}$ and $Y_{jl}$ to $v_{jl}$. The elements of $\ul{B}$ are $\Gamma$-stable, and hence $B\mathrel{\mathop:}=\ul{B}\otimes_RK$ is a $K$-subalgebra of $A$. Since $B$ is semi-affinoid, it suffices to see that the inclusion $B\subseteq A$ is finite. This may be verified after the base change $\cdot\otimes_KK'$; hence it suffices to see that $\ul{A}'$ is finite over $\ul{B}'\mathrel{\mathop:}=\ul{B}\otimes_RR'$, which is the $R'$-subalgebra of ff type of $\ul{A}'$ generated by the $u_{il}$ and the $v_{jl}$. Let $\b'$ denote the ideal of definition of $\ul{B}'$ induced by the given formal presentation of $\ul{B'}$. By the formal version of Nakayama's Lemma, cf.\ \cite{eisenbudca} Ex. 7.2, it suffices to see that $\ul{A}'/\b'\ul{A}'$ is finite over $\ul{B}'/\b'$. Since each $s_i\in\ul{A}'$ satisfies an integral equation over $\ul{B}'$ with coefficients in $\b'$, the ideal $\b'\ul{A}'$ is an ideal of definition for $\ul{A}'$. Hence $\ul{A}'/\b'\ul{A}'$ modulo its nilradical is generated by the $t_j$ as a $k$-algebra, and the $t_j$ are likewise integral over $\ul{B}'$. We conclude that $\ul{A}'/\b'\ul{A}'$ modulo its nilradical is finite over $\ul{B}'$, and again by \cite{eisenbudca} Ex. 7.2, it follows that $\ul{A}'/\b'\ul{A}'$ is finite over $\ul{B}'$, as desired.
\end{proof}

\begin{cor}\label{maindescentcor}
The functor $\cdot\otimes_KK'$ is an equivalence from the category of semi-affinoid $K$-spaces to the category of semi-affinoid $K'$-spaces  with a semi-linear $\Gamma$-action.
\end{cor}

Descent of morphisms works for semi-affinoid targets:

\begin{prop}\label{descffprop}
The functor $X\mapsto X\otimes_KK'$ from the category of uniformly rigid $K$-spaces to the category of uniformly rigid $K'$-spaces together with a semi-linear $\Gamma$-action is faithful; that is, the natural map
\[
\Hom_K(Y,X)\rightarrow\Hom_{K'}(Y_{K'},X_{K'})^\Gamma
\]
is injective for all uniformly rigid $K$-spaces $X$ and $Y$. If $X$ is semi-affinoid, then the above map is even bijective.
\end{prop}
\begin{proof}
We use a subscript $K'$ to indicate the image under the functor $\cdot\times_KK'$. Let $X$ and $Y$ be uniformly rigid $K$-spaces, and let $\phi,\psi\colon Y\rightarrow X$ be morphisms such that $\phi_{K'}=\psi_{K'}$. The projection $Y_{K'}\rightarrow Y$ being surjective, $\phi$ and $\psi$ must coincide on physical points; we may thus assume that $X$ and $Y$ are semi-affinoid. Let $V\subseteq Y$ be the closed uniformly rigid coincidence subspace of $\phi$ and $\psi$, that is, the pullback of the diagonal of $X\times_K X$ under the morphism $(\phi,\psi)$. Its formation commutes with base change, which implies that $V_{K'}=Y_{K'}$; we conclude that $V=Y$, which means that $\phi=\psi$. 

Let us now assume that $X$ is semi-affinoid, and let $\phi'\colon Y_{K'}\rightarrow X_{K'}$ be a morphism that is equivariant with respect to the $\Gamma$-actions. By what we have shown so far, we may assume that $Y$ is semi-affinoid as well. Since morphisms of semi-affinoid $K$-spaces correspond to homomorphisms of rings of global functions, finite Galois descent for $K$-algebras shows that $\phi'$ descends to a morphism $\phi\colon Y\rightarrow X$.
\end{proof}

\begin{remark}
Since the category of semi-affinoid $K$-spaces is closed under the formation of fibered products, a semi-affinoid $K'$-group equipped with a semi-linear $\Gamma$-action descends uniquely to a semi-affinoid $K$-group.
\end{remark}

Let us conclude this paragraph by listing some open questions on Galois descent for uniformly rigid spaces:

\begin{enumerate}
\item Can a semi-affinoid $K'$-space with a semi-linear $\Gamma$-action possibly descend to a uniformly rigid $K$-space which is not semi-affinoid?
\item Let $X'$ be a semi-affinoid $K'$-space equipped with a semi-linear $\Gamma$-action, and let $U'\subseteq X'$ be a $\Gamma$-stable semi-affinoid subdomain. By Corollary \ref{maindescentcor} and Proposition \ref{descffprop}, $U'$ descends to a semi-affinoid $K$-space $U$, and the given open immersion $U'\hookrightarrow X'$ descends to a morphism $U\rightarrow X$ which defines $U$ as a representable subset of $X$. Is $U$ necessarily admissible open in $X$?
\item Let $X$ and $Y$ be uniformly rigid $K$-spaces, and let $\phi'\colon Y\otimes_KK'\rightarrow X\otimes_KK'$ be a $\Gamma$-equivariant morphism; then there exists a unique map on physical points $|\phi|\colon|Y|\rightarrow|X|$ such that $|\phi|$ spans a commutative square with $|\phi'|$ and with the maps of physical points that are associated to the projections $Y\otimes_KK'\rightarrow Y$ and $X\otimes_KK'\rightarrow X$. Is $|\phi|$ is continuous? If $U\subseteq X$ is admissible open, then $|\phi|^{-1}(U)\otimes_KK'$ is admissible open in $Y\otimes_KK'$, but we do not know whether this implies the admissibility of $|\phi|^{-1}(U)$ in $Y$. 
\end{enumerate}

\subsection{Tubes around closed uniformly rigid subspaces}

We prove that if $X$ is a semi-affinoid $K$-space and if $Z\subseteq X$ is a closed uniformly rigid subspace which is contained in a finite union $U$ of retrocompact semi-affinoid subdomains in $X$ (cf.\ \cite{urigspaces} Def.\ 2.22), then $U$ contains a tube around $Z$ in $X$. The corresponding statement in rigid geometry is established in \cite{kiehlderham} Satz 1.6 and in \cite{kisinlocconst} Lemma 2.3. We follow the approach in \cite{kisinlocconst}; in doing so, we have to overcome the problem that Kisin's arguments use a nontrivial existence result concerning formal models of tf type for certain admissible open subspaces (\cite{frg1} 4.4) which is not available in the uniformly rigid setting.

The following Proposition generalizes \cite{urigspaces} Prop.\ 2.40; the corresponding statement for affinoid spaces is provided by \cite{kisinlocconst} Lemma 2.1. For the notion of retrocompact coverings for semi-affinoid spaces, we refer to \cite{urigspaces} Def.\ 2.33. In the proof of the following proposition, we will use rooted trees to describe coverings of semi-affinoid spaces, cf.\ \cite{urigspaces} Section 2.3.3. If $I$ is a rooted tree, we write $v(I)$ to denote cardinality or volume of $I$, and we let $\leaves(I)$ denote its set of leaves. Moreover, if $i\in I$ is a vertex, we let $\children(i)$ denote the set of children of $i$, we write $\subtree(i)$ to denote the rooted subtree with root $i$, and if $i$ is different from the root of $I$, we let $\parent(i)$ denote the parent of $i$. A rooted tree $I$ is called linear if each vertex has at most one child, and $I$ is called almost linear if the same holds except possibly for the root.

\begin{prop}\label{zarretromixadmprop}
Let $X$ be a semi-affinoid $K$-space, and let $(X_i)_{i\in I}$ be a finite covering of $X$, where for each $i\in I$, the admissible open subspace $X_i\subseteq X$ is
\begin{packed_enum}
\item Zariski-open in the sense of \cite{urigspaces} Prop.\ 2.40 or
\item a retrocompact semi-affinoid subdomain.
\end{packed_enum}
Then the covering $(X_i)_{i\in I}$ of $X$ has a retrocompact refinement; in particular, it is $\sT_\srig$-ad\-mis\-si\-ble.
\end{prop}
\begin{proof}
Let us write $I$ as a disjoint union of subsets $I_z$ and $I_r$ such that $X_i$ is Zariski-open in $X$ for every $i\in I_z$ and such that $X_i$ is a retrocompact semi-affinoid subdomain in $X$ for every $i\in I_r$. After passing to a refinement, we may assume that for each $i\in I_z$, $X_i=D(f_i)$ for some semi-affinoid function $f_i$ on $X$, where $D(f_i)$ denotes the non-vanishing locus of $f_i$. We proceed by induction on the cardinality of $I_z$. If $I_z$ is empty, there is nothing to show, so let us assume that $I_z$ is nonempty.

In a first step, we show that we can refine the covering $(X_i)_{i\in I}$ by adding finitely many retrocompact subdomains such that, after passing to this refinement, the inclusion
\[
\bigcup_{j\in I_z}Z_j\subseteq\bigcup_{i\in I_r}X_i\quad\quad(*)
\]
holds: let us fix an element $j\in I_z$, and let $Z_j=V(f_j)\subseteq X$ denote the Zariski-closed complement of $X_j$ in $X$, which we consider as a reduced closed semi-affinoid subspace of $X$. For each $i\in I_z$, the intersection $X_i\cap Z_j$ is the non-vanishing locus of the restriction of $f_i$ to $Z_j$; hence by our induction hypothesis, the covering $(X_i\cap{Z_j})_{i\in I\setminus\{j\}}$ of $Z_j$ has a retrocompact refinement $(Z_{jl})_{l\in L_j}$. Let $\alpha_j\colon L_j\rightarrow I\setminus\{j\}$ be a corresponding refinement map, and let $l$ be an element of $\alpha_j^{-1}(I_z\setminus\{j\})$, so that the retrocompact subdomain $Z_{jl}\subseteq Z_j$ is contained in the Zariski-open subset $Z_j\cap X_{\alpha_j(l)}$ of $Z_j$. We claim that there exists a retrocompact semi-affinoid subdomain $X_{jl}'$ in $X$ that is contained in $X_{\alpha_j(l)}$ and that contains $Z_{jl}$. Indeed, by \cite{urigspaces} Prop.\ 2.40 and the subsequent discussion, there exists a $\sT_\srig$-admissible covering of $X_{\alpha_j(l)}$ by retrocompact semi-affinoid subdomains $X_{\alpha_j(l),\geq\varepsilon}\subseteq X$, with $\varepsilon\in\sqrt{|K^*|}$ tending to zero. It follows that $(X_{\alpha_j(l),\geq\varepsilon}\cap Z_j)_{\varepsilon\rightarrow 0}$ is an admissible covering of $Z_j\cap X_{\alpha_j(l)}$; since $Z_{jl}$ is quasi-compact, it suffices to take $X'_{jl}=X_{\alpha_j(l),\geq\varepsilon}$ with $\varepsilon>0$ small enough. If we enlarge the covering $(X_i)_{i\in I}$ by adding the finitely many retrocompact subdomains $X'_{jl}$ with varying $l$ and $j$, then the new covering which we obtain is a finite refinement of of the original one since $X'_{jl}\subseteq X_{\alpha_j(l)}$,  this refinement satisfies ($*$) holds, as desired. Thus from now on, we may and do assume that ($*$) holds for $(X_i)_{i\in I}$.

Let $\fX$ be a flat affine formal $R$-model of ff type for $X$. For each $i\in I_r$, we consider a representation $\fX_i\rightarrow\fX$ of $X_i$ as a retrocompact semi-affinoid subdomain of $X$ with respect to $\fX$ in the sense of \cite{urigspaces} Def.\ 2.22, together with a factorization into open immersions and admissible blowups; the existence of such formal data is a consequence of \cite{urigspaces} Cor.\ 2.27 and the remark following \cite{urigspaces} Cor.\ 2.29. Let $J$ denote the almost linear rooted tree obtained by glueing the resulting linear rooted trees along their roots, let $r\in J$ be the root, and let $(\fX_j, \phi_j,\beta_j)_{j\in J}$ be the associated formal data, where for each $j\in J$ which is not a leaf, $\beta_j:\fX_j'\rightarrow\fX_j$ is an admissible blowup, and where $\phi_j:\fX_j\hookrightarrow\fX_{\parent(j)}'$ is an open immersion for each $j\in J\setminus\{r\}$. We construct a rooted tree $J'$ as follows: let $J'$ first be a copy of $J$, and let $r'\in J'$ denote its root. For each ordered pair $(j,j')\in\children(r)\times\children(r)$ with $j'\neq j$, we glue a copy of $\subtree(j')$ to the vertex in $J'$ corresponding to $j$. Now $J'$ is a rooted tree such that $\children(r)$ is naturally identified with $\children(r')$ and such that for each $j\in\children(r')$, $\subtree(j)$ is a linear rooted tree with volume $v(J)-1$. We can now recursively apply this construction to each such $\subtree(j)$, for $j\in\children(r')$; since the volumes strictly decrease, this process must terminate. Let $J'$ denote the rooted tree that we obtain in the end. Via pullback in the category of flat formal $R$-schemes of ff type (cf.\ \cite{urigspaces} Lemma 2.24), $(\fX_j,\phi_j,\beta_j)_{j\in J}$ induces a formal structure $(\fX_j,\phi_j,\beta_j)_{j\in J'}$ on $J'$. In particular, for each $j\in J'$ we have an attached flat formal $R$-scheme of ff type $\fX_j$ with a uniformly rigid generic fiber $X_j$. If $j\in\leaves(J')$, then by \cite{urigspaces} Cor.\ 2.25, $X_j$ is a finite intersection of the $X_i$, with $i$ varying in $I_r$; moreover,
\[
\bigcup_{i\in I_r}X_i\,=\,\bigcup_{j\in\leaves(J')}X_j\quad\quad(**)
\]
by construction.

For each $i\in I_z$ and each $j\in J'$, we let $\fZ_{ij}$ denote the schematic closure of $Z_i\cap X_j$ in $\fX_j$. Since $\fX_j$ is noetherian, the open formal subscheme $\fX_j\setminus\fZ_{ij}$ has a finite affine open covering. Hence the uniformly rigid generic fiber $(\fX_j\setminus\fZ_{ij})^\urig$ is a finite union of retrocompact semi-affinoid subdomains in $X$ which is contained in the complement $X_i$ of $Z_i$. It thus suffices to see that the $X_j$, $j\in \leaves(J')$, together with the $(\fX_j\setminus\fZ_{ij})^\srig$ with $i\in I_z$, $j\in J'$ cover $X$, because by what we have just seen, we then obtain a retrocompact refinement of $(X_i)_{i\in I}$. Let $x\in X$ be a point; we must show that
\[
x\in \bigcup_{j\in\leaves(J')}X_j\quad\cup\quad\bigcup_{j\in J',i\in I_z}(\fX_j\setminus\fZ_{ij})^\srig\;.
\]
By ($*$) and ($**$), we have an inclusion $\bigcup_{i\in I_z} Z_i\subseteq\bigcup_{j\in \children(r')} X_j$. Since the $\fZ_{ij}$ are $R$-flat, the associated specialization maps $\sp_{\fZ_{ij}}$ are surjective onto closed points; since the closed points lie very dense, we thus obtain an inclusion $\bigcup_{i\in I_z}\fZ_{ir'}\subseteq\bigcup_{j\in\children(r')}\fX_j$. Hence, $(\fX_j)_{j\in\children(r')}$ together with $(\fX_{r'}\setminus\fZ_{ir})_{i\in I_z}$ covers $\fX_{r'}$, and it follows that the corresponding uniformly rigid generic fibers cover $X$. If $x\in(\fX_{r'}\setminus\fZ_{ir})^\srig$ for some $i\in I_z$, there is nothing more to show. If this is not the case, then $x\in X_j$ for some $j\in\children(r')$, and we may replace $J'$ by $\subtree(j)$. By induction on the volume of $J'$, the statement follows.
\end{proof}

Let us recall from \cite{urigspaces} Cor.\ 2.37 ($ii$) that finite unions of retrocompact semi-affinoid subdomains in semi-affinoid $K$-spaces are admissible. We can now prove the uniformly rigid analog of \cite{kisinlocconst} Lemma 2.3:

\begin{cor}\label{tubecor}
Let $X=\sSp A$ be a semi-affinoid $K$-space, and let $Z\subseteq X$ be a closed uniformly rigid subspace which is defined by functions $f_1,\ldots,f_n\in A$ and which is contained in a finite union $U$ of retrocompact semi-affinoid subdomains in $X$. There exists some $\varepsilon\in\sqrt{|K^*|}$ such that the tube\index{tube} 
\[
X_{\leq\varepsilon}\,\mathrel{\mathop:}=\,X(\varepsilon^{-1}f_1,\ldots,\varepsilon^{-1}f_n)\,\mathrel{\mathop:}=\,\{x\in X\,;\,|f_i(x)|\leq\varepsilon\,\,\forall\, i=1,\ldots,n\}
\]
is contained in $U$.
\end{cor}
\begin{proof}
Let $X_f\subseteq X$ denote the complement of $Z$ in $X$; then Proposition \ref{zarretromixadmprop} shows that $(X_f,U)$ is an admissible covering of $X$. By \cite{urigspaces} Prop.\ 2.40 and the subsequent remark, $(X_{\geq\delta})_\delta$ is an admissible covering of $X_f$, where $\delta$ varies in $\sqrt{|K^*|}$ and where 
\[
X_{\geq\delta}\,=\,\bigcup_{i=1}^n\,\{x\in X\,;\,|f_i(x)|\geq\delta\}\;.
\]
By transitivity of admissibility for coverings, it follows that $(X_{\geq\delta})_\delta\cup U$ is an admissible covering of $X$. Since $X$ is quasi-compact, this covering admits a finite subcovering, which means that $X=U\cup X_{\geq\delta}$ for $\delta$ small enough. If we choose $\varepsilon<\delta$, then $X_{\leq\varepsilon}$ and $X_{\geq\delta}$ are disjoint, and it follows that $X_{\leq\varepsilon}\subseteq U$, as desired.
\end{proof}

The following lemma shows that Corollary \ref{tubecor} can be applied in interesting situations:

\begin{lem}\label{tubeauxlem}
Let $Y$ be a semi-affinoid $K$-space, let $\fX$ be a flat formal $R$-scheme of locally tf type with uniformly rigid generic fiber $X$, let $\fU\subseteq\fX$ be an affine open subscheme, and let $\phi\colon Y\rightarrow X$ be a uniformly rigid morphism. Then $\phi^{-1}(\fU^\srig)$ is a finite union of retrocompact semi-affinoid subdomains in $Y$.
\end{lem}
\begin{proof}
Let $(\fX_i)_{i\in I}$ be an affine open covering of $\fX$, and let $(X_i)_{i\in I}$ be the induced admissible semi-affinoid covering of $X$ (cf.\ \cite{urigspaces} Def.\ 2.46 ($i$). Let $(Y_j)_{j\in J}$ be a retrocompact refinement of $(\phi^{-1}(X_i))_{i\in I}$ (cf.\ \cite{urigspaces} 2.36 ($ii$)), and let $\psi\colon J\rightarrow I$ be a corresponding refinement map. Let us write $U\mathrel{\mathop:}=\fU^\srig$, and for each $i\in I$, let us write $U_i\mathrel{\mathop:}=U\cap X_i$. Then
\[
\phi^{-1}(U)\,=\,\bigcup_{j\in J}(\phi|_{Y_j})^{-1}(U_{\psi(j)})\;.
\]
Since $Y_j$ is retrocompact in $Y$ for all $j$, it suffices to show that for any $j\in J$, the preimage $(\phi|_{Y_j})^{-1}(U_{\psi(j)})$ is a finite union of retrocompact semi-affinoid subdomains in $Y_j$. Hence we may thus assume that $\fX$ is affine, which implies that $X$ is semi-affinoid and that $U$ is a retrocompact semi-affinoid subdomain in $X$. By \cite{urigspaces} Cor.\ 2.25 ($i$) and the remark following \cite{urigspaces} Cor.\ 2.29, $\phi^{-1}(U)$ is weakly retrocompact in $Y$, as desired.
\end{proof}

\section{Algebraizations and $R$-Envelopes}\label{envsection}

As we have pointed out in the introduction, smooth formal $R$-schemes of locally ff type $\fX$ lack points with values in finite unramified extensions of $R$. If we replace the Jacobson-adic topology on $R[[S]]$, say, by the $\pi$-adic topology, then new unramified points appear: for instance, if $R[[S]]^\pi$ denotes the ring $R[[S]]$ with its $\pi$-adic topology, then the stalk of the formal $R$-scheme $\Spf R[[S]]^\pi$ in the generic point of its special fiber $\Spec k[[S]]$ is a discrete valuation ring which is a formally unramified local extension of $R$; it corresponds to the Gauss point in $M(R[[S]]\otimes_RK)$, \cite{urigspaces} Section 4. This observation motivates the following definition:

\begin{defi}\label{envdefi}\index{envelope!proper}
Let $\fS$ be an affine formal $R$-scheme of ff type, let $A$ denote its topological ring of global sections, and let $\fX$ be a proper formal $\fS$-scheme. We write $A^\pi$ and $A^\textup{discrete}$ to denote the ring $A$ equipped with its $\pi$-adic and its discrete topology respectively.
\begin{enumerate}
\item The algebraization of $\fS$ is the completion morphism $\fS\rightarrow\fS_\alg$ of affine noetherian formal schemes which, on the level of topological rings of global sections, corresponds to the completion homomorphism $A^\textup{discrete}\rightarrow A$.
\item The $R$-envelope of $\fS$ is the completion morphism $\fS\rightarrow\fS_\pi$ of affine noetherian formal schemes which, on the level of topological rings of global sections, corresponds to the completion homomorphism $A^\pi\rightarrow A$. 
\item An algebraization of $\fX\rightarrow\fS$ is a proper $\fS_\alg$-scheme $\fX_\alg$ such that $\fX\rightarrow\fS$ is obtained from $\fX_\alg\rightarrow\fS_\alg$ via formal completion along a subscheme of definition of $\fS$.
\item An $R$-envelope of $\fX\rightarrow\fS$ is a proper formal $\fS_\pi$-scheme $\fX_\pi$ such that $\fX\rightarrow\fS$ is obtained from $\fX_\pi\rightarrow\fS_\pi$ via formal completion along a subscheme of definition of $\fS$.
\end{enumerate}
\end{defi}

\begin{remark}\label{envremarks}
Let us make the following remarks:
\begin{enumerate}
\item By \cite{egain} 0.7.2.4, the ring $A$ is complete and separated in its $\pi$-adic topology; hence $\Spf A^\pi$ is well-defined.
\item There is a natural factorization $\fS\rightarrow\fS_\pi\rightarrow\fS_\alg$ of completion morphisms, and every algebraization $\fX_\alg\rightarrow\fS_\alg$ induces an $R$-envelope via $\pi$-adic formal completion, that is, via pullback under $\fS_\pi\rightarrow\fS_\alg$.
\item The algebraization and the $R$-envelope of $\fX\rightarrow\fS$ are unique up to unique isomorphism if they exist. This follows from Grothendieck's Formal Existence Theorem (cf.\ \cite{egaiii} 5.4.1) which, in the case of $R$-envelopes, applies after reducing modulo powers of $\pi$.
\item Blowups algebraize to blowups, and admissible blowups algebraize to admissible blowups.
\item The concept of	 $R$-envelopes of affine formal $R$-schemes of ff type has already appeared in \cite{strauch_deformation} and \cite{huber_finiteness_ii}.
\item The formation of algebraizations of $R$-envelopes does not commute with localization or with the formation of fibered products; moreover, fibered products of algebraizations or $R$-envelopes need not be locally noetherian. 
\item Special fibers of $R$-envelopes are $k$-schemes, and generic fibers of $R$-envelopes exist within Huber's category of adic spaces. The complement of the adic generic fiber of $\fS$ within the adic generic fiber of $\fS_\pi$ will be called the boundary of $\fS^\rig$.
\end{enumerate}
\end{remark}

Letf $\fY$ be an affine smooth formal $R$-scheme of ff type, let $\fX$ be a separated formal $R$-scheme of locally tf type, let $\phi\colon Y\rightarrow X$ be a morphism of uniformly rigid generic fibers, and let $\ul{\Gamma}_\phi$ denote the schematic closure of the graph of $\phi$ in $\fY\times\fX$. In Section \ref{graphpropersec}, we have seen, under more general assumptions, that the projection $p_\fY|_{\ul{\Gamma}_\phi}:\ul{\Gamma}_\phi\rightarrow\fY$ is proper. In this section, we will prove that
\begin{enumerate}
\item the $R$-envelope $(p_\fY|_{\ul{\Gamma}_\phi})_\pi:\ul{\Gamma}_\pi\rightarrow\fY_\pi$ of $p_\fY|_{\ul{\Gamma}_\phi}$ exists, and
\item this $R$-envelope is an isomorphism above the generic points of the special fiber $\fY_{\pi,k}$. 
\end{enumerate}

A corollary of this result may colloquially be stated as follows: for $\fX$ and $\fY$ as above, if $\phi$ is a morphism of rigid generic fibers respecting the natural uniform structures, then, on the level of adic spaces, $\phi$ extends to certain formally unramified boundary points of the generic fiber of $\fY$.

\subsection{Generic fibers of $R$-envelopes}\label{genfibenvsec}

In the proof of Theorem \ref{mainweilextthm}, we will need to use generic fibers\index{generic fiber!of an envelope} of $R$-envelope, which exist in Huber's category of adic spaces. For an introduction to the theory of adic spaces, we refer to \cite{huber_adicsppaper} and \cite{huberbook} Chapter 1. The adic space $t(\fX)$ associated to a locally noetherian formal scheme $\fX$ is defined in \cite{huber_adicsppaper} 4.1. By \cite{huber_adicsppaper}, the functor  $t:\fX\mapsto t(\fX)$ is fully faithful and commutes with open immersions. If $\fX=\Spf A$ is affine, then $t(\fX)\mathrel{\mathop:}=\Spa(A,A)$ is affinoid. By \cite{huberbook} 1.2.2, the fibered product 
\[
t(\fX)\otimes_RK\,\mathrel{\mathop:}=\,t(\fX)\times_{\Spa(R,R)}\Spa(K,R)
\]
exists for any locally noetherian formal $R$-scheme $\fX$. Moreover, it is seen from the explicit construction in the proof of \cite{huberbook} 1.2.2 that $t(\fX)\otimes_RK$ is affinoid whenever $\fX$ is affine and adic over $R$. Let us point out that $t(\fX)\otimes_RK$ needs not even be quasi-compact when $\fX$ is not adic over $R$. Indeed, if $\fX$ is of locally ff type over $R$, it is seen from the construction in the proof of \cite{huberbook} 1.2.2 that $t(\fX)\otimes_RK$ is the adic space $(\fX^\rig)^\ad$ associated to the Berthelot generic fiber  of $\fX$. Here $\ad$ denotes the natural fully faithful functor from the category of rigid $K$-spaces to the category of adic $\Spa(K,R)$-spaces. We refer to \cite{huber_adicsppaper} 4.3 for the definition of $\ad$.

If $\fX$ is a locally noetherian formal $R$-scheme, not necessarily of locally ff type, the adic space $t(\fX)\otimes_RK$ is called the generic fiber of $\fX$. We will briefly write
\[
\fX_K\,\mathrel{\mathop:}=\,t(\fX)\otimes_RK\quad.
\]


\subsection{Existence of $R$-envelopes of graphs}\label{graphenvexsec}

Let $\fX$ be a separated formal $R$-scheme of locally tf type, let $\fY$ be an affine formal $R$-scheme of ff type, let $X$ and $Y$ denote the uniformly rigid generic fibers of $\fX$ and $\fY$, and let $\phi:Y\rightarrow X$ be a morphism. By Theorem \ref{graphproperthm}, which holds under more general assumptions, the projection $p_\fY|_{\ul{\Gamma}_\phi}$ from the schematic closure $\ul{\Gamma}_\phi$ of the graph of $\phi$ in $\fY\times\fX$ to $\fY$ is proper. 

\begin{prop}\label{graphenvexprop}
There exists a unique closed formal subscheme 
\[
(\ul{\Gamma}_\phi)_\pi\,\subseteq\,\fY_\pi\times\fX
\]
such that  the restriction $p_{\fY_\pi}|_{(\ul{\Gamma}_\phi)_\pi}$ of the projection $p_{\fY_\pi}$ is the $R$-envelope of $p_\fY|_{\ul{\Gamma}_\phi}$.
\end{prop}
\begin{proof}
For $n\in\N$, we let a subscript $n$ indicate reduction modulo $\pi^{n+1}$. The affine formal scheme $\fY_n$ is the completion of the affine scheme $\fY_{\pi,n}$ along a subscheme of definition of $\fY$, and the corresponding completion morphism corresponds to an isomorphism of non-topological rings of global sections. The closed formal subscheme $\ul{\Gamma}_\phi\subseteq\fY\times\fX$ induces closed formal subschemes $(\ul{\Gamma}_\phi)_n\subseteq\fY_n\times\fX_n$, and $\fY_n\times\fX_n$ is the completion of $\fY_{\pi,n}\times\fX_n$. By Corollary \ref{graphpropercor}, $(\ul{\Gamma}_\phi)_n$ is proper over $\fY_n$; hence by \cite{egaiii} 5.1.8, $(\ul{\Gamma}_\phi)_n$ algebraizes to a uniquely determined closed subscheme 
\[
(\ul{\Gamma}_\phi)_{\pi,n}\subseteq\fY_{\pi,n}\times\fX_n\;.
\]
Uniqueness imlies that $(\ul{\Gamma}_\phi)_{\pi,n}=(\ul{\Gamma}_\phi)_{\pi,n+1}\cap(\fY_{\pi,n}\times\fX_n)$. By \cite{egain} 10.6.3, 
\[
(\ul{\Gamma}_\phi)_\pi\,\mathrel{\mathop:}=\,\varinjlim_n\,(\ul{\Gamma}_\phi)_{\pi,n}
\]
is a formal scheme, and one immediately verifies that it is a closed formal subscheme of $\fY_\pi\times\fX$ which is proper over $\fY_\pi$ via $p_{\fY_\pi}$ and whose completion is $\ul{\Gamma}_\phi$.
\end{proof}

\subsection{Open coverings induce fpqc-coverings of algebraizations}

In Remark \ref{envremarks} ($v$), we have mentioned the fact that the formation of algebraizations does not commute with localization: if $\fX$ is an affine formal $R$-scheme of ff type and if $\fU\subseteq\fX$ is an affine open formal subscheme, then $\fU_\alg$ needs not be an open formal subscheme of $\fX_\alg$. However, if $(\fU_i)_{i\in I}$ is an affine open covering of $\fX$, then induced family of morphisms $(\fU_{i,\alg}\rightarrow\fX_\alg)_{i\in I}$ is faithfully flat; that is, an open covering of $\fX$ induces an fpqc-covering (cf.\ \cite{sga3} I Exp.\ IV 6.3.1) of $\fX_\alg$. We prove an analog statement in the proper case:
 
\begin{prop}\label{thefaithflatprop}
Let $\fX$ be an affine formal $R$-scheme of ff type, and let $\phi:\fX'\rightarrow\fX$ be a proper morphism whose algebraization $\phi_\alg:\fX'_\alg\rightarrow\fX_\alg$ exists. Let $(U_i)_{i\in I}$ be a finite affine open covering of $\fX'_\alg$, and let $(\fU_i)_{i\in I}$ be the induced affine open covering of $\fX'$; then the canonical family of morphisms of schemes
\[
(\psi_i\colon \fU_{i,\alg}\rightarrow U_i\subseteq \fX'_\alg)_{i\in I}
\]
is faithfully flat.
\end{prop}
\begin{proof}
The $\psi_i$ are flat, adic completions of noetherian rings being flat. Since the $U_i$ cover $\fX'_\alg$, it suffices to show that for any $i\in I$, the restricted family is faithfully flat over $U_i$; hence we must show that for every $i\in I$ and every closed point $y\in U_i$, there exists an index $j\in I$ such that $y\in\im\psi_j$. We must beware of the fact that the closed point $y\in U_i$ needs not be closed in $\fX_\alg'$. Let $x\in \fX_\alg$ denote the $\phi_\alg$-image of $y$, let $A$ be the ring of functions on $\fX$, let $\a$ be an ideal of definition for $A$, and let $\p_x\subseteq A$ be the prime ideal corresponding to $x$. If $x$ is a closed point of $\fX_\alg$, then $x\in V(\a)$, and hence $y\in\phi_\alg^{-1}(V(\a))$ which is a subscheme of definition of $\fX'$. Since $\fX'$ is covered by the $\fU_j$, there exists an index $j\in J$ such that $y\in\fU_j$, and we obtain $y\in\im\psi_j$. Let us now consider the case where $x\in \fX_\alg$ is not closed.

Let $B_i$ denote the ring of functions on $U_i$, and let $\m_y\subseteq B_i$ be the maximal ideal corresponding to $y$. Let $L_y$ and $L_x$ denote the residue fields of $B_i$ and $A$ in $y$ and $x$ respectively. Since $L_y=B_i/\m_y$ is of finite type over $A/\p_x$, it is of finite type over $L_x=Q(A/\p_x)$ and, hence, finite over $L_x$. By \cite{artin_tate} Theorem 1, it follows that $L_x$ is of finite type over $A/\p_x$. By \cite{artin_tate} Theorem 4, we conclude that $A/\p_x$ has only a finite number of prime ideals and that each nonzero prime ideal of $A/\p_x$ is maximal. Since $\p_x$ is not maximal in $A$ by our assumption on $x$, $A/\p_x$ has Krull dimension $1$. The maximal ideals of $A/\p_x$ correspond to the maximal ideals of $A$ which contain $\p_x$; hence they are open in the separated and complete $\a(A/\p_x)$-adic quotient topology on $A/\p_x$. By \cite{urigspaces} Lemma 2.1, this topology coincides with the Jacobson-adic topology. Since $A/\p_x$ is integral, we conclude from \cite{bourbakica} III.2.13 Corollary of Proposition 19 that the one-dimensional domain $A/\p_x$ is local and complete in its maximal-adic topology.

Let $R'$ denote the integral closure of $A/\p_x$ in $L_x$. Since $A$ is excellent, $R'$ is finite over $A/\p_x$, and it follows that $R'$ is a complete discrete valuation ring whose valuation topology coincides with the $\a$-adic topology. Since $R'$ is finite over $A$, the rings of functions on the affine $R'$-schemes $\fU_{i,\alg}\otimes_AR'$ are $\a$-adically complete. Hence, we may perform the base change $\cdot\otimes_AR'$ and thereby assume that $A=R'$ is a complete discrete valuation ring such that $x\in\Spec R'$ is the generic point. Let $K'$ denote the field of fractions of $R'$. We are now in a classical rigid-analytic situation over $R'$.

 Since the generic fiber $\fX'_{\alg,K'}$ of $\fX'_{\alg}$ is a $K'$-scheme of finite type, $y\in \fX'_{\alg,K'}$ is a closed point; moreover, the closed points of $\fX'_{\alg,K'}$ correspond to the points of its rigid analytification $(\fX'_{\alg,K'})^\an$. Since $\fX'_\alg$ is proper, $(\fX'_{\alg,K'})^\an$ coincides with the rigid generic fiber of the completion $\fX'$ of $\fX'_\alg$, and hence $(\fX'_{\alg,K'})^\an$ is covered by the maximal spectra of the rings of functions on the generic fibers of the $R'$-schemes $\fU_{j,\alg}$, with $j$ varying in $I$. It follows that there exist a $j\in I$ and a point in $\fU_{j,\alg}$ mapping to $y$ under $\psi_j$, as desired.
\end{proof}

\begin{remark}\label{thefaithflatrem}
If, in the situation of Proposition \ref{thefaithflatprop}, $(\fU_i)_{i\in I}$ is any finite affine open covering of $\fX'$, then there always exists a finite affine open refinement $(\fU'_j)_{j\in J}$ of $(\fU_i)_{i\in I}$ such that $(\fU'_j)_{j\in J}$ is induced by a finite affine open covering of $\fX'_\alg$ via formal completion. 
\end{remark}

\subsection{$R$-regular points in algebraizations}

\begin{defi}
A point in a locally noetherian $R$-scheme is called $R$-regular if it admits an affine open neighborhood whose ring of functions is $R$-regular in the sense of \cite{egaiv} 6.8.1.
\end{defi}

\begin{lem}\label{genisolem}
Let $X$ be a flat locally noetherian $R$-scheme, let $\phi\colon X'\rightarrow X$ be an admissible blowup, and let $\eta\in X_k$ be a generic point. On any quasi-compact open neighborhood of $\eta$ in $X$ whose special fiber is integral, $\phi$ is defined by a $\pi$-adically open coherent ideal whose vanishing locus does not contain $\eta$. In particular, if $X_k$ is reduced in $\eta$, then $\phi$ is an isomorphism around an open neighborhood of $\eta$ in $X$.
\end{lem}
\begin{proof}
We may assume that $X$ is quasi-compact and that $X_k$ is integral. Let $\sI\subseteq\O_X$ be a $\pi$-adically open ideal defining $\phi$. If $\sI\subseteq\pi \O_\fX$, we may uniquely divide $\sI$ by $\pi$ without changing $\phi$ and without changing the fact that $\sI$ is $\pi$-adically open. If the resulting coherent ideal still lies in $\pi \O_X$, we repeat the procedure. This process terminates after finitely many steps. Indeed, if it did not terminate, then $\pi$ would be a unit in $\O_X$, $\sI$ being $\pi$-adically open. However, $X_k$ contains $\eta$ and, hence, is nonempty. We may therefore assume that $\sI$ is not contained in $\pi \O_X$. Since $X_k$ is integral, this means precisely that $\eta\notin V(\sI)$. Since $\phi$ is an isomorphism on the open complement of $V(\sI)$, the claim has been shown.
\end{proof}

\begin{prop}\label{liftrregpointsprop}
Let $\fX$ be an affine flat formal $R$-scheme of ff type, let $\phi:\fX'\rightarrow\fX$ be an admissible blowup, and let $\phi_\alg:\fX'_\alg\rightarrow\fX_\alg$ denote its algebraization. Let $(U_i)_{i\in I}$ be a finite affine open covering of $\fX'_\alg$, and let $(\fU_i)_{i\in I}$ be the finite affine open covering of $\fX'$ which is induced via formal completion. Then for every generic point $\eta\in \fX_{\alg,k}$ that is $R$-regular in $\fX_\alg$, there exist 
an index $i\in I$ and a generic point $\eta'\in \fU_{i,\alg,k}$ above $\eta$ that is $R$-regular in $\fU_{i,\alg}$.
\end{prop}
\begin{proof}
Since $\fX_{\alg,k}$ is regular and, hence, reduced in $\eta$, Lemma \ref{genisolem} shows that there exists an open subscheme $U\subseteq \fX'_\alg$ that maps, via $\phi_\alg$, isomorphically onto an $R$-regular open neighborhood of $\eta$. By Proposition \ref{thefaithflatprop}, the family of morphisms
\[
(\psi_i\colon \fU_{i,\alg}\rightarrow U_i\subseteq \fX'_\alg)_{i\in I}
\]
is faithfully flat. Since the schemes $U_i$ are of finite type over $X$, they are excellent; hence by \cite{egaiv} 7.8.3 (v), the morphisms $\psi_i$ are regular. For each $i\in I$, we let $V_i\subseteq \fU_{i,\alg}$ denote the $\psi_i$-preimage of $U$; then the disjoint union of the $V_i$ is faithfully flat over $U$, and by \cite{egaiv} 6.8.3 ($i$) it is $R$-regular. Let $i\in I$ be an index such that $\eta$ lies in the image of $V_i$. Let $\xi\in V_{i,k}$ be a point above $\eta$, and let $\eta'$ be a maximal generalization of $\xi$ in $V_{i,k}$. Then $\eta'$ is a generic point of $\fU_{i,\alg,k}$ that is $R$-regular in $\fU_{i,\alg}$ and that lies above $\eta$.
\end{proof}

\begin{remark}\label{normdvrrem}
In the situation of the Proposition \ref{liftrregpointsprop}, let us assume that both $\fX$ and $\fX'$ are normal. Then by \cite{egaiv} 7.8.3 (v) the schemes $\fU_{i,\alg}$ are normal as well, and it follows that the completed stalks of $\fU_{i,\alg}$ and $\fX_\alg$ in $\eta'$ and $\eta$ respectively are complete discrete valuation rings. Of course, the completed stalks of the algebraizations in $\eta$ and $\eta'$ coincide with the completed stalks of the associated $R$-envelopes in these points.
\end{remark}

\subsection{Extensions of morphisms to $R$-envelopes}

In Section \ref{graphenvexsec}, we considered a situation where $\fY$ is an affine flat formal $R$-scheme of ff type with uniformly rigid generic fiber $Y$, where $\fX$ is a separated flat formal $R$-scheme of locally tf type with uniformly rigid generic fiber $X$ and where $\phi\colon Y\rightarrow X$ is a morphism of uniformly rigid spaces. In Theorem \ref{graphproperthm} and Proposition \ref{graphenvexprop}, we showed that  the projection $p_\fY|_{\ul{\Gamma}_\phi}$ from the schematic closure $\ul{\Gamma}_\phi$ of the graph of $\phi$ in $\fY\times\fX$ to $\fY$ is proper and that its $R$-envelope $(p_\fY|_{\ul{\Gamma}_\phi})_\pi$ exists as a closed subscheme of $\fY_\pi\times\fX$. We now show that if $\fY$ is smooth, then $(p_\fY|_{\ul{\Gamma}_\phi})_\pi$ is an isomorphism above the generic points of $\fY_\pi$.

\begin{thm}\label{complfiberisothm}
Let us assume that $\fY^\rig$ is normal, and let $\eta$ be a generic point in $\fY_{\pi,k}$ that is $R$-regular in the spectrum of the ring of functions on $\fY$. Then $(p_\fY|_{\ul{\Gamma}_\phi})_\pi$ induces an isomorphism under the base change 
\[
\cdot\times_{\fY_\pi}\Spf\hat{\O}_{\fY_\pi,\eta}\;.
\]
\end{thm}
\begin{proof}
If $\phi$ extends to a morphism from $\fY$ to $\fX$, then $p_\fY|_{\ul{\Gamma}_\phi}$ is an isomorphism, and there is nothing to show. Let $(Y_i)_{i\in I}$ be a treelike covering of $Y$ together with a model $(\fY_i)_{i\in I}$ in the sense of the discussion following \cite{urigspaces} Def.\ 2.31 such that $\fY_r=\fY$ for the root $r\in I$ and such that for each leaf $i\in I$, the restriction of $\phi$ to $Y_i$ extends to a morphism from $\fY_i$ to $\fX$; such a covering exists by \cite{urigspaces} Cor.\ 2.14 ($iv$). Since $\fY$ is generically normal, Corollary \ref{normisblowupcor} shows that we may assume all $\fY_i$, $i\in I$, to be normal. After possibly passing to a refinement, by the argument in Remark \ref{thefaithflatrem} we may assume that all affine open coverings in this treelike formal covering are induced by affine open coverings on the respective admissible algebraic blowups. By Proposition \ref{liftrregpointsprop}, there exists a child $i$ of the root of $I$ together with a point $\eta'$ in the spectrum of the ring of functions on $\fY_i$ such that $\eta'$ is $R$-regular in this spectrum, such that $\eta'$ is generic in the special fiber $\fY_{i,\pi,k}$ and such that $\eta'$ lies above $\eta$. Let $R_{\eta'}$ and $R_\eta$ denote the completed stalks of $\fY_{i,\pi}$ and $\fY_\pi$ in $\eta'$ and $\eta$ respectively; according to Remark \ref{normdvrrem}, $R_{\eta'}/R_\eta$ is a local extension of complete discrete valuation rings. Let $\ul{\Gamma}_{\phi,i}$ denote the schematic closure of the graph of $\phi|_{Y_i}$ in $\fY_i\times\fX$, and let $(\ul{\Gamma}_{\phi,i})_\pi\subseteq\fY_{i,\pi}\times\fX$ denote its $R$-envelope.
Using induction on the volume of $I$, we may assume that $p_{\fY_{i,\pi}}|_{(\ul{\Gamma}_{\phi,i})_\pi}$ becomes an isomorphism under the base change
\[
\cdot\times_{\fY_{i,\pi}}\Spf R_{\eta'}\;.
\]
The pullback $\ul{\Gamma}_\phi\times_\fY\fY_i$ is a closed formal subscheme of $\fY_i\times\fX$ which is proper over $\fY_i$, and its $R$-envelope $(\ul{\Gamma}_\phi\times_\fY\fY_i)_\pi=(\ul{\Gamma}_\phi)_\pi\times_{\fY_\pi}\fY_{i,\pi}\rightarrow\fY_{i,\pi}$ exists. Since $\ul{\Gamma}_\phi\times_\fY\fY_i$ is a closed formal subscheme of $\fY_i\times\fX$ whose generic fiber is the graph of $\phi|_{Y_i}$,
\[
\ul{\Gamma}_{\phi,i}\,\subseteq\,\ul{\Gamma}_\phi\times_\fY\fY_i
\]
is the closed formal subscheme defined by the $\pi$-torsion ideal. Since completions of locally noetherian schemes are flat and since $R$-envelopes are unique, it follows that the closed formal subscheme in $(\ul{\Gamma}_\phi)_\pi\times_{\fY_\pi}\fY_{i,\pi}$ defined by the $\pi$-torsion ideal is the $R$-envelope $(\ul{\Gamma}_{\phi,i})_\pi$ of $\ul{\Gamma}_{\phi,i}$ over $\fY_{i,\pi}$. 
We have to show that the morphism $(\ul{\Gamma}_\phi)_\pi\times_{\fY_\pi}\Spf R_\eta\rightarrow \Spf R_\eta$ is an isomorphism. Since it is of tf type and since $R_{\eta'}$ is faithfully flat over $R_\eta$, descent theory reduces the problem to proving that
\[
(\ul{\Gamma}_\phi)_\pi\times_{\fY_\pi}\Spf R_{\eta'}\rightarrow\Spf R_{\eta'}
\]
is an isomorphism. Since $\Spf R_{\eta'}$ is flat over $\Spf R_\eta$, it is flat over $\fY_\pi$, and it follows that the domain of the above morphism is flat over $R$. Moreover, since $(\ul{\Gamma}_{\phi,i})_\pi\subseteq(\ul{\Gamma}_\phi)_\pi\times_{\fY_\pi}\fY_{i,\pi}$ is defined by the $\pi$-torsion ideal and since $R_{\eta'}$ is flat over $\fY_{i,\pi}$, the closed formal subscheme
\begin{eqnarray*}
(\ul{\Gamma}_{\phi,i})_\pi\times_{\fY_{i,\pi}}\Spf R_{\eta'}&\subseteq&((\ul{\Gamma}_\phi)_\pi\times_{\fY_\pi}\fY_{i,\pi})\times_{\fY_{i,\pi}}\Spf R_{\eta'}\\
&\cong&(\ul{\Gamma}_\phi)_\pi\times_{\fY_\pi}\Spf R_{\eta'}
\end{eqnarray*}
is defined by $\pi$-torsion as well. We conclude that the above closed immersion is in fact an isomorphism.  By our induction hypothesis, the projection to $\Spf R_{\eta'}$ is an isomorphism, as desired.
\end{proof}

In the above proof, the $\pi$-torsion in $(\ul{\Gamma}_\phi)_\pi\times_{\fY_\pi}\fY_{i,\pi}$ is due to the fact that admissible blowups are in general not flat. We have used the fact that under the above assumptions, they are flat locally on their $R$-envelopes.

We can strengthen the statement of Theorem \ref{complfiberisothm} as follows:

\begin{cor}\label{genisocor}
In the situation of Theorem \ref{complfiberisothm}, the morphism $(p_\fY|_{\ul{\Gamma}_\phi})_\pi$ is an isomorphism over an open neighborhood of $\eta$.
\end{cor}
\begin{proof}
Let us abbreviate $\phi\mathrel{\mathop:}=(p_\fY|_{\ul{\Gamma}_\phi})_\pi$; we have to show that there exists an open neighborhood $\fU$ of $\eta$ in $\fY_\pi$ such that the restriction $\phi^{-1}(\fU)\rightarrow \fU$ of $\phi$ is an isomorphism. Let $k(\eta)$ denote the residue field of $\fY_\pi$ in $\eta$. By Theorem \ref{complfiberisothm}, we know that the special fiber $\phi_k$ of $\phi$ becomes an isomorphism under the base change $\cdot\times_{\fY_{\pi,k}}\Spec k(\eta)$. Since $\phi_k$ is proper, it follows from \cite{egaiv} 18.12.7 that there exists an open neighborhood $U_k$ of $\eta$ in $\fY_{\pi,k}$ such that the restriction $\phi_k^{-1}(U_k)\rightarrow U_k$ of $\phi_k$ is a closed immersion. Let $\fU\subseteq\fY_\pi$ be the open formal subscheme whose special fiber is $U_k$. By \cite{egaiii} 4.8.1, the restriction $\phi^{-1}(\fU)\rightarrow \fU$ of $\phi$ is a finite morphism. By \cite{eisenbudca} Ex. 7.2, a finite homomorphism of $\pi$-adically complete and separated $R$-algebras that is surjective modulo $\pi$ is surjective; hence $\phi^{-1}(\fU)\rightarrow \fU$ is a closed immersion as well. We may shrink $\fU$ and thereby assume that $\fU$ as affine, $\fU=\Spf A$ for some $\pi$-adically complete $R$-algebra $A$. Then the above restriction of $\phi$ corresponds to an ideal $I$ in $A$. Let $\p\subseteq A$ be the prime ideal corresponding to $\eta\in U_k$. By Theorem \ref{complfiberisothm} and using faithfully flat descent, we see that $A_\p/IA_\p=A_\p$, that is, $I A_\p=0$. Since $A$ is noetherian, $I$ is finitely generated. After shrinking $\fU$ further, we may thus assume that the restriction of $\phi$ is an isomorphism, as desired.
\end{proof}

\subsection{$R$-envelopes of Hartogs figures and analytic continuation}\label{ancontsec}

In the proof of our main Theorem \ref{mainweilextthm}, we will need an ff type analog of analytic continuation theory for Hartogs figures, cf.\ \cite{luetkebohmert_mer} Satz 6: let $\fX=\Spf R[[S_1,\ldots,S_n,Z]]$ denote the $n+1$-dimensional open formal unit disc over $R$, and let $f$ be a $Z$-distinguished global function on $\fX$, that is, a formal power series whose reduction in $k[[Z]]$ is a nonzero non-unit. While the non-vanishing locus $\fX_f$ of $f$ on $\fX$ is empty, the non-vanishing locus $\fX_{\pi,f}$ of $f$ on the $R$-envelope of $\fX$ is nonempty. Let $r$ be a natural number, and let $\fX'$ denote the affine part of the admissible blowup of $\fX$ in the ideal $(\pi^{r+1},S_1,\ldots,S_n)$ where the pullback ideal is generated by $\pi^{r+1}$. The natural morphism of affine flat formal $R$-schemes of ff type $\fX'\rightarrow\fX$ is induced by a unique morphism of $R$-envelopes $(\fX')_\pi\rightarrow\fX_\pi$ which restricts to a morphism $(\fX')_{\pi,f}\rightarrow\fX_{\pi,f}$. The resulting cartesian diagram
\[
\begin{diagram}
(\fX')_{\pi,f}&\rInto&(\fX')_\pi\\
\dTo&&\dTo\\
\fX_{\pi,f}&\rInto&\fX_\pi
\end{diagram}
\]
may be considered as an analog of a Hartogs figure; cf.\ \cite{luetkebohmert_mer} Section 3 for the notion of Hartogs figures in rigid geometry. The following statement is the aforementioned analog of \cite{luetkebohmert_mer} Satz 6:

\begin{thm}\label{ancontthm}
Let $g$ and $h$ be global functions on $\fX_{\pi,f}$ and on $(\fX')_\pi$ respectively whose pullbacks to $(\fX')_{\pi,f}$ coincide. Then $g$ and $h$ are the pullbacks of a unique global function on $\fX_\pi$.
\end{thm}

We first establish an elementary statement on distinguished power series: let $A$ be a complete noetherian local ring with residue field $\kappa$. A formal power series $f\in A[[Z]]$ is called distinguished if its reduction in $\kappa[[Z]]$ is a nonzero non-unit. 

\begin{lem}\label{distfinlem}
If $f\in A[[Z]]$ is a distinguished power series, then the continuous $A$-endomorphism 
\[
\phi\colon A[[Z]]\rightarrow A[[Z]]\quad,\quad Z\mapsto f
\]
is finite and free.
\end{lem}
\begin{proof}
We first show that $\phi$ is finite. Let $\m$ denote the maximal ideal of $A$. Since $\phi$ is $A$-linear, $\phi(\m)A[[Z]]=\m A[[Z]]$, and it follows that $(\m,f)$ is an ideal of definition of $A[[Z]]$. Since $A[[Z]]$ is $(\m,Z)$-adically separated and complete, it follows from \cite{eisenbudca} Ex. 7.2 that $\phi$ is finite if and only if $A/(\m,f)$ is finite over the residue field $\kappa=A/(\m,Z)$ of $A$. By the Weierstrass Division Theorem for formal power series in $\kappa[[Z]]$, $A/(\m,f)$ is finite free over $\kappa$ of rank $n$, as desired, where $n$ is the $Z$-adic valuation of the reduction of $f$ modulo $\m$. We have thus shown that $\phi$ is finite.

To show that $\phi$ is free, it suffices to prove that $\phi$ is flat, the noetherian ring $A[[Z]]$ being local. Let $\tilde{\phi}$ denote the reduction of $\phi$ modulo $\m$. We easily see that $\tilde{\phi}$ is injective: if $\ker\tilde{\phi}$ was nontrivial, then $\ker\tilde{\phi}=Z^nk[[Z]]$for some $n\in\N_{\geq 1}$, which would imply that $\tilde{\phi}(Z)^n=0$ in $k[[Z]]$ and, hence, $f\equiv0$ modulo $\m$, contrary to our assumption that $f$ is distinguished. Since $\kappa[[Z]]$ is a principal ideal domain and since $\tilde{\phi}$ is injective, it follows that $\tilde{\phi}$ is flat. By the Flatness Criterion (\cite{bourbakica} Chap. III \S 5.2 Theorem 1 ($i$)$\Leftrightarrow$($iii$)), it suffices to show that $\m A[[Z]]\otimes_{A[[Z]],\phi}A[[Z]]\rightarrow \phi(\m)A[[Z]]$ is bijective, which follows from the fact that $\phi$ is an $A$-homomorphism and that $A[[Z]]$ is $A$-flat.
\end{proof}

Let us now give the \emph{proof} of Theorem \ref{ancontthm}. Let us write $\fY=\fX$. By Lemma \ref{distfinlem}, the $R$-morphism $\phi\colon\fY\rightarrow\fX$ given by $\phi^*S_i=S_i$ for $1\leq i\leq n$ and $\phi^*Z=f$ is finite and free, and it induces a finite free morphism of $R$-envelopes $\phi_\pi:\fY_\pi\rightarrow\fX_\pi$ that restricts to a finite and free morphism $\fY_{\pi,f}\rightarrow\fX_{\pi,Z}$, where $\fX_{\pi,Z}$ denotes the non-vanishing locus of $Z$ on $\fX_\pi$. Let $\fX''$ denote the admissible blowup of $\fX$ in $(\pi^{r+1},S_1,\ldots,S_n)$; since $\phi^*S_i=S_i$, the pullback of $\fX''$ under the flat morphism $\phi$ is the admissible blowup of $\fY$ in the ideal generated by $\pi^{r+1}$ and the $S_i$ on $\fY$. By \cite{nicaise_traceformula} Lemma 2.18, the affine locus $\fY'$ in $\fY''$ where the pullback ideal is generated by $\pi^{r+1}$ is obtained from the corresponding locus $\fX'\subseteq\fX''$ via $\phi$-pullback. Since $\phi$ is finite, base change with respect to $\phi$ commutes with the formation of $R$-envelopes; that is, $\phi_\pi^*(\fX'_\pi)=(\phi^*\fX')_\pi$. It follows that the cartesian diagram of affine noetherian formal $R$-schemes
\[
\begin{diagram}
(\fY')_{\pi,f}&\rInto&(\fY')_\pi\\
\dTo&&\dTo\\
\fY_{\pi,f}&\rInto&\fY_\pi
\end{diagram}
\]
is obtained from the cartesian diagram
\[
\begin{diagram}
(\fX')_{\pi,Z}&\rInto&(\fX')_\pi\\
\dTo&&\dTo\\
\fX_{\pi,Z}&\rInto&\fX_\pi
\end{diagram}
\]
via $\phi_\pi$-pullback. A choice of module basis for the ring of functions on $\fY$ over the ring of functions on $\fX$ thus reduces the problem to the case where $f=Z$.

The ring of functions on $\fX_{\pi,Z}$ is easily identified with
\[
R[[S_1,\ldots,S_n,Z]]^\pi\langle Z^{-1}\rangle\,=\,\{\sum_{i=-\infty}^\infty a_i Z^i\,;\,a_i\in R[[S_1,\ldots,S_n]]\,,\,|a_i|\overset{i\rightarrow-\infty}{\longrightarrow} 0\}
\]
with uniquely determined coefficients $a_i$ where, as in Definition \ref{envdefi}, the superscript $\pi$ indicates that a ring is being equipped with the $\pi$-adic topology, and where $|\cdot|$ denotes the Gauss norm with respect to the absolute value on $R$. The ring of functions on $(\fX')_{\pi,Z}$ is given by
\[
R[[S_1,\ldots,S_n]]\langle\frac{S_1}{\pi^{r+1}},\ldots,\frac{S_n}{\pi^{r+1}}\rangle[[Z]]^\pi\langle Z^{-1}\rangle\quad,
\]
which admits the analogous explicit description. Let us write the function $g$ on $\fX_{\pi,Z}$ uniquely as $g=\sum_{i\in\Z} a_i Z^i$; we must show that $a_i=0$ for $i<0$. Since the natural pullback of $g$ to $(\fX')_{\pi,f}$ coincides with the natural pullback of
\[
h\,\in\,R[[S_1,\ldots,S_n]]\langle\frac{S_1}{\pi^{r+1}},\ldots,\frac{S_n}{\pi^{r+1}}\rangle[[Z]]\quad,
\]
it suffices to note that the natural homomorphism
\[
R[[S_1,\ldots,S_n]]\rightarrow R[[S_1,\ldots,S_n]]\langle \frac{S_1}{\pi^{r+1}},\ldots,\frac{S_n}{\pi^{r+1}}\rangle
\]
is injective, which is clear from the fact that after inverting $\pi$, it is a flat homomorphism of domains. \qed

\section{Models for uniformly rigid morphisms}\label{weilsec}

Let $\fX$ and $\fY$ be flat formal $R$-schemes of locally ff type, and let $\phi\colon Y\rightarrow X$ be a morphism of uniformly rigid generic fibers. In this section, we discuss criteria for the existence of a model $\ul{\phi}:\fY\rightarrow\fX$ of $\phi$ with respect to the given models $\fX$ and $\fY$. If in this situation $\fU\subseteq\fY$ is an open formal subscheme such that the restriction $\phi|_U$ of $\phi$ to the uniformly rigid generic fiber $U$ of $\fU$ extends to a morphism $\ul{\phi}|_\fU:\fU\rightarrow\fX$, we say that $\phi$ extends to $\fU$ and that $\ul{\phi}|_\fU$ extends $\phi$.

\subsection{Affine formal targets containing all points}
If the given model $\fX$ of the target $X$ of $\phi$ is affine, then the situation is rather simple, because morphism to affine formal schemes correspond to continuous homomorphisms of topological rings of global sections:

\begin{lem}\label{affineextlem}
If $\fX$ is affine and if $\fY$ is normal or if $\fY_k$ is reduced, then $\phi$ extends uniquely to a morphism $\ul{\phi}\colon\fY\rightarrow\fX$.
\end{lem}
\begin{proof}
We may work locally on $\fY$ and thereby assume that $\fY$ is affine as well; then both $X$ and $Y$ are semi-affinoid. Let us write $\fX=\Spf \ul{A}$, $\fY=\Spf \ul{B}$, $X=\sSp A$ and $Y=\sSp B$. By \cite{urigspaces} Cor.\ 2.14 ($iv$), $\phi$ extends to a morphism $\ul{\phi}'\colon\fY'\rightarrow\fX$, where $\fY'=\Spf \ul{B}'$ is a finite admissble blowup of $\fY$. We obtain inclusions 
\[
\ul{B}\subseteq\ul{B}'\subseteq \mathring{B}\;,
\]
where $\mathring{B}$ denotes the ring of power-bounded functions in $B$, cf.\ \cite{urigspaces} Def.\ 2.6. It suffices to show that $\ul{B}=\mathring{B}$. If $\fY$ is normal, this equality is given by \cite{urigspaces} Prop.\ 2.9. If $\fY_k$ is reduced, we argue as follows: let $f\in B$ be power-bounded, and let $n$ be the natural number with the property that $\pi^nf\in \ul{B}-\pi \ul{B}$. If $n\geq 1$, then $|\pi^n f|_\sup<1$ because $|f|_\sup\leq 1$, and hence the class of $\pi^n f$ modulo $\pi$ is nilpotent. However, this cannot be the case, since $\ul{B}_k$ is reduced and since $\pi^nf\notin\pi \ul{B}$. It follows that $n=0$, which means that $f\in \ul{B}$.
\end{proof}

\subsection{Reduction to the case where the model of the domain is local}
We prove that if $\fY$ is normal or if $\fY_k$ is reduced, then $\phi$ extends to an open neighborhood of a closed point $y\in\fY$ if and only if $\phi$ extends to the completion of $\fY$ in $y$. 

Adopting the notation introduced by Berthelot in the rigid-analytic setting in cf.\ \cite{berthelot_rigcohpreprint} 1.1.2, we define:

\begin{defi} If $V$ is a locally closed subset in a formal $R$-scheme of ff type $\fZ$, we write $\fZ|_V$ to denote the completion of $\fZ$ along $V$, and we set
\[
]V[_\fZ\,\mathrel{\mathop:}=\,\sp_\fZ^{-1}(V)\;,
\]
where $\sp_\fZ:\fZ^\urig\rightarrow\fZ$ denotes the specialization map attached to $\fZ$, cf.\ \cite{urigspaces} Section 2.4. The set $]V[_\fZ$ is called the formal fiber of $V$. If $T$ is any topological space, we write $T_\cl$ to denote the set of locally closed points in $T$.
\end{defi}

\begin{remark}
It is easily seen that $]V[_\fZ\,=\,(\fZ|_V)^\srig$; in particular, $]V[_\fZ$ is admissible open in $\fZ^\urig$, and if $V$ is affine, then $]V[_\fZ$ is semi-affinoid. 
\end{remark}

If $\fU$ is an open formal subscheme in the given model $\fX$ of the target $X$ of $\phi$, then the set of (locally) closed points $y\in\fY$ whose formal fiber $]y[_\fY$ maps to the formal fiber $]\fU[_\fX$ of $\fU$ via $\phi$ is open:

\begin{prop}\label{mainweilresopenprop}
Let $\fU\subseteq\fX$ be an open formal subscheme; then
\[
\fY_\cl\setminus\sp_\fY(Y\setminus\phi^{-1}(\fU^\srig))
\]
is open in $\fY_\cl$.
\end{prop}
\begin{proof}
Let us first assume that $\phi$ extends to a morphism $\ul{\phi}\colon\fY\rightarrow\fX$. In this case, the set in the statement of the proposition is identified with $\ul{\phi}^{-1}(\fU)_\cl$, and there is nothing to show. Indeed, if $y$ is a closed point in $\fY$ such that $\ul{\phi}(y)\in\fU$, then $\sp_\fY^{-1}(y)$ maps to $\fU^\srig$ under $\phi$ since the specialization map is functorial. Conversely, if $\sp_\fY^{-1}(y)$ maps to $\fU^\srig$ via $\phi$, then $\ul{\phi}(y)\in\fU$, again by functoriality of the specialization map and because $\sp_\fY^{-1}(y)$ is non-empty, $\fY$ being $R$-flat, cf.\ \cite{urigspaces} Remark 2.5.

In the general case, we choose  a treelike covering  $(Y_i)_{i\in I}$ of $Y$ together with a model $(\fY_i)_{i\in I}$ in the sense of the discussion following \cite{urigspaces} Def.\ 2.31 such that $\fY_r=\fY$ for the root $r\in I$ and such that for each leaf $i\in I$, the restriction of $\phi$ to $Y_i$ extends to a morphism 
\[
\ul{\phi}_i\colon\fY_i\rightarrow\fX\;;
\]
 such a covering exists by \cite{urigspaces} Cor.\ 2.14 ($iv$). By what we have seen so far, the set
\[
S_i\,\mathrel{\mathop:}=\,(\fY_i)_\cl\setminus \sp_{\fY_i}(Y_i\setminus(Y_i\cap\phi^{-1}(\fU^\srig)))
\]
is open in $(\fY_i)_\cl$ for each leaf $i\in I$. We claim that the same statement holds in fact for all $i\in I$. Using induction on the volume of $I$, we may assume that it holds for all children $i$ of the root $r$ of $I$; it remains to show that the statement holds for the root $r$ itself. Let $\ul{\psi}\colon\fY'\rightarrow\fY$ denote the admissible blowup of $\fY$ that is part of the chosen model of the chosen treelike  covering; for each $i\in\children(r)$, we consider $\fY_i$ as an open formal subscheme of $\fY'$. It is clear that
\[
S'\,\mathrel{\mathop:}=\,\fY'_\cl\setminus\sp_{\fY'}(Y\setminus\phi^{-1}(\fU^\srig))
\]
is the union of the sets $S_i$ considered above, where $i$ varies in $\children(r)$. Let us write
\[
S\,\mathrel{\mathop:}=\,\fY_\cl\setminus\sp_\fY(Y\setminus\phi^{-1}(\fU^\srig))
\]
to denote the subset of $\fY_\cl$ that we are interested in. We claim that 
\[
S=\fY_\cl\setminus\ul{\psi}(\fY'_\cl\setminus S')\;;
\]
then the openness of $S$ follows from the fact that $\ul{\psi}$ is proper and, hence, a closed map of underlying topological spaces. It remains to justify the above claim. Let $y$ be a closed point in $\fY$. Since the specialization map is functorial and since $\ul{\psi}^\srig$ is bijective, the formal fiber of $y$ is identified with union of the formal fibers of the closed points $y'$ in $\ul{\psi}^{-1}(y)$. Hence, $y\in S$ is equivalent to $\ul{\psi}^{-1}(y)\subseteq S'$, where by abuse of notation we consider $\ul{\psi}$ as a map on sets of closed points. The claim is now obvious.
\end{proof} 

\begin{cor}\label{redtoffcor}
Let $y\in \fY$ be a closed point such that $\fY$ is normal in $y$ or such that $\fY_k$ is reduced in $y$. Then the following are equivalent:
\begin{packed_enum}
\item $\phi|_{]y[_\fY}$ extends to a morphism $\ul{\phi}_y\colon\fY|_y\rightarrow\fX$.
\item $\phi$ extends to an open neighborhood of $y$ in $\fY$.
\end{packed_enum}
\end{cor}
\begin{proof}
The implication ($ii$)$\Rightarrow$($i$) is trivial; we need to prove its converse. Let $\fU\subseteq\fX$ be an affine open neighborhood of the $\ul{\phi}_y$-image of the unique physical point of $\fY|_y$; then 
\[
]y[_\fY\,\subseteq\,\phi^{-1}(]\fU[_\fX)\;.
\] 
By Lemma \ref{affineextlem}, it suffices to show that there exists an open neighborhood $\fV\subseteq\fY$ of $y$ such that $]\fV[_\fY\subseteq\phi^{-1}(]\fU[_\fX)$. Since the underlying topological space of $\fY$ is a Jacobson space, it thus suffices to see that the set of closed points $y'\in \fY$ satisfying
\[
]y'[_\fY\subseteq\phi^{-1}(]\fU[_\fX)
\]
 is open in the set $\fY_\cl$ of closed points of $\fY$. This set is identified with
\[
\fY_\cl\setminus\sp_\fY(Y\setminus \phi^{-1}(]\fU[_\fX))\;,
\]
and hence this statement is provided by Proposition \ref{mainweilresopenprop}.
\end{proof}

\subsection{A descent argument} We show that $\ul{\phi}$ exists if it exists fpqc-locally on $\fY$; this statement may be viewed as an analog of \cite{blr} 2.5/5.

\begin{lem}\label{extensiondesclem}
If $\fX$ is separated and if there exist a flat formal $R$-scheme of locally ff type $\fY'$ together with a faithfully flat quasi-compact morphism $\ul{\psi}\colon\fY'\rightarrow\fY$ and a morphism $\ul{\phi}'\colon\fY'\rightarrow\fX$ such that the induced diagram of uniformly rigid generic fibers
\[
\begin{diagram}
Y'&&\\
\dTo>\psi&\rdTo>{\phi'}&\\
Y&\rTo^\phi&X
\end{diagram}
\]
commutes, then $\phi$ extends uniquely to a morphism $\ul{\phi}\colon\fY\rightarrow\fX$.
\end{lem}
\begin{proof}
We may work locally on $\fY$ and, hence, assume that $\fY$ is quasi-compact. Let $\ul{\Gamma}_\phi\subseteq\fY\times\fX$ denote the schematic closure of the graph of $\phi$ in $\fY\times\fX$. Then $\ul{\psi}^*\ul{\Gamma}_\phi$ coincides with the graph $\Gamma_{\ul{\phi}'}$ of $\ul{\phi}'$, and hence the restricted projection 
\[
p_\fY|_{\ul{\Gamma}_\phi}\colon\ul{\Gamma}_\phi\rightarrow\fY
\]
becomes an isomorphism after base change with respect to $\ul{\psi}$. By Theorem \ref{graphproperthm}, $p_\fY|_{\ul{\Gamma}_\phi}$ is proper, hence adic and separated, so we deduce from Lemma \ref{isodesclem} that $p_\fY|_{\ul{\Gamma}_\phi}$ is an isomorphism, as desired.
\end{proof}

\subsection{Formal targets which are groups containing all points}\label{groupcasesec}

In the following, we assume that $\fY$ is $R$-smooth and that $\fX$ is a smooth formal $R$-group scheme of locally ff type. Under these assumptions, we show that the unique extension $\ul{\phi}$ of $\phi$ always exists. This statement is an analog of the Weil extension theorem for rational maps to group schemes, cf.\ \cite{weil} \S II n. 15 Prop. 1, \cite{blr} 4.4 and \cite{bosch_schloeter} 2.6. In order to emphasize that $\fX$ is a group, we write $\fG$ instead of $\fX$ and $G$ instead of $X$. Let us note that rigid and uniformly rigid $K$-groups and formal $R$-groups of ff type are automatically separated. Let us first consider the situation where $\fY$ is of locally tf type; in this case we can reduce to the diagonal as in the proof of \cite{bosch_schloeter} 2.6 and then apply our descent Lemma \ref{extensiondesclem}.

\begin{prop}\label{tftypedomainprop}
If $\fY$ is of locally tf type, then the unique extension $\ul{\phi}$ of $\phi$ exists.
\end{prop}
\begin{proof}
We may assume that $\fY$ is affine and connected. By \cite{urigspaces} 2.40 and \cite{frg1} 4.3 and 4.4, every admissible covering of $Y$ is refined by an affine open covering of a suitable admissible blowup of $\fY$. By Lemma \ref{genisolem} and Lemma \ref{affineextlem}, it follows that there exists some nonempty open formal subscheme $\fW\subseteq\fY$ such that $\phi$ extends to a morphism 
\[
\ul{\phi}|_\fW\colon\fW\rightarrow\fG\;.
\]
Let 
\[
\psi\colon Y\times Y\rightarrow G
\]
denote the morphism sending $(x_1,x_2)$ to $\phi(x_1)\cdot \phi(x_2)^{-1}$. The product $\fW\times\fW$ is an open formal subscheme of $\fY\times\fY$, and the morphism 
\[
\ul{\psi}|_{\fW\times\fW}\colon\fW\times\fW\rightarrow\fG
\]
sending $(y_1,y_2)$ to $\ul{\phi}|_\fW(y_1)\cdot\ul{\phi}|_\fW(y_2)^{-1}$ extends $\psi$. We claim that $\psi$ extends to an open formal subscheme of $\fY\times\fY$ containing the diagonal $\Delta_\fY$. Let $\fH\subseteq\fG$ be an affine open neighborhood of the identify section, let $H$ denote its semi-affinoid generic fiber, and let $\fU\subseteq\fW\times\fW$ be the $\ul{\psi}|_{\fW\times\fW}$-preimage of $\fH$. Let $(y,y)$ be any closed point of $\Delta_\fY$; we must show that $\psi$ extends to an open neighborhood of $(y,y)$ in $\fY\times\fY$. If $(y,y)\in\fW\times\fW$, there is nothing to show, so we may assume that $(y,y)\notin\fW\times\fW$. Let $\fV$ be any connected affine open neighborhood of $(y,y)$ in $\fY\times\fY$. Then $\fU\cap\fV$ is $R$-dense in $\fV$, cf.\ Corollary \ref{schemdensecor}. Indeed, since $\fV_k$ is an integral scheme, it suffices to observe that $\fU\cap\fV$ is nonempty, and already $\fU\cap\fV\cap\Delta_\fY$ is nonempty since $(\Delta_\fY)_k\cong\fY_k$ is an integral scheme and since both $\fU\cap\Delta_\fY$ and $\fV\cap\Delta_\fY$ are nonempty open formal subschemes of $\Delta_\fY$, the former containing $\Delta_\fW$ and the latter containing $(y,y)$. 

Let $d$ denote the relative dimension of $\fY$ in $y$. After shrinking $\fV$, we may assume that conditions ($i$)--($iii$) in the proof of \cite{bosch_schloeter} Theorem 2.6 are satisfied on $\fV$. That is, $\fV$ does not meet any component of $(\fY\times\fY)\setminus\fU$ not containing $(y,y)$, and there exists an $R$-regular sequence of global functions $f_1,\ldots,f_{2d-1}$ on $\fV$ such that $\Delta_\fY\cap\fV$ is defined by $f_1,\ldots,f_d$, such that the vanishing locus of the $f_i$ is a formal subscheme $\fC\subseteq\fV$ of relative dimension one and such that $\fC\cap(\fV\setminus\fU)=\{(y,y)\}$. Indeed, this follows from the fact that $\fY$ is $R$-smooth and that $\fU\cap\fV\cap\Delta_\fY$ is $R$-dense in $\Delta_\fY$ by choosing a suitable regular parameter system for the stalk of $\fY\times\fY$ in $(y,y)$.

Let $V$ denote the semi-affinoid generic fiber of $\fV$. Since $\psi$ maps $\Delta_Y$ to the unit section of $G$, \cite{kisinlocconst} Lemma 2.3 shows that there exists an $\varepsilon>0$ in $\sqrt{|K^*|}$ such that the tube
\[
V^\r(\varepsilon^{-1}f_1,\ldots,\varepsilon^{-1}f_{2d-1})
\]
in the affinoid $K$-space $V^\r$ maps to $H^\r$ via $\psi^\r$. Hence, $V(\varepsilon^{-1}f_1,\ldots,\varepsilon^{-1}f_{2d-1})$ maps to $H$ via $\psi$, and so the same holds for the Hartogs figure
\[
V(\varepsilon^{-1}f_1,\ldots,\varepsilon^{-1}f_{2d-1})\cup(\fV\cap\fU)^\srig\;.
\]
By \cite{luetkebohmert_mer} Theorem 7, the $\psi$-pullback of any global function on $H$ extends to a global function on $V$. We thus obtain a morphism of semi-affinoid $K$-spaces $\psi'_V\colon V\rightarrow H$ coinciding with $\psi|_V$ on the above Hartogs figure. The coincidence subspace $(\psi'_V,\psi|_V)^{-1}(\Delta_G)$ is a closed semi-affinoid subspace of $V$ which contains the above Hartogs figure and, hence, coincides with $V$. It follows that $\psi'_V$ and $\psi|_V$ coincide, which implies that $\psi|_V$ factors through $H$. Since the affine formal $R$-schemes $\fV$ and $\fH$ are smooth, they are normal by Corollary \ref{smoothimpliesnormalcor}, and by Lemma \ref{affineextlem} it follows that $\psi|_V$ extends uniquely to a morphism $\ul{\phi}|_\fV\colon\fV\rightarrow\fH$, as desired. Our claim has been shown.

Let $\fV$ now denote any open neighborhood of $\Delta_\fY$ in $\fY\times\fY$ such that the restriction of $\psi$ to the uniformly rigid generic fiber $V$ of $\fV$ extends to a morphism $\ul{\psi}|_\fV$ from $\fV$ to $\fG$. Let us consider the morphism $\fV\cap(\fY\times\fW)\rightarrow\fG$ sending $(y_1,y_2)$ to $\ul{\psi}|_\fV(y_1,y_2)\cdot\ul{\phi}|_\fW(y_2)$. On uniformly rigid generic fibers, it spans a commutative triangle together with $\phi$ and the generic fiber of the projection $p_1\colon\fY\times\fY\rightarrow\fY$ to the first factor. Moreover, the restriction of $p_1$ to $\fV\cap(\fY\times\fW)$ is faithfully flat. Indeed, flatness is clear, and surjectivity follows from the fact that $\fV$ contains $\Delta_\fY$ and that $\fW\subseteq\fY$ is $R$-dense. It now follows from Lemma \ref{extensiondesclem} that $\phi$ extends uniquely to a morphism $\ul{\phi}\colon\fY\rightarrow\fG$, as desired.
\end{proof}

The statement of Proposition \ref{tftypedomainprop} generalizes to the case where $\fY$ is of locally ff type over $R$:

\begin{thm}\label{ffteverywheredefinedthm}
The unique extension $\ul{\phi}:\fY\rightarrow\fG$ of $\phi$ always exists.
\end{thm}
\begin{proof}
Let
\[
\psi\colon Y\times Y\rightarrow G
\]
be the morphism sending $(y_1,y_2)$ to $\phi(y_1)\cdot \phi(y_2)^{-1}$.  We claim that $\psi$ extends to an open neighborhood of $\Delta_\fY$ in $\fY\times\fY$. 

Let $\fH\subseteq\fG$ be an affine open neighborhood of the identity section, and let $H\subseteq G$ denote the uniformly rigid generic fiber of $\fH$. The formal $R$-scheme $\fY\times\fY$ is $R$-smooth and, hence, normal by Corollary \ref{smoothimpliesnormalcor}. By Corollary \ref{redtoffcor}, it thus suffices to show that for every closed point $(y,y)$ in $\Delta_\fY$, the formal fiber $](y,y)[_{\fY\times\fY}$ maps to $H$ under $\psi$. To establish our claim, we may thus replace $\fY$ by its completion along a closed point $y$ and thereby assume that $\fY$ is local. Moreover, after some finite possibly ramified base field extension we may assume that the residue field of $\fY$ naturally coincides with $k$, cf.\ Proposition \ref{smoothlocalstructureprop} and its proof. By Lemma \ref{localstructurerationalpointlem}, $\fY$ is then isomorphic to $\D^d_R$, where $d$ denotes the relative dimension of $\fY$ over $R$ and where $\D_R=\Spf R[[S]]$. Then $\fY\times\fY=\D^{2d}_R$, and we must show that $\psi$ factors through $H$. Since $X$ is the (non-admissible) union of the admissible open subspaces 
\[
\B^d_{\leq\varepsilon,K}\times\B^d_{\leq\varepsilon,K}\quad\textup{for}\quad\varepsilon\rightarrow 1\quad\textup{in}\quad\sqrt{|K^*|}\;,
\]
where for $\sqrt{|K^*|}$ we let $\B_{\leq\varepsilon}$ denote the closed unit disc of radius $\varepsilon$, we may, after some finite ramified base field extension, assume that $\fY=\B^d_R$, where $\B_R=\Spf R\langle S\rangle$, and that $\phi$ extends to a morphism $\phi'$ on a strictly larger closed polydisc $Y'$ whose radius lies in the value group of the base field. Let $\fY'=\B^d_R$ be the natural smooth $R$-model of $Y'$, and let $\ul{\tau}\colon\fY\rightarrow\fY'$ be the $R$-morphism corresponding to the inclusion of $Y$ into $Y'$. Then the the physical image of $\ul{\tau}\times\ul{\tau}$ is a single point, so suffices to show that $\phi'$ extends to a morphism $\fY'\rightarrow\fG$. This follows from Proposition \ref{tftypedomainprop}, so our claim has been shown.

Let us now return to the situation where $\fY$ is a general smooth formal $R$-scheme of locally ff type. Let $\fV\subseteq\fY\times\fY$ be an open neighborhood of $\Delta_\fY$ with uniformly rigid generic fiber $V$ such that $\psi|_V$ extends to a morphism $\ul{\psi}|_\fV\colon\fV\rightarrow\fG$, let $\ul{\Gamma}_\phi$ denote the schematic closure of the graph of $\phi$ in $\fY\times\fG$, and let $\ul{\phi}'\colon\ul{\Gamma}_\phi\rightarrow\fG$ denote the restriction of the projection $p_\fG\colon\fY\times\fG\rightarrow\fG$. Let $\fV'$ denote the preimage of $\fV$ under the morphism $(\id_\fY\times p_\fY|_{\ul{\Gamma}_\phi})\colon\fY\times\ul{\Gamma}_\phi\rightarrow\fY\times\fY$; we claim that $\fV'$ is faithfully flat over $\fY$ via the first projection $p_1$. Flatness follows from the fact that $\ul{\Gamma}_\phi$ is $R$-flat, from the fact that flatness is preserved under base change, cf.\ Proposition \ref{flatnessbasechangeprop}, and from the fact that open immersions of locally noetherian formal schemes are flat. To establish surjectivity, we argue as follows: Since images of morphisms of $k$-schemes of finite type are constructible, cf.\ \cite{egaiv} 1.8.5, and since the closed points in a $k$-scheme of finite type lie very dense, if suffices to see that the image of $p_1|_{\fV'}$ contains all closed points in $\fY$. Let $x\in\fY$ be a closed point; since $p_\fY|_{\ul{\Gamma}_\phi}$ induces an isomorphism of uniformly rigid generic fibers and since $\sp_\fY$ is surjective, there exists a closed point $x'\in\ul{\Gamma}_\phi$ above $x$; then $(x,x')$ is a point in $\fV'$ projecting to $x$. Surjectivity has thus been shown. Let us now consider the morphism $\fV'\rightarrow\fG$ sending a point $(x,x')$ to $\ul{\psi}|_\fV(x,p_\fY|_{\ul{\Gamma}_\phi}(x'))\cdot\ul{\psi}'(x')$. Together with $p_1$ and $\phi$, we obtain a commutative triangle on uniformly rigid generic fibers. It follows from Lemma \ref{extensiondesclem} that $\phi$ extends to a morphism $\ul{\phi}\colon\fY\rightarrow\fG$, as desired.
\end{proof}

\subsection{Formal targets which are groups containing all formally unramified points}

We now modify the setup which we described at the beginning of Section \ref{weilsec} and in Section \ref{groupcasesec}. As before, we let $\fY$ be a smooth formal $R$-scheme of locally ff type with uniformly rigid generic fiber $Y$. Let $\fG$ be a smooth quasi-paracompact formal $R$-group scheme of locally tf type, let $G$ be a smooth quasi-paracompact rigid $K$-group, and let $\iota:\fG^\rig\hookrightarrow G$ be a retrocompact open immersion respecting the given group structures. In this situation, we show that every uniformly rigid morphism $\phi$ from $Y$ to the uniform rigidification $G^\ur$ of $G$ factors through $\fG^\urig$ (and, by Theorem \ref{ffteverywheredefinedthm}, thus extends to a morphism $\ul{\phi}:\fY\rightarrow\fG$), provided that the image of $\iota$ contains all formally unramified points of $G$.

\begin{defi}
A not necessarily finite discrete analytic extension field $K'$ of $K$ is called formally unramified over $K$ if the induced local extension of complete discrete valuation rings is formally unramified. If $X$ is a quasi-separated rigid space and if $U\subseteq X$ is an open rigid subspace, we say that $U$ contains all formally unramified points\index{point!formally unramified} of $X$ if for every formally unramified discrete analytic extension field $K'$ over $K$, every $K'$-valued point of $X\hat{\otimes}_KK'$ lies in the admissible open subspace $U\hat{\otimes}_KK'$; cf.\ \cite{bgr} 9.3.6 for the definition of the functor $\cdot\,\hat{\otimes}_KK'$.
\end{defi}

\begin{thm}\label{mainweilextthm}
If the image of $\iota$ contains all formally unramified points of $G$, then $\phi$ extends to a morphism $\ul{\phi}:\fY\rightarrow\fG$.
\end{thm}

The \textit{proof} of Theorem \ref{mainweilextthm} is rather elaborate, so we divide it into several steps.

\textit{Step 1: Choice of a model for $\iota$.} By \cite{bosch_frgnotes} 2.8/3, there exist a quasi-paracompact $R$-model of locally tf type $\fG''$ for $G$ and an admissible blowup $\ul{b}_1:\fG'\rightarrow\fG$ such that the retrocompact open immersion $\iota$ of $\fG^\rig$ into $G$ is induced by a morphism $\fG'\rightarrow\fG''$. Since $\fG^\rig\subseteq G$ is retrocompact, \cite{frg2} Corollary 5.4 and \cite{bosch_frgnotes} 2.6/13--14 show that after replacing $\fG''$ and $\fG'$ by suitably chosen admissible blowups, we may assume that $\fG'\subseteq\fG''$ is an open formal subscheme. 

\textit{Step 2: Reduction to the case of a local formal domain.} By Corollary \ref{redtoffcor}, we may assume that $\fY$ is local. Clearly we may assume that $\fY$ is nonempty.

\textit{Step 3: Extension of $\phi$ to an open $R$-dense part of the $R$-envelope.} Now and only now we use the fact that the image of $\iota$ contains all formally unramified points of $G$. Let $\ul{\Gamma}_\phi$ denote the schematic closure of the graph of $\phi$ in $\fY\times\fG''$. By Corollary \ref{genisocor}, the $R$-envelope $(p_\fY|_{\ul{\Gamma}_\phi})_\pi$ of the projection $p_\fY|_{\ul{\Gamma}_\phi}$ from $\ul{\Gamma}_\phi$ to $\fY$ is an isomorphism over an open neighborhood $\fW\subseteq\fY_\pi$ of the generic point $\eta$ in $\fY_{\pi,k}$. Let us abbreviate
\[
\ul{\phi}'_\pi\,\mathrel{\mathop:}=\,(p_{\fG''})_\pi|_{(\ul{\Gamma}_\phi)_\pi}\;,
\]
where 
\[
(p_{\fG''})_\pi\colon\fY_\pi\times\fG''\rightarrow\fG''
\]
denotes the projection to the second factor. 
By Corollary \ref{smoothimpliesnormalcor} and Remark \ref{normdvrrem}, the completed stalk $R_\eta$ of $\fY_\pi$ in $\eta$ is a complete discrete valuation ring; by Proposition \ref{regprop}, it is a formally unramified local extension of $R$. Since $\fG^\rig$ contains all formally unramified points of $G$, $\eta\in(\ul{\phi}'_\pi)^{-1}(\fG')$. After shrinking $\fW$, we may therefore assume that 
\[
\fW\,\subseteq\,(\ul{\phi}'_\pi)^{-1}(\fG')
\]
so that we obtain a morphism $\ul{\phi}'_\pi|_\fW\colon\fW\rightarrow\fG'\rightarrow\fG$, where $\fW\subseteq\fY_\pi$ is open and $R$-dense.

\textit{Step 4: Reduction to the case where the formal domain is an open formal polydisc.} By Theorem \ref{ffteverywheredefinedthm}, it suffices to show that $\phi$ factors through the image of $\iota$, which may be verified after an extension of the base field. After some finite possibly ramified base field extension $K'/K$, the residue field of $\fY$ naturally coincides with $k$, cf.\ Proposition \ref{smoothlocalstructureprop} and its proof. By Lemma \ref{localstructurerationalpointlem}, $\fY$ is then isomorphic to $\fY=\Spf R[[T_1,\ldots,T_d]]$. The outcome of Step 7 is preserved under the possibly ramified finite local extension of discrete valuation rings $R'/R$ corresponding to $K'/K$ because $(\fY\times_RR')_{\pi'}=\fY_\pi\times_RR'$, where $\pi'$ denotes a uniformizer of $R'$, and because $\fW\times_RR'\subseteq\fY_\pi\times_RR'$ is $R'$-dense. Since the initial assumption that the image of $\iota$ contains all formally unramified points of $G$ is only used in Step 3, we may and will thus assume in the rest of the proof that $\fY=\Spf R[[T_1,\ldots,T_d]]$.

\textit{Step 5: The twist of $\phi$.} Let $\delta: G^\sr\times G^\sr\rightarrow G^\sr$ and $\ul{\delta}_\fG:\fG\times\fG\rightarrow\fG$ denote the twisted multiplication morphisms sending $(g_1,g_2)$ to $g_1\cdot g_2^{-1}$, and let 
\[
\psi\colon Y\times Y\rightarrow G^\sr
\]
denote the morphism $\delta\circ(\phi\times\phi)$ sending $(y_1,y_2)$ to $\phi(y_1)\cdot \phi(y_2)^{-1}$; we say that $\psi$ is the twist of $\phi$. Since $\iota$ respects group structures, the generic fiber $\delta_\fG:=\ul{\delta}_\fG^\urig:\fG^\urig\times\fG^\urig\rightarrow\fG^\urig$ of $\ul{\delta}_\fG$ is obtained from $\delta$ via restriction.

\textit{Step 6: An affine open neighborhood $\fH$ of the identity section.}  Let $\fH\subseteq\fG$ be an affine open neighborhood of the identity section, and let $H\subseteq\fG^\urig$ denote its semi-affinoid generic fiber. 

\textit{Step 7: Preimages under twisted multiplication.} We have a possibly strict inclusion of preimages
\[
(\delta_\fG)^{-1}(H)\,\subseteq\,\delta^{-1}(H)\;.
\]
A natural $R$-model for $(\delta_\fG)^{-1}(H)$ is given by 
\[
\ul{\delta}_\fG^{-1}(\fH)\subseteq\fG\times\fG\;.
\]
Let $\fH'\subseteq\fG'$ denote the strict transform of $\fH$ under $\fG'\rightarrow\fG$; then $\fH'$ is an open formal subscheme of $\fG''$ via the inclusion $\fG'\subseteq\fG''$. By \cite{bosch_frgnotes} 2.8/3, there exists an admissible blowup $\ul{b}_1:(\fG''\times\fG'')'\rightarrow\fG''\times\fG''$ such that $\delta$ extends to a morphism $\ul{\delta}:(\fG''\times\fG'')'\rightarrow\fG''$; then 
\[
\ul{\delta}^{-1}(\fH')\,\subseteq\,(\fG''\times\fG'')'
\]
is an $R$-model of $\delta^{-1}(H)$. In fact, $(\delta_\fG)^{-1}(H)$ has a model which is an open formal subscheme of $\ul{\delta}^{-1}(\fH')$: let $\ul{\delta}_\fG^{-1}(\fH)'\subseteq\fG'\times\fG'$ denote the strict transform of $\ul{\delta}_\fG^{-1}(\fH)$ under $(\ul{b}_1\times \ul{b}_1):\fG'\times\fG'\rightarrow\fG\times\fG$; we view it as an open formal subscheme of $\fG''\times\fG''$ via the inclusion $\fG'\times\fG'\subseteq\fG''\times\fG''$. Let $\ul{\delta}_\fG^{-1}(\fH)''\subseteq(\fG''\times\fG'')'$ be the strict transformation of $\ul{\delta}_\fG^{-1}(\fH)'$ under $\ul{b}_1$; then $\ul{\delta}_\fG^{-1}(\fH)''$ is a model of $(\delta_\fG^\urig)^{-1}(H)$, and it is contained in $\ul{\delta}^{-1}(\fH')$ because the specialization map of $(\fG''\times\fG'')'$ is surjective onto closed points.


\textit{Step 8: Formulation of an intermediate claim.}  We claim that the twist $\psi$ of $\phi$ extends to a morphism $\ul{\psi}:\fV\rightarrow\fH$, where $\fV$ is an open neighborhood of the diagonal in $\fY\times\fY$. Steps 9 till 18 will yield the proof of this claim. The idea is to simply mimic the analytic continuation argument which was used in the proof of \cite{bosch_schloeter} Theorem 2.6; however, life is obfuscated by the fact the we need to work on $R$-envelopes.

\textit{Step 9: Reduction of the intermediate claim to a statement on analytic continuation.} Since $\fY$ is an open formal polydisc, the product $\fY\times\fY$ is local. In this situation, the intermediate claim thus says that $\psi$ extends to a morphism $\ul{\psi}$ from $\fY\times\fY$ to $\fH$. To establish this claim, it suffices to show that the following holds:
\begin{enumerate}
\item $\psi^{-1}(H)$ contains a semi-affinoid subdomain $T$ such that the restriction homomorphism $\Gamma(Y\times Y,\O_{Y\times Y})\rightarrow\Gamma(T,\O_{Y\times Y})$ is injective, and
\item the $\psi$-pullback $\psi^*h\in\Gamma(\psi^{-1}(H),\O_{Y\times Y})$ of any semi-affinoid function $h$ on $H$ extends uniquely to a function on $Y\times Y$.
\end{enumerate}
Indeed, since $\fY\times\fY$ is normal, condition ($ii$) implies that $\psi^*$ induces a morphism $\ul{\psi}:\fY\times\fY\rightarrow\fH$ such that the restriction of $\ul{\psi}^\urig$ to any open semi-affinoid subspace of $\psi^{-1}(H)$ coincides with the restriction of $\psi$ to that subspace. Let $V\subseteq Y\times Y$ be the closed uniformly rigid coincidence subspace of $\psi$ and $\ul{\psi}^\urig$, that is, the pullback of the diagonal of $G^\ur\times G^\ur$ under $\psi\times\ul{\psi}^\urig$. Then $V$ contains $T$ because $T\subseteq\psi^{-1}(H)$, so condition ($i$) shows that $V=Y\times Y$, which implies that $\psi=\ul{\psi}^\urig$. To show the intermediate claim, it thus suffices to establish ($i$) and ($ii$) above.

\textit{Step 10: Extension of $\psi$ to an open part of the $R$-envelope which has $R$-dense intersection with the diagonal.}  Let
\[
\fY_\pi\overset{(\Delta_\fY)_\pi}{\rightarrow}(\fY\times\fY)_\pi\overset{\ul{p}_{1,\pi},\ul{p}_{2,\pi}}{\rightrightarrows}\fY_\pi
\]
be the diagram which is obtained from the diagram
\[
\fY\overset{\Delta_\fY}{\rightarrow}\fY\times\fY\overset{\ul{p}_1,\ul{p}_2}{\rightrightarrows}\fY
\]
by passing to rings of global sections, where $\Delta_\fY$ is the diagonal morphism and where the $\ul{p}_i$ are the projections. Then $(\Delta_\fY)_\pi$ is a closed immersion, and $\ul{p}_{i,\pi}\circ(\Delta_\fY)_\pi=\id_{\fY_\pi}$ for $i=1,2$. By abuse of notation, we will not distinguish between the closed immersion $(\Delta_{\fY})_\pi$ and the closed formal subscheme of $(\fY\times\fY)_\pi$ that it defines. We set
\begin{eqnarray*}
\fU&:=&(\ul{p}_{1,\pi})^{-1}(\fW)\cap (\ul{p}_{2,\pi})^{-1}(\fW)\\
&\subseteq&(\fY\times\fY)_\pi\;,
\end{eqnarray*}
where $\fW\subseteq\fY_\pi$ was obtained in Step 3; then $(\Delta_{\fY})_\pi^{-1}(\fU)=\fW$ is $R$-dense in $\fY_\pi$. We define a morphism $\ul{\psi}_\pi|_\fU\colon\fU\rightarrow\fG$ by setting 
\[
\ul{\psi}_\pi|_\fU\,\mathrel{\mathop:}=\,\ul{\delta}\circ(\ul{\phi}'_\pi|_\fW\circ (\ul{p}_{1,\pi})|_\fU,\ul{\phi}'_\pi|_\fW\circ (\ul{p}_{2,\pi})|_\fU)\;,
\]
where for $i=1,2$, $(\ul{p}_{i,\pi})|_\fU\colon\fU\rightarrow\fW$ denotes the restriction of $\ul{p}_{i,\pi}$ and where $\ul{\phi}'|_\fW:\fW\rightarrow\fG$ was defined in Step 3. The restriction of $\ul{\psi}_\pi|_\fU$ to $\fU\cap(\Delta_{\fY})_\pi$ factors through the unit section of $\fG$; hence in particular
\[
\fU\cap(\Delta_{\fY})_\pi\,\subseteq\,(\ul{\psi}_\pi|_\fU)^{-1}(\fH)\;.
\]
We may thus shrink $\fU$ such that $\ul{\psi}_\pi|_\fU$ factors through $\fH$ without losing the property that $\fU\cap(\Delta_{\fY})_\pi$ is $R$-dense in $(\Delta_{\fY})_\pi$. Since the special fiber of $(\Delta_{\fY})_\pi\cong\fY_\pi$ is integral, we may shrink $\fU$ further such that $\fU$ is a basic open subset in $(\fY\times\fY)_\pi$, again without losing the property that $\fU\cap(\Delta_{\fY})_\pi$ is $R$-dense in $(\Delta_{\fY})_\pi$. Let $f$ be a global function on $(\fY\times\fY)_\pi$ such that $\fU$ is the complement of the formal hypersurface $V(f)$ defined by $f$. We may assume that $f$ is nonzero, for if $\fU$ was empty, then $(\Delta_{\fY})_\pi$ and, hence, $Y$ would be empty, and there would be nothing to show. Moreover, we may assume that $f$ is not a unit: indeed, let us assume that $f$ is a unit. Then $\fU$ is all of $(\fY\times\fY)_\pi$, and hence $\fW$ is all of $\fY_\pi$. In the situation of Step 3, the morphism $(p_\fY|_{\ul{\Gamma}_\phi})_\pi$ is then an isomorphism, and hence $p_\fY|_{\ul{\Gamma}_\phi}$ is an isomorphism as well. It follows that $\phi$ extends to a morphism $\ul{\phi}':\fY\rightarrow\fG''$ which extends to a morphism $\ul{\phi}'_\pi:\fY_\pi\rightarrow\fG''$. Since $\fW\subseteq(\ul{\phi}'_\pi)^{-1}(\fG')$, the equality $\fW=\fY_\pi$ implies that $\ul{\phi}'_\pi$ and, hence, $\ul{\phi}'$ factor over $\fG'$. This shows that $\phi$ factors over the image of $\iota$, and Theorem \ref{ffteverywheredefinedthm} shows that $\phi$ extends to a morphism $\ul{\phi}:\fY\rightarrow\fG$, which yields the conclusion of Theorem \ref{mainweilextthm} and in particular the intermediate claim. We have thus reduced to the case where $f$ is a nonzero non-unit. 

\textit{Step 11: A tube containing the diagonal.} We now choose the semi-affinoid subdomain $T\subseteq\psi^{-1}(H)$ of condition ($i$) in Step 9 as a tubular neighborhood of a smooth closed formal $R$-subscheme $\fT\subseteq\fY\times\fY$ of relative dimension one with the property that $\fT$ contains $\Delta_\fY$ and that $\fT_\pi\subseteq(\fY\times\fY)_\pi$ intersects the complement $V(f)$ of $\fU$ in $(\fY\times\fY)_\pi$ transversely: since $\fU\cap(\Delta_{\fY})_\pi$ is $R$-dense in $(\Delta_{\fY})_\pi$, we may argue as in the proof of \cite{bosch_schloeter} Thm.\ 2.6 to find an $R$-isomorphism 
\[
\fY\times\fY\cong\Spf R[[S_1,\ldots,S_{2d-1},Z]]
\]
such that the diagonal $\Delta_\fY$ is defined by $S_1,\ldots,S_d$ and such that $f$ is $Z$-di\-stin\-guished in the sense of Section \ref{ancontsec}. We briefly write $\ul{S}$ to denote the system $S_1,\ldots,S_{2d-1}$; let $\fT\subseteq\fY\times\fY$ be the closed formal subscheme defined by $\ul{S}$. Since $\psi$ maps $\Delta_Y$ to the identity of $G^\sr$, Corollary \ref{tubecor} and Lemma \ref{tubeauxlem} show that $\psi^{-1}(H)$ contains some tube around the diagonal and, hence, some tube around the vanishing locus $\fT$ of the $S_i$. More precisely speaking, there exists an $r\in\N$ such that 
\[
T\,:=\,(Y\times Y)(\pi^{-r}\ul{S})\subseteq \psi^{-1}(H)\;,
\]
where we use the notation of Corollary \ref{tubecor}. The restriction homomorphism
\[
\Gamma(Y\times Y,\O_{Y\times Y})\rightarrow\Gamma(T,\O_{Y\times Y})
\]
is injective since it is a flat homomorphism of domains, where flatness holds by \cite{urigspaces} Rem.\ 2.49 and where $\Gamma(T,\O_{Y\times Y})$ is a domain by Proposition \ref{formaltubestructureprop}; hence $T$ does satisfy the requirements of condition ($i$) in Step 9. It remains to establish condition ($ii$) in Step 9.

\textit{Step 12: A formal model for the tube.} Let $\ul{\tau}\colon(\fY\times\fY)\langle\pi^{-r}\ul{S}\rangle\rightarrow\fY\times\fY$ be the restriction of the blowup of $\fY\times\fY$ in the ideal $(\pi^r,\ul{S})$ to the affine open part where the pullback of this ideal is generated by $\pi^r$; then $\ul{\tau}$ is a model of the open subspace $T\subseteq Y\times Y$. Let
\[
\ul{\tau}_\pi\colon(\fY\times\fY)\langle\pi^{-r}\ul{S}\rangle_\pi\rightarrow(\fY\times\fY)_\pi
\]
denote the induced morphism of $R$-envelopes. By Proposition \ref{formaltubestructureprop}, the formal $R$-scheme $(\fY\times\fY)\langle\pi^{-r}\ul{S}\rangle$ is $R$-smooth and, hence, normal; cf.\ Corollary \ref{smoothimpliesnormalcor}. It thus follows from Lemma \ref{affineextlem} that the restriction of $\psi$ to the above tube is extends uniquely to a morphism of affine formal $R$-schemes $\ul{\psi}'\colon(\fY\times\fY)\langle\pi^{-r}\ul{S}\rangle\rightarrow\fH$. Let
\[
\ul{\psi}'_\pi\colon(\fY\times\fY)\langle\pi^{-r}\ul{S}\rangle_\pi\rightarrow\fH
\]
denote the induced morphism of $R$-envelopes.

\textit{Step 13: Application of the Continuation Theorem \ref{ancontthm}.} Let $h$ be a function on $H$, let $\psi^*h$ be its $\psi$-pullback to $\psi^{-1}(H)$, and let $(\psi^*h)|_T$ be the restriction of this pullback to $T$. To establish condition ($ii$) in Step 9, we must show that $(\psi^*h)|_T$ extends to a function on $Y\times Y$; such an extension is unique because the restriction homomorphism from $\Gamma(Y\times Y,\O_{Y\times Y})$ to $\Gamma(T,\O_{Y\times Y})$ is injective, as we have seen above. We may assume that $h$ is defined on $\fH$; then $(\ul{\psi}')^*h=(\psi^*h)|_T$, and in particular $(\psi^*h)|_T$ extends to $(\fY\times\fY)\langle\pi^{-1}\ul{S}\rangle$. The homomorphism of flat $R$-algebras $\ul{\tau}^*$ that corresponds to $\ul{\tau}$ is injective because, as we have just recalled, it is injective after inverting $\pi$. We have to show that $(\ul{\psi}')^*h$ extends to $\fY\times\fY$ or, equivalently, that $(\ul{\psi}'_\pi)^*h$ extends to $(\fY\times\fY)_\pi$. By the Continuation Theorem \ref{ancontthm}, it suffices to show that there exists a function $\tilde{h}$ on $\fU=(\fY\times\fY)_{\pi,f}$ such that the pullback of $\tilde{h}$ to $\ul{\tau}_\pi^{-1}(\fU)=(\fY\times\fY)\langle\pi^{-1}\ul{S}\rangle_{\pi,f}$ coincides with the restriction of $(\ul{\psi}'_\pi)^*h$; here we use the fact that $f$ is $Z$-distinguished. Let us abbreviate $\fU':=\ul{\tau}_\pi^{-1}(\fU)$.

\textit{Step 14: A compatibility statement.} We set $\tilde{h}:=(\ul{\psi}_\pi|_\fU)^*h$, where $\ul{\psi}_\pi|_\fU$ was defined in Step 10 in terms of the morphism $\ul{\phi}'_\pi|_\fW$ (cf.\ Step 3) and the twisted multiplication $\ul{\delta}$ of $\fG$. To show that $h'$ satisfies the requirements of Step 13 and hereby settle the intermediate claim, we must show that the compatibility relation 
\[
\ul{\psi}'_\pi|_{\fU'}\,=\,\ul{\psi}_\pi|_\fU\circ\ul{\tau}_\pi|_{\fU'}
\]
holds, where these morphisms have $\fH$ as a target. 

\textit{Step 15: Reduction to a more local compatibility statement.} By Proposition \ref{formaltubestructureprop}, the formal $R$-scheme $(\fY\times\fY)\langle\pi^{-r}\ul{S}\rangle$ is smooth and connected; hence its special fiber $(\fY\times\fY)\langle\pi^{-r}\ul{S}\rangle\times_Rk$ is $k$-smooth and connected. It follows that the ring of global functions on this special fiber is integral, which implies that the special fiber of $(\fY\times\fY)\langle\pi^{-r}\ul{S}\rangle_\pi$ is an integral scheme and, hence, that every open formal subscheme of $(\fY\times\fY)\langle\pi^{-r}\ul{S}\rangle_\pi$ is connected. By Proposition \ref{regprop}, smoothness of $(\fY\times\fY)\langle\pi^{-1}\ul{S}\rangle$ implies that the ring of functions on $(\fY\times\fY)\langle\pi^{-1}\ul{S}\rangle$ and, hence, on $(\fY\times\fY)\langle\pi^{-1}\ul{S}\rangle_\pi$ is $R$-regular; by \cite{egaiv} 7.8.3 ($v$), it follows that the ring of global functions on any affine open formal subscheme $\fU''$ of 
$(\fY\times\fY)\langle\pi^{-r}\ul{S}\rangle_\pi$ $R$-regular as well and, hence, regular. It follows that the ring of functions on every nonempty affine open formal subscheme of $(\fY\times\fY)\langle\pi^{-r}\ul{S}\rangle_\pi$ is a domain. Since flat homomorphisms of domains are injective, the restriction homomorphism 
\[
\Gamma(\fU',\O_{(\fY\times\fY)\langle\pi^{-r}\ul{S}\rangle_\pi})\rightarrow\Gamma(\fU'',\O_{(\fY\times\fY)\langle\pi^{-r}\ul{S}\rangle_\pi})
\]
is injective for any nonempty affine open formal subscheme $\fU''$ of $\fU'$. Hence, it suffices to show that the compatibility relation
\[
\ul{\psi}'_\pi|_{\fU''}\,=\,\ul{\psi}_\pi|_\fU\circ\ul{\tau}_\pi|_{\fU''}
\]
holds for some arbitrarily small nonempty affine open formal subscheme $\fU''$ of $\fU'$, where again the target of these morphisms is $\fH$.

\textit{Step 16: An extension of $\psi|_T$ to an $R$-envelope.} As we have seen in Step 12, the restriction $\psi|_T$ of $\psi$ to the tube $T$ extends to a morphism $\ul{\psi}'$ from the formal tube $(\fY\times\fY)\langle\pi^{-r}\ul{S}\rangle$ to $\fH$ and, hence, to a morphism $\ul{\psi}'_\pi$ from the $R$-envelope of this formal tube to $\fH$. We now consider a different $R$-model of $\psi|_T$ which also extends to an $R$-envelope and which relates both to $\ul{\psi}'_\pi$ and to $\ul{\psi}_\pi|_\fU$. Let us first get back to the notation which we introduced in Step 5. Let
\[
\ul{\Gamma}_{\phi\times\phi}\,\subseteq\,\fY\times\fY\times(\fG''\times\fG'')'
\]
be the schematic closure of the graph of $\phi\times\phi$, and let $\ul{\Gamma}_{\phi\times\phi}\langle\pi^{-r}\ul{S}\rangle$ be its strict transform under $\ul{\tau}$; then 
\[
\ul{\Gamma}_{\phi\times\phi}\langle\pi^{-r}\ul{S}\rangle\,\subseteq\,(\fY\times\fY)\langle\pi^{-r}\ul{S}\rangle\times(\fG''\times\fG'')'
\]
is the schematic closure of the graph of $(\phi\times\phi)|_T$. Since $\sp_{\ul{\Gamma}_{\phi\times\phi}\langle\pi^{-r}\ul{S}\rangle}$ is surjective onto the closed points of $\ul{\Gamma}_{\phi\times\phi}\langle\pi^{-r}\ul{S}\rangle$ and since $(\phi\times\phi)|_T$ factors through $H$, 
\[
\ul{\Gamma}_{\phi\times\phi}\langle\pi^{-r}\ul{S}\rangle\,\subseteq\,(\fY\times\fY)\langle\pi^{-r}\ul{S}\rangle\times\ul{\delta}^{-1}(\fH')\;.
\]
If $\ul{\psi}'':\ul{\Gamma}_{\phi\times\phi}\langle\pi^{-r}\ul{S}\rangle\rightarrow\fH'$ is the morphism which is obtained by projection to $\ul{\delta}^{-1}(\fH')$ and composition with $\ul{\delta}$, then the diagram
\[
\begin{diagram}
\ul{\Gamma}_{\phi\times\phi}\langle\pi^{-r}\ul{S}\rangle&\rTo^{\ul{\psi}''}&\fH'\\
\dTo<{p_{\fY\times\fY\langle\pi^{-r}\ul{S}\rangle}|_{\ul{\Gamma}_{\phi\times\phi}\langle\pi^{-r}\ul{S}\rangle}}&&\dTo>{\ul{b}_1|_{\fH'}}\\
\fY\times\fY\langle\pi^{-r}\ul{S}\rangle&\rTo^{\ul{\psi}'}&\fH
\end{diagram}
\]
commutes because it commutes after passing to generic fibers. By Proposition \ref{graphenvexprop}, the $R$-envelopes $(\ul{\Gamma}_{\phi\times\phi})_\pi$ and $(\ul{\Gamma}_{\phi\times\phi}\langle\pi^{-r}\ul{S}\rangle)_\pi$ exist, and
\[
\ul{\Gamma}_{\phi\times\phi}\langle\pi^{-r}\ul{S}\rangle_\pi\,\subseteq\,(\fY\times\fY)\langle\pi^{-r}\ul{S}\rangle_\pi\times\ul{\delta}^{-1}(\fH')\;.
\]
The resulting diagram
\[
\begin{diagram}
\ul{\Gamma}_{\phi\times\phi}\langle\pi^{-r}\ul{S}\rangle_\pi&\rTo^{\ul{\psi}''_\pi}&\fH'\\
\dTo<{(p_{\fY\times\fY\langle\pi^{-r}\ul{S}\rangle}|_{\ul{\Gamma}_{\phi\times\phi}\langle\pi^{-r}\ul{S}\rangle})_\pi}&&\dTo>{\ul{b}_1|_{\fH'}}\\
((\fY\times\fY)\langle\pi^{-r}\ul{S}\rangle)_\pi&\rTo^{\ul{\psi}'_\pi}&\fH
\end{diagram}
\]
commutes because of the commutativity of the previous diagram: indeed, since $\fH$ is affine, morphisms to $\fH$ correspond to continuous homomorphisms of rings of global sections, and the rings of global sections on $\ul{\Gamma}_{\phi\times\phi}\langle\pi^{-r}\ul{S}\rangle$ and $\ul{\Gamma}_{\phi\times\phi}\langle\pi^{-r}\ul{S}\rangle_\pi$ coincide by Theorem \ref{mainflatbasechangethm} and by flatness of completion. By Proposition \ref{formaltubestructureprop} and Corollary \ref{genisocor}, the morphism 
\[
(p_{\fY\times\fY\langle\pi^{-r}\ul{S}\rangle}|_{\ul{\Gamma}_{\phi\times\phi}\langle\pi^{-r}\ul{S}\rangle})_\pi
\]
is an isomorphism over a sufficiently small nonempty affine open formal subscheme $\fU''\subseteq\fU'$. By the commutativity of the above diagram, the pullback of the restriction of $(\ul{\psi}'_\pi)^*h$ under this isomorphism coincides with the restriction of the $\ul{\psi}''_\pi$-pullback of $(\ul{b}_1|_{\fH'})^*h$. 

\textit{Step 17: A morphism from the $R$-envelope of a product to a product of $R$-envelopes.} In order to relate the constructions in Step 16 with the morphism $\ul{\psi}_\pi|_\fU$, we exhibit a natural morphism of $R$-envelopes
\[
\gamma=(\gamma_1,\gamma_2)\,:\,(\ul{\Gamma}_{\phi\times\phi})_\pi\rightarrow(\ul{\Gamma}_\phi)_\pi\times (\ul{\Gamma}_\phi)_\pi\;.
\]
as follows: the schematic closure $\ul{\Gamma}_{\phi\times\phi}$ of $\phi\times\phi$ in $\fY\times\fY\times(\fG''\times\fG'')'$ is the strict transform of the schematic closure $\ul{\Gamma}_\phi\times\ul{\Gamma}_\phi$ of $\phi\times\phi$ in $\fY\times\fY\times\fG''\times\fG''$ under $\ul{b}_1$, and the induced admissible blowup $\ul{b}_1':\ul{\Gamma}_{\phi\times\phi}\rightarrow\ul{\Gamma}_\phi\times\ul{\Gamma}_\phi$ extends to a an admissible blowup of $R$-envelopes
\[
(\ul{b}_1')_\pi\,:\,(\ul{\Gamma}_{\phi\times\phi})_\pi\rightarrow(\ul{\Gamma}_\phi\times\ul{\Gamma}_\phi)_\pi\;.
\]
For $i=1,2$, we now exhibit a natural morphism
\[
\ul{q}_{i,\pi}\,:\,(\ul{\Gamma}_\phi\times\ul{\Gamma}_\phi)_\pi\rightarrow(\ul{\Gamma}_\phi)_\pi\,,
\]
and we define 
\[
\gamma_i\,:=\,\ul{q}_{i,\pi}\circ(\ul{b}_1')_\pi\;.
\]
Let $p_i:Y\times Y\rightarrow Y$ denote the projection to the $i$-th factor, which coincides with the uniformly rigid generic fiber of the corresponding projection $\ul{p}_i:\fY\times\fY\rightarrow\fY$.  The schematic closure of the graph of $\phi\circ p_i$ in $\fY\times\fY\times\fG''$ is given by $(\ul{p}_i\times\id_{\fG''})^{-1}(\ul{\Gamma}_\phi)\cong\fY\times\ul{\Gamma}_\phi$. By \cite{egaiii} 5.4.1, the natural morphism of proper formal $\fY\times\fY$-schemes
\[
\ul{q}_i'\,:\,\ul{\Gamma}_\phi\times\ul{\Gamma}_\phi\rightarrow (\ul{p}_i\times\id_{\fG''})^{-1}(\ul{\Gamma}_\phi)
\]
sending $(y_1,g_1,y_2,g_2)$ to $(y_1,y_2,g_i)$ extends uniquely to a morphism of $R$-envelopes
\[
(\ul{q}_i')_\pi\,:\,(\ul{\Gamma}_\phi\times\ul{\Gamma}_\phi)_\pi\rightarrow ((\ul{p}_i\times\id_{\fG''})^{-1}(\ul{\Gamma}_\phi))_\pi\;.
\] 
We claim that the two closed formal subschemes
\[
((\ul{p}_i\times\id_{\fG''})^{-1}(\ul{\Gamma}_\phi))_\pi\quad\textup{and}\quad(\ul{p}_{i,\pi}\times\id_{\fG''})^{-1}((\ul{\Gamma}_{\phi})_\pi)
\]
of $(\fY\times\fY)_\pi\times\fG''$ coincide; once we have shown this, we define $\ul{q}_{i,\pi}$ to be $(\ul{q}_i')_\pi$ followed by the projection to $(\ul{\Gamma}_{\phi})_\pi$. And indeed, formal completion with respect to a subscheme of definition of $\fY\times\fY$ corresponds to base change with respect to the completion morphism $\fY\times\fY\rightarrow(\fY\times\fY)_\pi$, and
\begin{eqnarray*}
&&(\ul{p}_{i,\pi}\times\id_{\fG''})^{-1}((\ul{\Gamma}_{\phi})_\pi)\times_{(\fY\times\fY)_\pi}(\fY\times\fY)\\
&=&(\ul{\Gamma}_\phi)_\pi\times_{\fY_\pi}(\fY\times\fY)\\
&=&((\ul{\Gamma}_\phi)_\pi\times_{\fY_\pi}\fY)\times_{\fY}(\fY\times\fY)\;,
\end{eqnarray*}
where the base change $\cdot\times_\fY(\fY\times\fY)$ functor is understood with respect to the morphism $\ul{p}_i$. Hence, the closed formal subscheme $(\ul{p}_{i,\pi}\times\id_{\fG''})^{-1}((\ul{\Gamma}_{\phi})_\pi)$ of $(\fY\times\fY)_\pi\times\fG''$ is indeed the envelope of 
\[
(\ul{p}_i\times\id_{\fG''})^{-1}(\ul{\Gamma}_\phi)\;,
\]
as claimed.

\textit{Step 18: Proof of the compatibility statement and end of proof of the intermediate claim.}
We have chosen a function $h$ on $\fH$. By Step 15, we must show that the pullback of $h$ under $\ul{\psi}'_\pi|_{\fU''}$ coincides with the pullback of $h$ under $\ul{\psi}_\pi|_\fU\circ\ul{\tau}_\pi|_{\fU''}$ for some sufficiently small nonempty open subset $\fU''$ of $\fU'$. We choose a nonempty open subset $\fU''$ of $\fU'$ as in Step 16 such that the projection
\[
(p_{\fY\times\fY\langle\pi^{-r}\ul{S}\rangle}|_{\ul{\Gamma}_{\phi\times\phi}\langle\pi^{-r}\ul{S}\rangle})_\pi\quad(*)
\]
is an isomorphism above $\fU''$; then it suffices to check the above coincidence after pullback with respect to the restriction of this projection, which we use to identify $\fU''$ with an open formal subscheme of $\ul{\Gamma}_{\phi\times\phi}\langle\pi^{-r}\ul{S}\rangle_\pi$.   Let $h'$ denote the $\ul{b}_1|_{\fH'}$-pullback of $h$; then by Step 16, the pullback of $(\ul{\psi}'_\pi|_{\fU''})^*h$ under ($*$) is given by the pullback of $\ul{\delta}^*(h')$ under the projection
\[
\ul{\Gamma}_{\phi\times\phi}\langle\pi^{-r}\ul{S}\rangle_\pi\rightarrow\ul{\delta}^{-1}(\fH')\;.
\]
The closed formal subscheme 
\[
\ul{\Gamma}_{\phi\times\phi}\langle\pi^{-r}\ul{S}\rangle_\pi\subseteq(\fY\times\fY)\langle\pi^{-r}\ul{S}\rangle_\pi\times(\fG''\times\fG'')'
\]
is the schematic preimage of $(\ul{\Gamma}_{\phi\times\phi})_\pi$ under $\ul{\tau}_\pi:(\fY\times\fY)\langle\pi^{-r}\rangle_\pi\rightarrow(\fY\times\fY)_\pi$, because the completion of that schematic preimage with respect to an ideal of definition of $(\fY\times\fY)\langle\pi^{-r}\ul{S}\rangle$ is
\begin{eqnarray*}
&&\left((\ul{\Gamma}_{\phi\times\phi})_\pi\times_{(\fY\times\fY)_\pi}(\fY\times\fY)\langle\pi^{-r}\ul{S}\rangle_\pi\right)\times_{(\fY\times\fY)\langle\pi^{-r}\ul{S}\rangle_\pi}(\fY\times\fY)\langle\pi^{-r}\ul{S}\rangle\\
&=&\left((\ul{\Gamma}_{\phi\times\phi})_\pi\times_{(\fY\times\fY)_\pi}(\fY\times\fY)\right)\times_{(\fY\times\fY)}(\fY\times\fY)\langle\pi^{-r}\ul{S}\rangle\\
&=&\ul{\Gamma}_{\phi\times\phi}\times_{(\fY\times\fY)}(\fY\times\fY)\langle\pi^{-r}\ul{S}\rangle\\
&=&\ul{\Gamma}_{\phi\times\phi}\langle\pi^{-r}\ul{S}\rangle\;.
\end{eqnarray*}
Hence, we have a cartesian diagram
\[
\begin{diagram}
(\ul{\Gamma}_{\phi\times\phi})\langle\pi^{-r}\ul{S}\rangle_\pi&\rTo&(\fY\times\fY)\langle\pi^{-r}\ul{S}\rangle_\pi&\\
\dTo<{\ul{\tau}'_\pi}&&\dTo>{\ul{\tau}_\pi}&\quad\quad(1)\\
(\ul{\Gamma}_{\phi\times\phi})_\pi&\rTo&(\fY\times\fY)_\pi&
\end{diagram}
\]
and a commutative diagram
\[
\begin{diagram}
(\ul{\Gamma}_{\phi\times\phi})\langle\pi^{-r}\ul{S}\rangle_\pi&\rTo&\ul{\delta}^{-1}(\fH')&\\
\dTo<{\ul{\tau}'_\pi}&&\dInto&\quad\quad(2)\\
(\ul{\Gamma}_{\phi\times\phi})_\pi&\rTo&(\fG''\times\fG'')'&,
\end{diagram}
\]
where the horizontal morphisms are induced by the projections and where $\ul{\tau}'_\pi$ is the restriction of $\ul{\tau}_\pi\times\id_{(\fG''\times\fG'')'}$. The restriction of $\ul{\delta}^*(h')$ to $\ul{\delta}_\fG^{-1}(\fH)''$ coincides with the pullback of $\ul{\delta}_\fG^*h$ under the admissible blowup $\ul{\delta}_\fG^{-1}(\fH)''\rightarrow\ul{\delta}_\fG^{-1}(\fH)$ which is induced by $\ul{b}_1$ and $\ul{b}_1$. Let us consider the diagram
\[
\begin{diagram}
(\ul{\Gamma}_{\phi\times\phi})_\pi&\rTo&(\fY\times\fY)_\pi&\\
\dTo<\gamma&&\dTo>{(\ul{p}_{1,\pi},\ul{p}_{2,\pi})}&\quad\quad\quad\quad(3)\\
(\ul{\Gamma}_\phi)_\pi\times(\ul{\Gamma}_\phi)_\pi&\rTo&\fY_\pi\times\fY_\pi&
\end{diagram}
\]
where the horizontal morphisms are again given by the respective projections; it commutes by construction of the morphism $\gamma$ in Step  17. Finally, let us consider the diagram
\[
\begin{diagram}
(\ul{\Gamma}_{\phi\times\phi})_\pi&\rTo&(\fG''\times\fG'')'&\\
\dTo<\gamma&&\dTo>{\ul{b}_1}&\quad\quad(4)\\
(\ul{\Gamma}_\phi)_\pi\times(\ul{\Gamma}_\phi)_\pi&\rTo&\fG''\times\fG''&
\end{diagram}
\]
where the horizontal morphisms are given by the projections such that the lower horizontal morphism is $\ul{\phi}'_\pi\times\ul{\phi}'_\pi$; it commutes, again, by construction of $\gamma$. In $\fY_\pi\times\fY_\pi$ we have the open formal subscheme $\fW\times\fW$. Over $\fW\times\fW$, the lower horizontal map in (3) is an isomorphism, and the restriction $\ul{\phi}'_\pi|_\fW\times\ul{\phi}'_\pi|_\fW$ of the lower horizontal map in (4) to the preimage of $\fW\times\fW$, which we identify with $\fW\times\fW$ via the aforementioned isomorphism, factors through $\ul{\delta}_\fG^{-1}(\fH)'$. The preimage of $\fW\times\fW$ under the right map in (3) is $\fU$, and the preimage of $\fU$ under the right map in (1) is $\fU'$. It follows that the preimage of $\ul{\delta}_\fG^{-1}(\fH)''$ under the lower left part of (2) contains the preimage of $\fU'$ under the upper map in (1). Let $d$ denote the pullback of $h$ under 
\[
\fW\times\fW\overset{\ul{\phi}'_\pi|_\fW\times\ul{\phi}'_\pi|_\fW}{\longrightarrow}\ul{\delta}_\fG^{-1}(\fH)'\overset{\ul{b}_1\times \ul{b}_1}{\rightarrow}\ul{\delta}_\fG^{-1}(\fH)\overset{\ul{\delta}_\fG}{\rightarrow}\fH\;;
\]
then the function $(\ul{\psi}_\pi|_\fU\circ\ul{\tau}_\pi|_{\fU''})^*h$ is the pullback of $d$ under
\[
\fU''\subseteq\fU'\overset{\ul{\tau}_\pi}\rightarrow\fU\overset{(\ul{p}_{1,\pi},\ul{p}_{2,\pi})}{\longrightarrow}\fW\times\fW
\]
which, via the diagrams (3) and (1), is identified with the pullback of $d$ under
\[
\fU''\subseteq(\ul{\tau}'_{\pi})^{-1}(\gamma^{-1}(\fW\times\fW))\overset{\ul{\tau}'_\pi}{\rightarrow}\gamma^{-1}(\fW\times\fW)\overset{\gamma}{\rightarrow}\fW\times\fW
\]
Combining diagrams (4) and (2), we see that this pullback coincides with the pullback of the function $\ul{b}_1^*(\ul{b}_1\times \ul{b}_1)^*\ul{\delta}_\fG^*h$ on $\ul{\delta}_\fG^{-1}(\fH)''$ via the upper map in diagram (2). Since $\ul{b}_1^*(\ul{b}_1\times \ul{b}_1)^*\ul{\delta}_\fG^*h$ coincides with the restriction of $\ul{\delta}^*h$ to $\ul{\delta}_\fG^{-1}(\fH)''$ and since the preimage of $\ul{\delta}_\fG^{-1}(\fH)''$ under the upper (or, equivalently, the lower) part of diagram (2) contains $\fU''$, the desired equality
\[
\ul{\psi}_\pi'|_{\fU''}^*h\,=\,(\ul{\psi}_\pi|_\fU\circ\ul{\tau}|_\pi|_{\fU''})^*h
\]
follows, which concludes the proof of the intermediate claim which we had formulated in Step 8. 

\textit{Step 19: Reduction to the diagonal.} To prove the statement of the theorem, it suffices by Theorem \ref{ffteverywheredefinedthm} to show that $\phi$ factors, set-theoretically, through the image of $\iota$. The morphism $\phi\times\phi$ has a model
\[
(\ul{\phi\times\phi})':\ul{\Gamma}_{\phi\times\phi}\rightarrow\ul{\delta}^{-1}(\fH')\subseteq(\fG''\times\fG'')'\;,
\]
and the morphism $\psi=\delta\circ(\phi\times\phi)$ has the induced model 
\[
\ul{\psi}''=\ul{\delta}\circ(\ul{\phi\times\phi})':\ul{\Gamma}_{\phi\times\phi}\rightarrow\fH'\subseteq\fG'\;.
\]
The morphism $(\ul{\phi\times\phi})'$ extends uniquely to a morphism
\[
(\ul{\phi\times\phi})'_\pi:(\ul{\Gamma}_{\phi\times\phi})_\pi\rightarrow\ul{\delta}^{-1}(\fH')\subseteq(\fG''\times\fG'')'
\]
of $R$-envelopes, and we obtain the induced extension
\[
\ul{\psi}''_\pi=\ul{\delta}\circ(\ul{\phi\times\phi})'_\pi\,:\,(\ul{\Gamma}_{\phi\times\phi})_\pi\rightarrow\fH'\subseteq\fG'
\]
of $\ul{\psi}''$. The morphism
\[
\mu^\ad\circ(\ul{\psi}''_{\pi,K},\ul{\phi}'_{\pi,K}\circ\gamma_{2,K}):(\ul{\Gamma}_{\phi\times\phi})_{\pi,K}\rightarrow G^\ad
\]
coincides with the morphism $\phi'_{\pi,K}\circ\gamma_{1,K}$, where $\mu$ denotes multiplication on $G$. Indeed, if we consider, for $i=1,2$, the adic generic fiber of the $i$-th component of the commutative diagram (4) of Step 18, then for any point $y$ with values in some adic space, $(\ul{\phi\times\phi})'_{\pi,K}(y)_i=\ul{\phi}'_{\pi,K}(\gamma_{i,K}(y))$, and hence for any such functorial point $y$,
\begin{eqnarray*}
&&(\mu^\ad\circ(\ul{\psi}''_{\pi,K},\ul{\phi}'_{\pi,K}\circ\gamma_{2,K}))(y)\\
&=&(\mu^\ad\circ(\delta^\ad\circ(\ul{\phi\times\phi})'_{\pi,K},\ul{\phi}'_{\pi,K}\circ\gamma_{2,K}))(y)\\
&=&\mu^\ad(\delta^\ad((\ul{\phi}'_{\pi,K}\circ\gamma_{1,K})(y),(\ul{\phi}'_{\pi,K}\circ\gamma_{2,K})(y)),(\ul{\phi}'_{\pi,K}\circ\gamma_{2,K})(y))\\
&=&(\ul{\phi}'_{\pi,K}\circ\gamma_{1,K})(y)\;.
\end{eqnarray*}
The multiplication $\mu^\ad$ on $G^\ad$ preserves the adic subgroup space $\fG'_K\subseteq G^\ad$. Moreover, both $\ul{\psi}''_{\pi,K}$  and $\ul{\phi}'_{\pi,K}|_{\fW_K}$ factor through $\fG'_K$. Since $\ul{\phi}'_{\pi,K}$ extends $\phi^\ad$, it suffices to show that for every classical physical point 
\[
y\in\fY_K\cong(\ul{\Gamma}_\phi)_K\subseteq(\ul{\Gamma}_{\phi})_{\pi,K}\;,
\]
there exists a physical point $u\in \gamma_{2,K}^{-1}(\fW_K)$ with the property that $\gamma_{1,K}(u)=y$. Let us recall that for $i=1,2$, we have $\gamma_i=\ul{q}_{i,\pi}\circ(\ul{b}'_1)_\pi$, where
\[
(\ul{b}'_1)_\pi\,:\,(\ul{\Gamma}_{\phi\times\phi})_\pi\rightarrow(\ul{\Gamma}_\phi\times\ul{\Gamma}_\phi)_\pi
\]
is an admissible blowup and where the $\ul{q}_{i,\pi}$ were defined in Step 17. The generic fiber $(\ul{b}'_1)_{\pi,K}$ of $(\ul{b}'_1)_\pi$ is an isomorphism, cf.\ \cite{huberhabil} 3.9.6 and 3.9.18, so we may consider $(\ul{\Gamma}_\phi\times\ul{\Gamma}_\phi)_\pi$ and the $\ul{q}_{i,\pi}$ instead of $(\ul{\Gamma}_{\phi\times\phi})_\pi$ and the $\gamma_i$. Let $y\in\fY_K$ be a classical point; then the valuation ring $R_y$ of its residue field is a finite extension of $R$, and there is a natural morphism $\chi_y:\Spf R_y\rightarrow\ul{\Gamma}_\phi\rightarrow(\ul{\Gamma}_\phi)_\pi$ such that $y$ is the $(\chi_{y})_K$-image of the unique point in $(\Spf R_y)_K$, cf.\ \cite{urigspaces} Lemma 2.3 and its proof. It suffices to find a discrete valuation ring $R_u$ together with a morphism $\chi_u:\Spf R_u\rightarrow\ul{q}_{2,\pi}^{-1}(\fW)$ and a flat morphism $\lambda:\Spf R_u\rightarrow \Spf R_y$ such that $\chi_y\circ\lambda=\ul{q}_{1,\pi}\circ\chi_u$. To do so, it suffices to show that the schematic preimage of $\chi_y$ under $\ul{q}_{1,\pi}$ is identified with $\Spf R_y\times\ul{\Gamma}_{\phi,\pi}$ such that the restriction of $\ul{q}_{2,\pi}$ is given by the projection to the second factor. Indeed, if we intersect with $\ul{q}_{2,\pi}^{-1}(\fW)$ we obtain $\Spf R_y\times\fW$, and it suffices now to find a morphism $\Spf R_u\rightarrow\Spf R_y\times\fW$ such that the composition with the projection to $\Spf R_y$ is flat. The open formal subscheme $\fW\subseteq\fY_\pi$ is faithfully flat and regular over $R$; hence $\Spf R_y\times\fW$ is faithfully flat and regular over $R_y$. Let us look at the special fiber over $k_y$. It is regular, and splits as a disjoint union of integral components. By Remark \ref{normdvrrem}, the complete stalk of $\Spf R_y\times\fW$ in the generic point of any of these components is a complete discrete valuation ring with the desired properties.

It remains to see that the preimage of $\chi_u$ under $\ul{q}_{1,\pi}$ is indeed identified with $\Spf R_y\times\ul{\Gamma}_{\phi,\pi}$. Let us recall from Step 17 that we have identified 
\[
((\ul{p}_1\times\id_{\fG''})^{-1}(\ul{\Gamma}_\phi))_\pi
\]
with the schematic preimage of $(\ul{\Gamma}_\phi)_\pi$ under the morphism 
\[
\ul{p}_{1,\pi}\times\id_{\fG''}:(\fY\times\fY)_\pi\times\fG''\rightarrow\fY_\pi\times\fG''\;;
\]
the schematic preimage of $\chi_y:\Spf R_y\rightarrow(\ul{\Gamma}_\phi)_\pi$ under the natural morphism from $((\ul{p}_1\times\id_{\fG''})^{-1}(\ul{\Gamma}_\phi))_\pi$ to $(\ul{\Gamma}_\phi)_\pi$ is thus identified with the schematic preimage of 
\[
\Spf R_y\overset{\chi_y}{\rightarrow}\ul{\Gamma}_{\phi,\pi}\subseteq\fY_\pi\times\fG''
\]
under $\ul{p}_{1,\pi}\times\id_{\fG''}$. This preimage is $(\Spf R_y)\times\fY_\pi$, because 
\begin{eqnarray*}
R_y\hat{\otimes}_{A^\pi}(A\hat{\otimes}_RA)^\pi &\cong&(R_y\otimes_A(A\hat{\otimes}_RA))^\pi\\
&\cong&(R_y\hat{\otimes}_A(A\hat{\otimes}_RA))^\pi\\
&\cong&(R_y\hat{\otimes}_RA)^\pi\\
&\cong& (R_y\otimes_R A)^\pi
\end{eqnarray*}
since $R_y$ is finite over $R$ and, hence, over $A^\pi$ and $A$. Let us consider the resulting diagram
\[
\begin{diagram}
\Spf R_y\times\ul{\Gamma}_{\phi,\pi}&\rTo&\Spf R_y\times\fY_\pi\\
\dTo&&\dTo\\
(\ul{\Gamma}_\phi\times\ul{\Gamma}_\phi)_\pi&\rTo&((\ul{p}_1\times\id_{\fG''})^{-1}(\ul{\Gamma}_\phi))_\pi
\end{diagram}
\]
where the upper morphism is $\id_{\Spf R_y}$ times the projection $\ul{\Gamma}_{\phi,\pi}\rightarrow\fY_\pi$ and where the left vertical morphism is induced from the morphism
\[
\chi_u\times\id_{\ul{\Gamma}_\phi}:\Spf R_y\times\ul{\Gamma}_\phi\rightarrow\ul{\Gamma}_\phi\times\ul{\Gamma}_\phi
\]
of proper formal $\fY\times\fY$-schemes by passing to $R$-envelopes. It is cartesian, as desired, because it is cartesian after formal completion with respect to an ideal of definition of $\fY\times\fY$ and by uniqueness of $R$-envelopes. We thus obtain the desired description of the schematic preimage of $\chi_u$ under $\ul{q}_{1,\pi}$ and hereby finish the proof of Theorem \ref{mainweilextthm}. \hfill\qed

\section{Formal Néron models}\label{formalnmsec}

In the introduction, we defined Néron models for smooth uniformly rigid $K$-spaces, cf.\ Definition \ref{nmurigdefi}. We can now clarify how these Néron models relate to Néron models for smooth rigid $K$-spaces in the sense of Definition \ref{nmrigdefi}, at least for groups satisfying mild finiteness conditions.

Let us first note that for a smooth quasi-paracompact and quasi-separated rigid $K$-space, equipped with its Raynaud-type uniform structure, the universal property of Definition \ref{nmrigdefi} is weaker than the universal property of Definition \ref{nmurigdefi}. To do so, we choose an isomorphism $\r\circ\urig\cong\rig$ as well as an isomorphism between $\r\circ\ur$ and the identity functor on the category of quasi-paracompact and quasi-separated rigid $K$-spaces. In the following, we will use these isomorphisms implicitly to identify the respective functors.

\begin{prop}\label{neronmodinverseprop}
Let $X$ be a smooth quasi-paracompact and quasi-separated rigid $K$-space. If $(\fX,\phi)$ is the Néron model for $X^\ur$ and if $\fX$ is of locally tf type, then $(\fX,\phi^\r)$ is the Néron model for $X$.
\end{prop}
\begin{proof}
It suffices to show that for any affine smooth formal $R$-scheme of tf type $\fZ$, the natural map $\Hom(\fZ^\urig,X^\ur)\rightarrow\Hom(\fZ^{\urig,\r},X^{\ur,\r})$ induced by $\r$ is bijective. Since $\fZ$ is of tf type over $R$, there exists an isomorphism $\fZ^\urig\cong\fZ^{\rig,\ur}$. In view of the identifications $\r\circ\urig=\rig$ and $\r\circ\ur=\id$, we have to show that the map
\[
\alpha_\r\,:\,\Hom(\fZ^{\rig,\ur},X^\ur)\rightarrow\Hom(\fZ^\rig,X)
\]
is bijective. The functor $\ur$ yields a map $\alpha_\ur$ in the opposite direction, which is bijective by \cite{urigspaces} Cor.\ 2.61. Since $\r\circ\ur=\id$, the map $\alpha_\r$ is left inverse to $\alpha_\ur$ and, hence, coincides with the inverse of $\alpha_\ur$; we thus conclude that $\alpha_r$ is bijective.
\end{proof}

\subsection{General existence results}

Combining Lemma \ref{affineextlem} with Corollary \ref{smoothimpliesnormalcor}, we immediately obtain the following statement:

\begin{prop}\label{affinenmprop}
A smooth affine formal $R$-scheme of ff type is the formal Néron model of its uniformly rigid generic fiber.
\end{prop}

Moreover, Theorem \ref{ffteverywheredefinedthm} yields the following:

\begin{prop}\label{genfibneronprop}
Let $\fG$ be a smooth formal $R$-group scheme of locally ff type; then $\fG$ is the formal Néron model of its uniformly rigid generic fiber.
\end{prop}

In the following we use without further notice the fact that the functor $\ur$ takes retrocompact open immersions to open immersions, cf.\ the second last paragraph of Section 2 in \cite{urigspaces}.

Theorem \ref{mainweilextthm} yields the following uniformly rigid analog of the criterion \ref{bsthm} ($i$) of Bosch and Schlöter which we reproduced in the introduction:

\begin{thm}\label{maincompthm}
Let $G$ be a quasi-paracompact smooth rigid $K$-group, and let $\fG$ be a quasi-pa\-ra\-com\-pact smooth formal $R$-group scheme of locally tf type together with a retrocompact open immersion $\fG^\rig\hookrightarrow G$ respecting the group structures. If $\fG^\rig$ contains all formally unramified points of $G$, then $(\fG,\phi^\ur)$ is a formal Néron model of $G^\ur$, and $\phi^\ur$ is an open immersion.
\end{thm}

Combining Theorem \ref{maincompthm} with \cite{bosch_schloeter} Theorem 1.2, we obtain the following general existence result:

\begin{cor}\label{neronex1cor}
Let $G$ be a smooth and quasi-paracompact rigid $K$-group whose group $G(K^\sh)$ of unramified points is bounded, and let $(\fG,\phi)$ be its Néron model; then $(\fG,\phi^\ur)$ is the Néron model of $G^\ur$, and $\phi^\ur$ is an open immersion.
\end{cor}
\begin{proof}
By \cite{bosch_schloeter} Theorem 1.2, the Néron model $(\fG,\phi)$ of $\fG$ exists, and $\phi$ is an open immersion respecting the group structures. Since $G$ is quasi-separated and since $\fG^\rig$ is quasi-compact, $\phi$ is automatically retrocompact. By \cite{wegel} Theorem 4, the formation of $\fG$ commutes with base change with respect to formally unramified local extensions of $R$. Hence, $\fG^\rig$ contains all formally unramified points of $G$, and by Theorem \ref{maincompthm} we conclude that $(\fG,\phi^\ur)$ is the Néron model of $G^\ur$. 
\end{proof}

Let us recall from \cite{blr} 10.1.1 that the Néron lft-model of a smooth $K$-scheme $\mathcal{X}$ is a smooth $R$-model $\ul{\mathcal{X}}$ of $\mathcal{X}$ satisfying the universal property that for every smooth $R$-scheme $\ul{\mathcal{Z}}$ with generic fiber $\mathcal{Z}$, every morphism $\mathcal{Z}\rightarrow \mathcal{X}$ extends uniquely to a morphism $\ul{\mathcal{Z}}\rightarrow\ul{\mathcal{X}}$.

\begin{cor}\label{neronex2cor}
Let $\mathcal{G}$ be a quasi-compact smooth $K$-group scheme that has a Néron lft-model $\ul{\mathcal{G}}$, let $\fG$ be the completion of $\ul{\mathcal{G}}$ along its special fiber, let $G$ denote the rigid analytification of $\mathcal{G}$, and let $\phi:\fG^\rig\rightarrow G$ be the morphism that is induced by the given identification $\mathcal{G}\cong\ul{\mathcal{G}}\times_RK$. If $\ul{\mathcal{G}}$ is quasi-compact or if $\mathcal{G}$ is commutative,  then $(\fG,\phi^\ur)$ is the formal Néron model of $G^\ur$, and $\phi^\ur$ is an open immersion.
\end{cor}
\begin{proof}
Analytifications of $K$-schemes of locally finite type are quasi-pa\-ra\-com\-pact by construction (cf.\ \cite{bgr} 9.3.4 Example 2), so $G^\ur$ is well-defined. By \cite{bosch_schloeter} Theorem 6.2, the morphism $\phi$ is a retrocompact open immersion, and $\fG$ is the Néron model of $G$. By \cite{blr} 10.1/3, the formation of $\ul{\mathcal{G}}$ commutes with formally unramified base change; it follows that the same holds for $\fG$ and, hence, that the image of $\phi$ contains all formally unramified points of $G$. The statement is now a consequence of Theorem \ref{maincompthm}.
\end{proof}

\subsection{Construction techniques}\label{constrsec}

We can now construct interesting new examples of uniformly rigid $K$-spaces admitting Néron models by implementing ideas of Chai in the uniformly rigid setting (cf.\ \cite{bisecartpre} 5.4). We thereby obtain formal Néron models $(\fG,\phi)$ where $\fG$ is not of locally tf type over $R$ and where $\phi$ is not surjective.

A morphism of uniformly rigid $K$-spaces is called étale if and only if its underlying morphism of rigid $K$-spaces is étale. It is easily seen (cf.\ \cite{liu_ag} 4/3.26 for the nontrivial implication) that a morphism $\phi$ of rigid or uniformly rigid $K$-spaces is an étale monomorphism if and only if
\begin{enumerate}
\item the morphism $\phi$ is injective on physical points, and
\item it induces isomorphisms of completed stalks.
\end{enumerate}
In particular, a morphism $\phi$ of uniformly rigid $K$-spaces is an étale monomorphism if and only if the same holds for $\phi^\r$, cf.\ \cite{urigspaces} Section 2.3.1 and the discussion following Prop.\ 2.55. 
Hence, \cite{conrad_ampleness} Thm. 4.2.7 shows that that if $K'/K$ is a finite Galois extension, then a morphism $\phi$ of rigid or uniformly rigid $K$-spaces is an étale monomorphism if and only if the same holds for $\phi\times_KK'$.

Let us first explain a construction technique involving formal completion and Galois descent. Let $K'/K$ be a finite Galois extension with Galois group $\Gamma$, let $R'/R$ denote the associated extension of discrete valuation rings, and let $k'/k$ be the induced residue field extension. Let $X$ be a  quasi-paracompact and quasi-separated rigid $K$-space, and let $X':=X\times_KK'$ be the rigid $K'$-space which is induced via base change. Let us assume that the Néron models $(\fX,\phi)$ and $(\fX',\phi')$ of $X^\ur$ and $(X')^\ur$ exist, that $\fX$ and $\fX'$  are separated and of locally tf type over $R$ and that $\phi$ and $\phi'$ are étale monomorphisms. Let
\[
\psi\colon\fX\times_RR'\rightarrow\fX'
\]
denote the base change morphism which is induced by the universal property of $\fX'$; its uniformly rigid generic fiber is then an étale monomorphism. (In general, $\psi^\urig$ will not be surjective: the 'volume' of the Néron model may grow under base change due to the appearance of points which take values in rings that are unramified over $R'$ and which are not induced from points taking values in rings unramified over $R$.) The equivariant $\Gamma$-action on  $X\times_KK'$ induces an equivariant $\Gamma$-action on $\fX'$; let $V_{k'}$ be a $\Gamma$-stable affine closed subscheme of $\fX'_{k'}$ that is contained in an affine open formal subscheme of $\fX'$. Since $\fX'$ is separated, $V_{k'}$ then has a $\Gamma$-stable affine open neighborhood $\fU'$ in $\fX'$. Let us write
\begin{eqnarray*}
Y'&:=&(\fX'|_{V_{k'}})^\srig\quad\textup{and}\\
U'&:=&(\fU')^\srig\;;
\end{eqnarray*}
then $Y'\subseteq U'$ is a $\Gamma$-equivariant open immersion of semi-affinoid $K'$-spaces, where the ring of global functions in $U'$ is actually an affinoid $K$-algebra. By Lemma \ref{semaffdesclem} and Proposition \ref{descffprop}, the open immersion $Y'\hookrightarrow U'$  descends uniquely to a morphism of semi-affinoid $K$-spaces $Y\rightarrow U$, and by the second paragraph of the present section, this morphism is an étale monomorphism. By finite Galois descent for rigid spaces (cf.\ \cite{conrad_ampleness} Thm.\ 4.2.3 and and Thm.\ 4.2.7), the open immersion $(U')^\r\hookrightarrow X'$ descends to a retrocompact open immersion $U^\r\hookrightarrow X$, and we obtain an open immersion $U\cong U^{\r,\ur}\hookrightarrow X^\ur$. We thus see that $Y'\rightarrow (X')^\ur$ descends to a morphism $Y\rightarrow X^\ur$ and that this morphism is an étale monomorphism.

Let $W_{k'}$ denote the $\psi_{k'}$-preimage of $V_{k'}$. There exists a unique reduced closed subscheme $W_k\subseteq\fX_k$ such that the underlying reduced subschemes of $W_{k'}$ and $W_k\times_kk'$ coincide. Indeed, since $k'/k$ is normal and since base change with respect to finite purely inseparable field extensions induces homeomorphisms, this follows from finite Galois descent for schemes in the special fibers. Let
\[
\fY\,\mathrel{\mathop:}=\,\fX|_{W_k}
\]
denote the completion of $\fX$ along $W_k$; then $\fY\otimes_RR'$ is the completion of $\fX\otimes_RR'$ along $W_{k'}$. By Proposition \ref{descffprop}, the morphism $\fY^\srig\otimes_KK'\rightarrow Y'$ descends to a morphism 
\[
\hat{\phi}\,:\,\fY^\srig\rightarrow Y\;,
\]
and by the second paragraph of the present section, this morphism is an étale monomorphism.

\begin{prop}\label{constrnmexprop}
The pair $(\fY,\hat{\phi})$ is the formal Néron model of $Y$.
\end{prop}
\begin{proof}
Let $\fZ$ be a smooth formal $R$-scheme of locally ff type, let $Z$ denote its uniformly rigid generic fiber, and let $\alpha: Z\rightarrow Y$ be a morphism; we must show that $\alpha$ extends uniquely to a morphism $\ul{\alpha}:\fZ\rightarrow\fY$. Since $\hat{\phi}$ is a monomorphism, uniqueness of $\ul{\alpha}$ follows from the fact that $\fZ$ is $R$-flat. Since $(\fX,\phi)$ is the formal Néron model of $X^\ur$, the composition $\tilde{\alpha}$ of $\alpha$ with the monomorphism $Y\rightarrow X^\ur$ extends uniquely to a morphism 
\[
\ul{\tilde{\alpha}}\colon\fZ\rightarrow\fX\quad;
\]
in particular, the image of $\tilde{\alpha}\otimes_KK'$ lies in the image of the étale monomorphism $\psi^\urig\colon(\fX\times_RR')^\srig\rightarrow (\fX')^\urig$. Since the morphism $\tilde{\alpha}\otimes_KK'$ factors through $Y'$, all points in its image specialize to $V_{k'}$ in $\fZ'$. It follows that all points in the image of $\tilde{\alpha}$ specialize to $W_k\subseteq\fZ$. Since $\sp_\fZ$ is surjective onto the closed points of $\fZ$, we conclude that $\ul{\tilde{\alpha}}$ factors uniquely as a morphism $\ul{\alpha}\colon\fZ\rightarrow\fY$ composed with the completion morphism $\fY\rightarrow\fX$. The generic fiber of $\ul{\alpha}$ must coincide with $\alpha$ since $Y\rightarrow X^\ur$ is a monomorphism.
\end{proof}

To conclude our discussion of construction techniques, let us observe that Néron models can be constructed by splitting off direct factors of groups:

\begin{lem}\label{proddecomplem}
Let $G_1$ and $G_2$ be uniformly rigid $K$-groups such that $G_1\times G_2$ has a Néron model $(\fG,\phi)$. Then $G_1$ and $G_2$ have Néron models $(\fG_1,\phi_1)$ and $(\fG_2,\phi_2)$ respectively, and there is a canonical isomorphism $\fG\cong\fG_1\times\fG_2$ such that $\phi=\phi_1\times\phi_2$.
\end{lem}
\begin{proof}
The product group structure on $G$ induces group endomorphisms $h_1$ and $h_2$ of $G$ such that $G_i$ is identified with the kernel of $h_i$ via the natural morphism $G_i\hookrightarrow G_1\times G_2$, such that $h_i^2=h_i$, $h_1\circ h_2=h_2\circ h_1=e_G$ and $h_1\cdot h_2=h_2\cdot h_1=\id_G$, where $e_G$ denotes the endomorphism of $G$ that is induced by the unit section. By the universal property of $(\fG,\phi)$, we obtain group endomorphisms $\ul{h}_i$ of $\fG$ satisfying the same identities. Let us set $\fG_i\mathrel{\mathop:}=\ker\ul{h}_i$; then $\phi$ induces morphisms $\phi_i:\fG_i^\srig\rightarrow G_i$. Clearly the natural morphism $\fG_1\times\fG_2\rightarrow\fG$ is an isomorphism such that $\phi=\phi_1\times\phi_2$, and $(\fG_i,\phi_i)$ is the formal Néron model of $G_i$ for $i=1,2$.
\end{proof}




\subsection{Uniformly rigid tori}\label{urigtorisec}
In this final section of the present paper, we assume that the residue field $k$ of $K$ has positive characteristic $p>0$. Following ideas of Chai (cf.\ \cite{bisecartpre} Section 4), we give an example for a uniformly rigid $K$-group $T$ which is not induced from a rigid $K$-group via uniform rigidification and which has a Néron model $(\fT,\phi)$ where $\phi$ is not surjective. In fact, we define a category of semi-affinoid $K$-groups, the category of uniformly rigid $K$-tori, containing many objects with these properties. Uniformly rigid $K$-tori provide a link between algebraic $K$-tori and abelian $K$-varieties with potentially ordinary reduction.

Let us recall that the formal multiplicative group $\hat{\G}_{m,R}$ over $R$ is defined to be the completion of the multiplicative group $\G_{m,R}$ along the unit section of its special fiber.

\begin{defi}
A semi-affinoid $K$-group $T$ is called a uniformly rigid $K$-torus if there exist a finite field extension $K'/K$ and an isomorphism of uniformly rigid $K'$-groups
\[
T\times_KK'\,\cong\,(\hat{\G}_{m,R'}^d)^\urig
\]
for some $d\in\N$, where $R'$ denotes the valuation ring of $K'$ and where $\hat{\G}_{m,R'}$ is the formal multiplicative group over $R'$. We say that $K'/K$ is a splitting field for $T$. If we can take $K'=K$, we say that the uniformly rigid $K$-torus $T$ is split. A morphism of uniformly rigid $K$-tori is a homomorphism of uniformly rigid $K$-groups; let $\uTor_K$ denote the category of uniformly rigid $K$-tori.
\end{defi}

By definition, the base change functor $\cdot\times_KK'$ restricts to a base change functor $\uTor_K\rightarrow\uTor_{K'}$ for every finite field extension $K'/K$.

\begin{remark}
If $T$ is a split uniformly rigid $K$-torus, then the Néron model $(\fT,\phi)$ of $T$ exists, and $\phi$ is an isomorphism. Indeed, by Proposition \ref{affinenmprop}, the smooth affine formal $R$-scheme $\hat{\G}_{m,R}^d$ is the Néron model of its uniformly rigid generic fiber for every $d\in\N$.
\end{remark}

We will prove that every uniformly rigid $K$-torus admits a Néron model. To do so, it is useful to describe the category of uniformly rigid $K$-tori in terms of character groups. Let us fix a separable algebraic closure $K^\sep$ of $K$, let $\Gamma$ denote its group of $K$-automorphisms, and let $\Mod_{\Z_p}^{\Gamma,\discr}$ denote the category of finite free $\Z_p$ modules equipped with a $\Gamma$-action that factors through a finite quotient of $\Gamma$. In the following, all separable algebraic extensions of $K$ will be be understood to be subfields of $K^\sep$.

\begin{lem}\label{chargrouplem}
Let $T$ be a uniformly rigid $K$-torus, let $K'/K$ be a Galois splitting field for $T$ with valuation ring $R'$, and let us write $T_{K'}:=T\times_KK'$. Then
\[
X^*(T)\,:=\,\Hom_{\uTor_{K'}}(T_{K'},\hat{\G}_{m,K'}^\urig)
\]
is an object in $\Mod_{\Z_p}^{\Gamma,\discr}$ which does not depend on the choice of $K'$.
\end{lem}
\begin{proof}
The endomorphism ring $\End(\hat{\G}_{m,R})$ is naturally isomorphic to $\Z_p$. Indeed, this follows from \cite{lubin_onepar} Sections 2.2 and 2.3: endomorphisms of $\hat{\G}_{m,R}$ correspond to endomorphisms of the one-dimensional multiplicative formal group law in the sense of \cite{lubin_onepar} Def.\ 1.2.1 and 1.3.1, and the proof of \cite{lubin_onepar} Lemma 2.3.1 carries over verbatim to the case of equal positive characteristic. Hence, $X^*(T)$ is naturally equipped with a structure of $\Z_p$-module. Since $T_{K'}\cong(\hat{\G}_{m,K'}^\urig)^d$ for some $d\in\N$ and since
\[
\Hom_{\uTor_{K'}}(\hat{\G}_{m,K'}^\urig, \hat{\G}_{m,K'}^\urig)\,=\,\Hom(\hat{\G}_{m,R'},\hat{\G}_{m,R'})\,=\,\Z_p
\]
by Proposition \ref{affinenmprop}, we see that the $\Z_p$-module $X^*(T)$ is finite free of rank $d$ and that $X^*(T)$ with its $\Gamma$-action does not depend on the choice of $K'$.
\end{proof}

\begin{defi}
Let $T$ be a uniformly rigid $K$-torus. The object $X^*(T)\in\Mod_{\Z_p}^{\Gamma,\discr}$ defined in Lemma \ref{chargrouplem} is called the character group of $T$.
\end{defi}

\begin{lem}
The contravariant functor
\[
X^*(\cdot):\uTor_K\rightarrow\Mod_{\Z_p}^{\Gamma,\discr}
\]
is an anti-equivalence of categories.
\end{lem}
\begin{proof}
We first show that $X^*(\cdot)$ is fully faithful: let $T_1$ and $T_2$ be uniformly rigid $K$-tori, and let $K'/K$ be a finite Galois extension splitting both $T_1$ and $T_2$. By Corollary \ref{maindescentcor}, the natural map
\[
\Hom_{\uTor_K}(T_1,T_2)\rightarrow\Hom_{\uTor_{K'}}(T_{1,K'},T_{2,K'})^\Gamma
\]
is bijective, and the natural map
\begin{eqnarray*}
\Hom_{\uTor_{K'}}(T_{1,K'},T_{2,K'})^\Gamma&\rightarrow& \Hom_{\Z_p-\textup{mod}}(X^*(T_2),X^*(T_1))^\Gamma\\
&=&\,\Hom_{\Mod_{\Z_p}^{\Gamma,\discr}}(X^*(T_2),X^*(T_1))
\end{eqnarray*}
is bijective by duality for finite free $\Z_p$-modules. Full faithfulness has thus been shown. To see that $X^*(\cdot)$ is essentially surjective, let $M$ be a finite free $\Z_p$-module of rank $d$ equipped with an action of the Galois group of a finite Galois extension $K'/K$, and let $R'$ denote the valuation ring of $K'$. Up to a choice of basis for $M$, we may identify $M$ with the character group of $(\hat{\G}_{m,R'}^d)^\urig$; by Lemma \ref{semaffdesclem}, this split uniformly rigid $K'$-torus with its induced $\Gal(K'/K)$-action descends to a uniformly rigid $K$-torus $T$ such that $X^*(T)\cong M$ in $\Mod_{\Z_p}^{\Gamma,\discr}$.
\end{proof}

There is a natural additive functor $T\mapsto T_{\Z_p}$ from the category $\Tor_K$ of algebraic $K$-Tori to $\uTor_K$ which, on the level of character groups, is given by $\cdot\otimes_\Z\Z_p$. Explicitly, if $T$ is an algebraic $K$-torus, if $K'/K$ is a Galois splitting field for $T$ with valuation ring $R'$ and if $\ul{T}'$ is the Néron lft-model of $T_{K'}$ over $R'$ (which exists by \cite{blr} Ex.\ 10.1/5), then
\[
T_{\Z_p}\,=\,(\ul{T}'|_e)^\urig/\Gamma\;,
\]
where $\ul{T}'|_e$ denotes the formal completion of $\ul{T}'$ along the identity section of its special fiber, and where $(\ul{T}'|_e)^\urig/\Gamma$ denotes the semi-affinoid $K$-group that is obtained from $(\ul{T}'|_e)^\urig$ via Galois descent by means of Corollary \ref{maindescentcor}. 

Let us recall from \cite{blr} Prop.\ 10.1/6 that every algebraic $K$-torus $\ul{T}$ admits a Néron lft-model $\ul{T}$. By Corollary \ref{neronex2cor}, the completion of $\ul{T}$ along its special fiber, together with the natural morphism from the uniformly rigid generic fiber of that completion to $T^{\an,\ur}$, is the Néron model of $T^{\an,\ur}$. Combining this fundamental result with Proposition \ref{constrnmexprop}, we obtain:

\begin{prop}\label{urigtornmexprop}
Let $T$ be an algebraic $K$-torus, let $K'/K$ be a Galois splitting field for $T$ with valuation ring $R'$, let
\[
\psi\,:\,\ul{T}\times_RR'\rightarrow\ul{T}'
\]
be the associated base change map of Néron lft-models, and let $W_k\subseteq\ul{T}_k$ be a closed subscheme such that $W_k\times_kk'$ coincides topologically with the $\psi$-preimage of the unit section of $\ul{T}'_{k'}$. Then $\ul{T}|_{W_k}$ together with the natural morphism $\hat{\phi}:(\ul{T}|_{W_k})^\urig\rightarrow T_{\Z_p}$ is the Néron model of $T_{\Z_p}$.
\end{prop}

To deduce that in fact every uniformly rigid $K$-torus has a Néron model, we will need to know that every uniformly rigid $K$-torus is a direct factor of the uniformly rigid $K$-torus associated to an algebraic $K$-torus. This is a consequence of the following result which, on the level of character groups, has been observed by Chai:

\begin{lem}\label{karoubilem}
The natural functor $\cdot\otimes_\Z\Z_p:\Tor_K\rightarrow\uTor_K$ identifies $\uRig_K$ with the pseudo-abelian envelope of $\Tor_K$.
\end{lem}
\begin{proof}
After passing to character groups, it suffices to show the corresponding statement for the functor 
\[
\cdot\otimes_\Z\Z_p \,:\,\Mod_\Z^\Gamma\rightarrow\Mod_{\Z_p}^{\Gamma,\discr}\;,
\]
where $\Mod_\Z^\Gamma$ denotes the category of finite free $\Z$-modules equipped with a continuous $\Gamma$-action. Let $\Gamma'$ be a finite quotient of $\Gamma$, and let $V$ be a finite-dimensional $\Q_p$-vector space equipped with a $\Gamma'$-action. It suffices to verify the following statements:
\begin{enumerate}
\item There exists a finite-dimensional $\Q$-vector space $W$ with a $\Gamma'$-action such that $V$ is, as a $\Q_p[\Gamma']$-module, a direct summand of $W\otimes_\Q\Q_p$.
\item Let $V'$ be a $\Gamma$-invariant linear complement of $V$ in $W\otimes_\Q\Q_p$, and let 
$\ul{V}\subseteq V$,
$\ul{V}'\subseteq V'$ be $\Gamma$-stable $\Z_p$-lattices. Then there exists 
a $\Gamma$-stable $\Z$-lattice $\ul{W}\subseteq W$ such that the $\Gamma$-equivariant decomposition $W\otimes_\Q\Q_p\cong V\oplus V'$ restricts to a decomposition
\[
\ul{W}\otimes_\Z\Z_p\cong\ul{V}\oplus\ul{V}'\\;.
\]
\end{enumerate}
To prove the first statement, we may assume that $V$ is irreducible; then $V$ is a direct summand of the regular representation $\Q_p[\Gamma']$, so it suffices to take $W=\Q[\Gamma']$. Let us prove the second statement. Clearly $\ul{V}\oplus\ul{V}'$ is a $\Gamma$-stable $\Z_p$-lattice in $W\otimes_\Q\Q_p$; we must descend it to a $\Gamma$-stable $\Z$-lattice in $W$. To do so, we choose finite systems of generators $(v_i)_{i\in I}$ and $(v'_i)_{i\in I'}$ of the $\Z_p[\Gamma]$-modules $\ul{V}$ and $\ul{V}'$ respectively. Since $W$ is $p$-adically dense in $W\otimes_\Q\Q_p$, there exist systems $(w_i)_{i\in I}$ and $(w_i')_{i\in I'}$ in $W$ such that 
\begin{eqnarray*}
v_i-w_i\,\in\,p(\ul{V}\oplus\ul{V}')&\forall&i\in I\quad\textup{and}\\
v'_i-w'_i\,\in\,p(\ul{V}\oplus\ul{V}')&\forall&i\in I'\quad.
\end{eqnarray*}
In particular, the $w_i$ and the $w_i'$ are contained in $\ul{V}\oplus\ul{V}'$. Let $\ul{W}\subseteq W$ be the $\Z[\Gamma]$-submodule generated by the $w_i$ and the $w_i'$. Since the inclusion $\ul{W}\otimes_\Z\Z_p\subseteq\ul{V}\oplus\ul{V}'$ reduces to an isomorphism of $\F_p$-vector spaces modulo $p$, Nakayama's Lemma implies that $\ul{W}\otimes_\Z\Z_p=\ul{V}\oplus\ul{V}'$ as $\Z_p$-modules; hence $\ul{W}$ is a $\Z$-lattice of $W$ with the desired property. 
\end{proof}

Combining Lemma \ref{karoubilem}, Proposition \ref{urigtornmexprop} and Lemma \ref{proddecomplem}, we obtain:

\begin{cor}
Every uniformly rigid $K$-torus has a Néron model.
\end{cor}

\begin{remark}
If $T$ is an algebraic $K$-torus with Galois splitting field $K'/K$ and Néron lft-models $\ul{T}$, $\ul{T}'$ and if $(\fT,\phi)$ is the Néron model of $T_{\Z_p}$, then on the level of rigid spaces, 
\[
\im\phi^\r\,=\,(\ul{T}^\rig\times_KK')\cap(\fY'|_e)^\rig\;;
\]
that is, $T_{\Z_p}$ is the unit ball around the origin in $T$ with respect to a coordinate system on $\ul{T}'$, and the Néron model of $T_{\Z_p}$ contains those points in that unit ball which extend to the Néron model of $T$. In particular, $\phi$ is surjective onto $T_{\Z_p}$ if and only if the base change conductor $c(T,K)$ of $T$ is zero.
\end{remark}

\begin{remark} Let us assume that the residue field $k$ of $R$ is perfect, and let $A$ be an abelian $K$-variety with potentially ordinary reduction, let $K'/K$ be a finite Galois extension such that $A\times_KK'$ has semi-stable reduction, let $\ul{A}'$ denote the Néron model of $A\times_KK'$, and let $e$ denote the identity section of its special fiber.
\begin{enumerate}
\item If $k$ is algebraically closed, then the completion of $\ul{A}'$ along $e$ is isomorphic to $(\hat{\G}_{m,R'})^d$, where $d$ denotes the dimension of $A$, and hence
\[
(\ul{A}'|_e)^\urig/\Gamma
\]
is a uniformly rigid $K$-torus; we thus obtain a functor from the category of abelian $K$-varieties with potentially ordinary reduction to $\uRig_K$.
\item If $k$ is not necessarily algebraically closed, then $\ul{A}'|_e$ is a possibly non-split formal $R'$-torus; it splits over the maximal unramified extension of $K'$ in $K^\sep$. One can define a uniformly rigid quasi-$K$-torus to be a semi-affinoid $K$-group which, after finite base field extension, become isomorphic to the uniformly rigid generic fiber of a formal torus. By studying character groups, one sees that the category $\uqTor_K$ of uniformly rigid quasi-$K$-tori is anti-equivalent to the category of finite free $\Z_p$-modules equipped with an action of $\Gamma$ that is continuous for the $p$-adic topology on $\Z_p$ and the profinite topology on $\Gamma$ and which is trivial on a finite index subgroup of the inertia subgroup $I\subseteq\Gamma$. The uniformly rigid $K$-group
\[
(\ul{A}'|_e)^\urig/\Gamma
\]
is an object of $\uqTor_K$.
\end{enumerate}
\end{remark}
\bibliographystyle{plain}
\bibliography{formal_neron_models}

\end{document}